\documentclass[reqno,11pt]{article}
\usepackage{amsmath}
\usepackage{soul}
\usepackage{amsfonts}
\usepackage{amssymb}
\usepackage{enumerate}
\usepackage{amstext}
\usepackage{amsbsy}
\usepackage{amsopn}
\usepackage{bbm,amsthm}
\usepackage{amscd}
\usepackage{tikz}
\usepackage{amsxtra}
\usepackage{upref}
\usepackage{graphicx}
\usepackage{epstopdf}
\usepackage{fancyhdr}

\newcommand\ignorethis[1]{}

\DeclareFontFamily{OML}{rsfs}{\skewchar\font'177}
\DeclareFontShape{OML}{rsfs}{m}{n}{ <5> <6> rsfs5 <7> <8> <9> rsfs7
  <10> <10.95> <12> <14.4> <17.28> <20.74> <24.88> rsfs10 }{}
\DeclareMathAlphabet{\mathfs}{OML}{rsfs}{m}{n}

\newtheorem{thm}{Theorem}[section]
\newtheorem{theorem}[thm]{Theorem}

\newtheorem{lemma}[thm]{Lemma}
\newtheorem{prop}[thm]{Proposition}
\newtheorem{proposition}[thm]{Proposition}

\newtheorem{definition}[thm]{Definition}
\newtheorem{cor}[thm]{Corollary}
\newtheorem{corollary}[thm]{Corollary}
\newtheorem*{theorem*}{Theorem}
\newtheorem*{maintheorem}{Main Theorem}
\newtheorem*{maintheoremrevisited}{Main Theorem Revisited}

\newtheorem*{example*}{Example}
\newtheorem{ourtheorem}{Theorem}
\newtheorem{conjecture}{Conjecture}
\newtheorem{question}{Question}

\theoremstyle{remark}

\newtheorem{remark}[thm]{Remark}
\newtheorem{remarks}[thm]{Remarks}

\numberwithin{equation}{section}

\newenvironment{myitemize}
{ \begin{itemize}
    \setlength{\itemsep}{2pt}
    \setlength{\parskip}{0pt}
    \setlength{\parsep}{0pt}     }
{ \end{itemize}    }

\renewcommand\top{{\operatorname{top}}}
\newcommand\lambdamax{\lambda_{\max}}

\renewcommand{\epsilon}{\varepsilon}
\def\wt{\widetilde}

\def\text#1{\textrm{#1}}

\def\emptyset{\varnothing}

\def\eps{\epsilon}

\def\vf{\varphi}

\def\Si{\Sigma}
\def \R{\mathbb R}
\def\RR{\mathbb R}
\def \N{{\mathbb N}}
\def \NN{{\mathbb N}}

\def \Z{\mathbb Z}

\def\ov{\overline}
\def\un{\underline}

\def\cU{\mathcal U}

\def\wh{\widehat}
\def\({\biggl(}
\def\){\biggr)}
\def\<{\mathbf\langle}
\def\>{\mathbf\rangle}

\def\lambdamin{\lambda_{\min}}

\def\geo{{\textrm{geo}}}

\def\cO{\mathcal O}
\def\ti{\pitchfork}
\def\wh{\widehat}

\DeclareMathOperator\diam{diam}
\DeclareMathOperator\Diff{Diff}
\DeclareMathOperator\dist{dist}

\DeclareMathOperator\Lip{Lip}
\DeclareMathOperator\Span{span}
\DeclareMathOperator\NUH{NUH}

\DeclareMathOperator\supp{supp}
\DeclareMathOperator\HPM{\mathbb{P}_{\mathrm h}}
\DeclareMathOperator\HPO{\operatorname{Per}_h}
\DeclareMathOperator\HC{HC}

\DeclareMathOperator\Prob{\mathbb P}

\DeclareMathOperator\Proberg{\Prob_{\mathrm{e}}}

\newcommand\MME{{\mathfs M}}

\newcommand\HrobR{h^*_{\Diff^r}}
\newcommand\HrobInf{h^*_{\Diff^\infty}}

\newcommand{\dbar}{d\hspace*{-0.10em}\bar{}\hspace*{0.14em}}

\newcommand\folL{\mathfs L}

\def\hsim{\overset{h}{\sim}}

\newcommand\choice[1]{\left(\begin{matrix} #1 \end{matrix}\right)}

\newcommand\Card{\operatorname{Card}}
\newcommand\diffsym{\operatorname{\triangle}}

\newcommand\Per{\operatorname{Per}}

\title{Measures of maximal entropy for surface diffeomorphisms}
\author{J\'er\^{o}me Buzzi, Sylvain Crovisier\thanks{S.C. was partially supported by the ERC project 692925 \emph{NUHGD}.}, and Omri Sarig\thanks{O.S. was partially supported by ISF grants 199/14 and 1149/18.}}
\date{\today}

\makeatletter
\renewcommand{\l@section}{\@dottedtocline{2}{3.8em}{3.2em}}
\renewcommand{\l@subsection}{\@dottedtocline{3}{3.8em}{3.2em}}
\newcommand{\subsectionruninhead}{\@startsection{subsection}{2}{0mm}{-\baselineskip}{-0mm}{\bf\large}}
\newcommand{\subsubsectionruninhead}{\@startsection{subsubsection}{3}{0mm}{-\baselineskip}{-0mm}{\bf\normalsize}}
\makeatother

\begin{document}

\maketitle

\begin{abstract}
We show that $C^\infty$ surface diffeomorphisms with positive topological entropy have at most finitely many ergodic measures of maximal entropy in general, and at most one in the topologically transitive case.  This answers  a question of Newhouse, who proved that such measures always exist.
To do this we generalize  Smale's spectral decomposition theorem to non-uniformly hyperbolic surface diffeomorphisms, we introduce homoclinic classes of measures, and we study their properties
using codings  by irreducible countable state Markov shifts.
\end{abstract}

\tableofcontents

\section{Introduction}

\subsection{Measures of maximal entropy}
A famous theorem of Newhouse says that $C^\infty$ diffeomorphisms on compact manifolds without boundary  have ergodic measures of maximal entropy \cite{Newhouse-Entropy}. He asked if the number of these measures is finite when  the manifold is  two-dimensional and the diffeomorphism has positive topological entropy \cite[Problem 2]{Newhouse1990}.

In this paper,  we answer this question positively. A by-product of our analysis is a Spectral Decomposition Theorem similar to Smale's result for Axiom A diffeomorphisms \cite{Smale-Diff-Dynam-Syst},  but for general surface diffeomorphisms with positive topological entropy.
Some of our results apply to $C^r$ diffeomorphisms with $r>1$ and to more general equilibrium measures.

Let $M$ be a {\em closed surface}:  a compact two-dimen\-sio\-nal $C^\infty$ Riemannian manifold without boundary.  Denote its  distance function by $d(\cdot,\cdot)$.
Let $f\colon M\to M$ be a diffeomorphism.
A compact invariant set $K$ is called {\em transitive}, if some point $x\in K$
has a dense forward and backward orbits $\{f^n(x):n\geq 0\}$ and ${\{f^{-n}(x):n\geq 0\}}$ in $K$.
We say that $f$ is {\em topologically transitive} if $M$ is transitive. $f$ is \emph{topologically mixing}
if for any non-empty open sets $U,V$, we have $f^n(U)\cap V\neq \emptyset$ for all large $n$.

Recall that the topological entropy $h_{\top}(f)$ is related to the metric entropies $h(f,\nu)$ of $f$-invariant Borel probability measures $\nu$ by the
variational principle, see \cite{Goodman} and  \cite[Chapter 8]{Walters-Book}:
$$
 h_\top(f)= \sup_{\nu\in\mathbb P(f)} h(f,\nu)=\sup_{\nu\in\Proberg(f)} h(f,\nu),
 $$
where $\mathbb P(f)$ is the set of $f$-invariant Borel probability measures on $M$ and
$\Proberg(f)$ the set of ergodic $\mu\in \mathbb P(f)$. If $K$ is an invariant compact set, $h_\top(f,K)$ denotes the topological entropy of $f|_K$.
\begin{definition}
A {\em measure of maximal entropy (m.m.e.)} is  a measure $\mu$ in $\mathbb P(f)$ such that $h(f,\mu)=h_\top(f)$.
\end{definition}
\noindent
It is well-known that almost every ergodic component of a m.m.e. is an ergodic m.m.e.
Ergodic m.m.e.'s are important to  classification problems~\cite{Adler-Weiss-PNAS,Boyle-Buzzi}
and to the asymptotic analysis of periodic orbits~\cite{Bowen-Closed-Geodesics,Burguet-equidistribution}.

We have already mentioned Newhouse's Theorem on the existence of a m.m.e. for $C^\infty$ diffeomorphisms.
In this paper,  we show:

\begin{maintheorem}
Let $f$ be a $C^\infty$ diffeomorphism on a closed surface, and suppose $h_{\top}(f)>0$. \label{Page-Main-Theorem}
Then:\hspace{-2cm}\mbox{}
\begin{myitemize}
\item[--] The number of ergodic measures of maximal entropy of $f$ is finite.
\item[--] When $f$ is topologically transitive, it has a unique measure of maximal entropy.
\item[--] When $f$ is topologically mixing, its unique m.m.e. is isomorphic to a Bernoulli scheme.
\end{myitemize}
\end{maintheorem}
The theorem would be false without the assumptions on the entropy or the dimension (think of the identity map and of the identity$\times$Anosov).

\medskip

This result
extends to diffeomorphisms $f$ defined on a possibly non-compact surface $M$, possibly with boundary but that have a \emph{global compact attractor}, i.e., an invariant
compact subset $\Lambda$ such that (i) $f|_\Lambda$ is transitive; (ii)
$d(f^n(x),\Lambda)\to 0$ as $n\to +\infty$ for any $x\in M$; (iii) $\Lambda\subset U$ with $U\subset M$ a boundaryless surface.
This is for instance the case for a positive Lebesgue measure set of non-hyperbolic parameters of the H\'enon maps~\cite{BenedicksCarleson1}.
The following consequence has been previously obtained for good parameters of H\'enon maps by Berger~\cite{Berger}:

\begin{corollary}\label{c.attractor}
Let $f$ be a $C^\infty$ diffeomorphism of a {two-dimensional} manifold
having a global compact attractor and positive entropy. Then $f$ has a unique measure of maximal entropy.
\end{corollary}

\subsection{Homoclinic classes and spectral decomposition}

A general question in dynamics is to find a decomposition into invariant elementary pieces. A famous example is Smale's ``spectral decomposition" for Axiom A diffeomorphisms \cite{Smale-Diff-Dynam-Syst}.
We discuss here generalizations of Smale's  spectral decomposition to general $C^\infty$ surface diffeomorphisms with positive entropy.

We first recall some definitions from \cite{Newhouse-Homoclinic}. Let  $f$ be a $C^r$ diffeomor\-phism on a closed manifold (of any dimension).
A hyperbolic periodic orbit of saddle type is a set $\mathcal O=\{f^i(x): i=0,\ldots,p-1\}$ such that $p\geq 1$, $f^p(x)=x$, and  $x$ has a positive Lyapunov exponent, a negative Lyapunov exponent and no zero Lyapunov exponents. Let
$$
\HPO(f):=\{\mathcal O: \mathcal O\text{ is a hyperbolic periodic orbit of saddle type} \}.
$$
For  $y\in \mathcal O$, let $W^u(y):=\{z: d(f^{-n}(z),f^{-n}(y))\xrightarrow[n\to\infty]{}0\}$ and $W^s(y):=\{z: d(f^{n}(z),f^{n}(y))\xrightarrow[n\to\infty]{}0\}$. Set $W^u(\mathcal O)=\bigcup_{y\in\mathcal O}W^u(y)$ and  $W^s(\mathcal O)=\bigcup_{y\in\mathcal O}W^s(y)$. These are $C^r$sub-manifolds.
Given two $\mathcal O,\mathcal O'\in\HPO(f)$, let $
W^u(\mathcal O)\pitchfork W^s(\mathcal O')
$
denote the collection of transverse intersection points of $W^u(\mathcal O)$ and $W^s(\mathcal O')$.
Two $\mathcal O_1,\mathcal O_2\in\HPO(f)$ are called  {\em homoclinically related} if
$W^u(\mathcal O_i)\pitchfork W^s(\mathcal O_j)\neq \emptyset \text{ for }i\neq j$.
We then write $\mathcal O_1\hsim \mathcal O_2$.

\begin{definition}
The {\em homoclinic class} of $\mathcal O$ is the set
$$
\HC(\mathcal O):=\ov{\{\mathcal O'\in\HPO(f):\mathcal O'\hsim\mathcal O\}}.$$
As in~\cite{Abdenur-Crovisier}, the integer $\gcd(\{\operatorname{Card}(\cO'):\cO'\in\HPO(f),\cO'\hsim\cO\})$
is called \emph{period} of the homoclinic class.
\end{definition}
Each homoclinic class is a transitive invariant compact set (\cite{Newhouse-Homoclinic}, see also section \ref{Sec.homoclinic}).

In the particular case when the non-wandering set $\Omega(f)$ is uniformly hyperbolic and contains a dense set of periodic points ($f$ is ``Axiom A"), Smale's spectral decomposition
theorem \cite{Smale-Diff-Dynam-Syst} asserts  that there are only finitely many different homoclinic classes, that these classes are pairwise disjoint, and that $\Omega(f)$
is the union of the sinks, the sources and the homoclinic classes.
Further properties were obtained in \cite{Bowen-Periodic-Points} and \cite{Bowen-MP-Axiom-A}.

Without the Axiom A assumption, it is not necessarily true that the homoclinic classes are finite in number, or pairwise disjoint, or that their union covers $\Omega(f)$.
But for $C^\infty$ diffeomorphisms on closed surfaces, ``everything works modulo a set negligible with respect to ergodic measures of positive entropy:"

\begin{ourtheorem}[Spectral decomposition]\label{mainthm-Spectral}
Let $f$ be a $C^\infty$ diffeomorphism on a closed surface
and consider a maximal family $\{\cO_i\}$ of non-homoclinically related hyperbolic periodic orbits.  Then: \begin{description}
  \item[\quad \it (1)  \sc Covering:] $\mu(\cup_i \HC(\cO_i))=1$ for any $\mu\in \mathbb P_e(f)$ with $h(f,\mu)>0$.
  \item[\quad \it (2) \sc Disjointness:] $h_\top(f,\HC(\cO_i)\cap \HC(\cO_j))=0$ for any $i\neq j$.
  \item[\quad \it (3)  \sc Period:] If $\ell_i$ is the period of the homoclinic class of $\cO_i$,
there is a compact set $A_i$ such that
\begin{myitemize}
\item[--] $\HC(\cO_i)=A_i\cup f(A_i)\cup\dots\cup f^{\ell_i-1}A_i$ and $f^{\ell_i}A_i=A_i$,
\item[--] $f^{\ell_i}:A_i\to A_i$ is topologically mixing,
\item[--] $f^j(A_i)\cap A_i$ has empty {relative} interior in $\HC(\cO_i)$ and zero topological entropy when $0<j<\ell_i$.
\end{myitemize}
 \item[\quad \it (4) \sc Finiteness:] For any $\chi>0$, the set of $\HC(\cO_i)$ such that $h_\top(f,\HC(\cO_i))\geq \chi$ is finite.
\item[\quad \it (5) \sc Uniqueness:] If $f$ is topologically transitive, at most one $\HC(\cO_i)$ has positive topological entropy. If $f$ is topologically mixing, then this $\HC(\cO_i)$ has period $\ell_i=1$  .
\end{description}
\end{ourtheorem}

We turn to the dynamics inside a homoclinic class. To begin with we recall some definitions and properties.
An ergodic invariant probability measure $\mu$ is {\em hyperbolic of saddle type},  if it has one positive Lyapunov exponent, one negative Lyapunov exponent, and no zero Lyapunov exponent. In dimension two, any $\mu\in\Proberg(f)$ with  positive entropy is hyperbolic of saddle type by Ruelle's inequality \cite{Ruelle-Entropy-Inequality}, and we denote its two exponents by $-\lambda^s(\mu)<0<\lambda^u(\mu)$. Let $\HPM(f)$ be the set of hyperbolic measures $\mu\in \Proberg(f)$ of saddle type.
By Pesin's Stable Manifold Theorem if $\mu\in \HPM(f)$, then $\mu$-almost every $x$ has stable and unstable manifolds  $W^s(x),W^u(x)$, see section \ref{s.classes-measure}.

The notion of homoclinic relation extends to measures, see the precise definition in Section~\ref{s.def-symbolic}.
In particular, for $\cO\in \HPO(f)$ and $\mu\in \HPM(f)$, we write $\cO\hsim \mu$ if
$W^u(x)\pitchfork W^s(\mathcal O)\neq \emptyset$ and $W^s(x)\pitchfork W^u(\mathcal O)\neq \emptyset$
for $\mu$-almost every $x$.
\medskip

Following the approach developed in~\cite{Sarig-JAMS}, we will  code parts of the dynamics
by \emph{countable state Markov shifts} $\sigma\colon \Sigma\to \Sigma$ (see Section~\ref{s.def-symbolic} for the definition).
The main novelty here is that,
by restricting to a homoclinic class we can obtain a shift which is \emph{irreducible}, i.e., which is topologically transitive. Irreducibility is important, because by the work of Gurevich,  irreducible countable state Markov shifts with finite Gurevich entropy can have at most one m.m.e \cite{Gurevich-Measures-Of-Maximal-Entropy}.

\medskip

The dynamics on each homoclinic class can be described both from the measured and symbolic viewpoints
as follow:

\begin{ourtheorem}\label{mainthm-class}
Let $f$ be a $C^\infty$ diffeomorphism on a closed surface and $\HC(\cO)$ a
homoclinic class with positive topological entropy. Then:
\begin{enumerate}[(1)]
\item $\HC(\cO)$ supports a unique $\mu\in \mathbb P(f)$
with entropy equal to $h_\top(f,\HC(\cO))$. The support of $\mu$ coincides with $\HC(\cO)$. The measure-preserving map $(f,\mu)$ is
isomorphic to the product of a Bernoulli scheme and of the cyclic permutation
of order $\ell:=\gcd(\{\operatorname{Card}(\cO'):\cO'\hsim\cO\})$.
If $\HC(\cO)$ is contained in a topologically mixing compact invariant set, then $\ell=1$.
\item Any $\nu\in \mathbb P_e(f|_{\HC(\cO)})$
with $h(f,\nu)>0$
is homoclinically related to $\cO$.
\item For any $\chi>0$, there exist an \emph{irreducible} locally compact {countable state} Markov shift $(\Sigma,\sigma)$ and a H\"older-continuous map
$\pi\colon \Sigma\to \HC(\cO)$ such that $\pi\circ \sigma=f\circ \pi$ and
\begin{myitemize}
\item[--] $\pi\colon \Sigma^\#\to \HC(\cO)$ is finite-to-one,
\item[--] {$\pi(\Sigma^\#)=\HC(\cO)$ mod $\nu$ for each $\nu\in \mathbb P_e(f)$}
such that $h(f,\nu)>\chi$.
\end{myitemize}
\end{enumerate}
\end{ourtheorem}
\noindent
Here and throughout,  $\Sigma^\#$ is the set of sequences in $\Sigma$
where some symbol is repeated infinitely many times in the future, and some (possibly different) symbol is repeated infinitely many times in the past.
The inclusion $\pi(\Sigma)\subset \HC(\cO)$ may be strict, see  Remark~\ref{r.potential} below.

\subsection{Finite regularity and equilibrium states}\label{intro-finite-regularity}
Our methods give information on equilibrium measures other than the measure of maximal entropy, and  under weaker regularity assumptions.

We let $f:M\to M$ be a $C^r$ diffeomorphism with $1<r\leq\infty$, i.e., $f$ is invertible and, together with its inverse, it is
continuously differentiable  $\lfloor r\rfloor $ times, and if $r\not\in\N\cup\{\infty\}$ then its $\lfloor r\rfloor$-th derivative is H\"older continuous with H\"older exponent $r-\lfloor r\rfloor$.

Suppose $\phi:M\to\R\cup\{-\infty\}$ is a Borel function such that $\sup\phi<\infty$.  An $f$-invariant probability measure $\mu$ is called an {\em equilibrium measure} for the potential $\phi$, if
$h(f,\mu)+\int\phi d\mu=P_{\top}(\phi)$, where
$$
P_{\top}(\phi):=\sup_{\nu\in\Proberg(f)} \{h(f,\nu)+\int\phi d\nu\}.
$$ For example, equilibrium measures of $\phi\equiv 0$ are measures of maximal entropy.

{The \emph{admissible potentials} are functions $\phi:M\to\RR\cup\{-\infty\}$
which are sums of functions of the following types:}
\begin{myitemize}
\item[--] H\"older-continuous functions.
\item[--] \emph{Geometric potentials}: $\phi_\geo^u(x):=-\beta\log\|Df|_{E^u(x)}(x)\|$ or $\phi_\geo^s(x):=-\beta\log\|Df^{-1}|_{E^s(x)}(x)\|$
where $\beta$ is a real number and $T_x M=E^u(x)\oplus E^s(x)$ is the Oseledets decomposition at $x$, with the convention that
 $\phi_\geo^t(x):=-\infty$ ($t=u,s$) if $E^t$  is not well-defined at $x$. The functions $\phi_\geo^u,\phi_\geo^s$ are measurable upper-bounded but not always continuous.
\end{myitemize}
Notice that $\phi$ is admissible for $f$ if and only if it is admissible for $f^{-1}$.

\medskip

Given $n\in\Z$, let $\|Df^n\|:=\max_{x\in M}\|Df^n|_{T_x M}\|$, $\lambda^u(f):=\lim_{n\to\infty}\frac{1}{n}\log \|Df^n\|$ and $\lambda^s(f):=\lim_{n\to\infty}\frac{1}{n}\log \|Df^{-n}\|$. Our statements involve the number
$\lambdamin(f):=\min\{\lambda^s(f),\lambda^u(f)\}$.
For each ergodic hyperbolic invariant measure $\mu$, let
\begin{equation}\label{delta-u/delta-s}
\delta^u(\mu):=h(f,\mu)/\lambda^u(\mu) \text{ and } \delta^s(\mu):=h(f,\mu)/\lambda^s(\mu)
\end{equation}
which are called \emph{stable and unstable dimensions} of $\mu$ (see~\cite{Ledrappier-Young-II}).

\medskip
The following result  extends the  Main Theorem to $C^r$ diffeomorphisms and other equilibrium measures:

\begin{maintheoremrevisited}\label{MainTheoremRevisited}
Let $f$ be a $C^r$ diffeomorphism with $r>1$ on a closed surface $M$ and let $\phi:M\to\R\cup\{-\infty\}$ be an admissible potential. Then:
 \begin{enumerate}[(1)]
\item {For any $\chi> \lambdamin(f)/r$}
  there are at most finitely many ergodic equilibrium measures for $\phi$ with entropy strictly bigger than $\chi$;
\item Each compact invariant transitive subset of $M$ carries at most one ergodic hyperbolic equilibrium measure $\mu$ for $\phi$ with
$\delta^u(\mu)>1/r$, and at most one ergodic hyperbolic equilibrium measure $\mu$ for $\phi$ with
$\delta^s(\mu)>1/r$.
 \end{enumerate}
\end{maintheoremrevisited}

In particular, this theorem applies under a small potential condition:
\begin{corollary}\label{Cor-finite-r}
Let $f$ be a $C^r$ diffeomorphism with $r>1$ on a closed surface and let $\phi$ be an admissible potential. {Assume also that
\begin{equation}\label{small-potential}
{P_{\top}(\phi)>\sup\phi+\frac{\lambdamin(f)}r}.
\end{equation}
Then  $\phi$  has at most finitely many ergodic equilibrium measures, and if $f$ is topologically transitive, then it has at most one.}
\end{corollary}
The role of~\eqref{small-potential} is to guarantee that every equilibrium measure has entropy larger than {$\lambdamin(f)/r$.} Notice that this condition
holds whenever
 $\sup\phi-\inf\phi<h_{\top}(f)-\frac{\lambda_{\min}(f)}{r}$.

\begin{corollary}\label{Cor-finite-eq-C-infty}
Let $f$ be a topologically transitive $C^\infty$ diffeomorphism on a closed surface. Then any admissible potential $\phi$  has at most one ergodic equilibrium measure with positive entropy. If $f$ is topologically mixing, then this measure, if it exists, is Bernoulli.
\end{corollary}

By the Pesin formula \cite{Ledrappier-Strelcyn}, SRB measures $\mu$ satisfy $\lambda^u(\mu)=h(f,\mu)>0$; in particular $\delta^u(\mu)=1$ and they are equilibrium measures of the geometric potential $\phi:=-\log\|Df|_{E^u}\|$.
The second version of the Main Theorem thus implies the following result of~\cite{Rodriguez-Hertz-Squared-Tahzibi-Ures-CMP} (see~\cite{Ledrappier-Mesures-de-Sinai} for the Bernoulli property).

\begin{corollary}[Hertz-Hertz-Tahzibi-Ures's theorem revisited]\label{c.HHTU}
Let $f$ be a $C^r$ diffeomorphism ($r>1$) on a closed surface.
Then each transitive invariant compact set $\Lambda$ supports at most one SRB measure.
When such a measure $\mu$ exists, its support coincides with a homoclinic class $\HC(\cO)$
satisfying $\mu\hsim \cO$. Moreover, if $f|\Lambda$ is topologically mixing, the unique SRB measure, if it exists, is Bernoulli.
\end{corollary}

We also have $C^r$ versions of the spectral decomposition theorem and of Theorem~\ref{mainthm-class}.
They are stated later in Section~\ref{s.spectral-decomposition}.

\subsection{Borel classification}

By Theorem~\ref{mainthm-class}, for every homoclinic class $H$, for every  $\chi>0$, there is a finite-to-one continuous coding $\pi_\chi:\Sigma_\chi^\#\to H$ such that $\Sigma_\chi$ is an irreducible countable state Markov shift and $\pi_\chi(\Sigma_\chi^\#)$  carries all ergodic measures on $H$  with entropy bigger than $\chi$.
Using  Hochman's Borel generator theorem \cite{Hochman1} (see \cite{Boyle-Buzzi}),  we can replace the family $\{\pi_\chi:\chi>0\}$  by a  single Borel conjugacy between the diffeomorphism and a Markov shift, after discarding from each system an invariant  Borel set which has measure zero for all ergodic measures with positive entropy (see \S \ref{s-Borel}).

Such conjugacies have been built for the natural extensions of interval maps, assuming a finite critical set \cite{Hofbauer1981} or  $C^\infty$ smoothness \cite{BuzziSIM} (this point of view was formalized in  \cite{Newhouse-On-Hofbauer}).
We show (see \S\ref{s.def-symbolic} for the period of a Markov shift):

\begin{ourtheorem}\label{thmBorelConj}
Let $f$ be a $C^\infty$ diffeomorphism on a closed surface. For any hyperbolic periodic orbit~$\cO$,
 $f:\HC(\cO)\to \HC(\cO)$  is Borel conjugate modulo \emph{zero entropy measures} to an irreducible countable state Markov shift with period equal to the period of the homoclinic class of $\cO$:
 $$
    \ell := \gcd(\{\operatorname{Card}(\cO'):\cO'\in\HPO(f),\; \cO'\hsim\cO\}).
 $$
\end{ourtheorem}

Using the Spectral Decomposition Theorem, this yields an alternate proof of the classification theorem from \cite{Boyle-Buzzi} in the $C^\infty$ smooth setting. More importantly, this shows that the complete invariant of \cite{Boyle-Buzzi} is determined by the topological entropies and the periods of the homoclinic classes, see Corollary~\ref{cor-Cinf-Borel}. When there is a mixing m.m.e., a more powerful version of Hochman's generator theorem \cite{Hochman2} provides Borel conjugacy after discarding only periodic orbits:

\begin{corollary}\label{Corollary-Classification-mixing}
Let $f$ be a $C^\infty$ diffeomorphism on a closed surface with positive topological entropy. If $f$ is topologically mixing,  then it is Borel conjugate modulo \emph{periodic orbits} to a mixing Markov shift with equal entropy. Such diffeomorphisms are classified
up to Borel conjugacy modulo periodic orbits  by their topological entropy.
\end{corollary}

\subsection{Dependence of the simplex of m.m.e.'s on the diffeomorphism}

Let $N_{\max}(f)$ be the number of ergodic m.m.e.'s of a diffeomorphism $f$. It is equal to the dimension of the simplex $\MME(f)$ of (possibly non ergodic) measures maximizing the entropy. The following upper semicontinuity property holds in the $C^\infty$ setting.

\begin{ourtheorem}\label{theorem-N-usc}
Let $M$ be a closed surface.
The function $N_{\max}:\Diff^\infty(M)\to\NN\cup\{\infty\}$ satisfies:
\begin{equation}\label{e.semi-con}
     \text{ if }h_\top(f)>0,\quad \limsup_{g\stackrel{C^\infty}\to f} N_{\max}(g)\leq N_{\max}(f).
\end{equation}
More precisely,
let $k$ be an integer, $f_n\in\Diff^\infty(M)$, and $\Sigma_n$ be some $k$-face of the simplex $\MME(f_n)$.
If $f_n\to f$ in $\Diff^\infty(M)$ and $\Sigma_n\to\Sigma$ in the Hausdorff topology,
then $\Sigma\subset\MME(f)$ and $\Sigma$ is a $k$-dimensional simplex.
\end{ourtheorem}

\begin{remark}
Upper semicontinuity of $N_{\max}$ can fail at diffeomorphisms with zero entropy:
consider for instance a sequence of rational rotations of the circle converging to an irrational one.
\end{remark}

\begin{remark}
In the $C^r$ topology, for finite $r$, our techniques only yield that the number of ergodic m.m.e.'s is locally bounded at $f\in\Diff^r(M)$ with $h_\top(f)>\lambdamin(f)/r$. We do not know if their number is actually upper semicontinuous under this hypothesis.
\end{remark}

\begin{remark}
Lower semicontinuity should not hold in general:  {e.g.,
different homoclinic classes of maximal entropy {may sometimes merge through} an arbitrarily small perturbation,
see~\cite{Diaz-Santoro} for an example in dimension $3$.}
These examples also indicate that the limit of ergodic m.m.e.'s does not have to be ergodic. More generally, the limiting simplex $\Sigma$ in the above statement can be strictly included in a $k$-face of $\MME(f)$.
\end{remark}

\subsection{Measured homoclinic classes in arbitrary dimension}
Using the recent generalization by Ben Ovadia~\cite{Ben-Ovadia} of~\cite{Sarig-JAMS}, our construction of a coding by irreducible Markov shifts (Theorem~\ref{Theorem-Symbolic-Dynamics-C^r} below) has a straightforward generalization to any dimension, with the same proof. We point out that this direct generalization is restricted to \emph{measured homoclinic classes,} i.e., classes of homoclinically related hyperbolic ergodic measures. Thus these classes appear as natural pieces of the dynamics whereas the classical topological homoclinic classes seem more difficult to analyze.

In particular, the following consequence of the coding (see Corollary~\ref{c-local-uniqueness})
holds in any dimension:
\emph{Let $f$ be a $C^r$ diffeomorphism, $r>1$, of a closed manifold, $\phi$ an admissible potential,
and $\cO$ a hyperbolic periodic orbit. Then there is at most one ergodic, hyperbolic, equilibrium measure for $\phi$ homoclinically related to $\cO$.
Moreover, when such an equilibrium measure exists,
its support coincides with $\HC(\cO)$,
and it is isomorphic to the product of a Bernoulli scheme with the cyclic permutation of order
$\gcd\{\operatorname{Card}(\cO') : \cO'\hsim \cO\}$.}

\subsection{Flows}
Combining with \cite{Lima-Sarig,Ledrappier-Lima-Sarig},
some of our results should extend to non-singular flows on three-dimensional manifolds,
and in particular to the geodesic flows of positive topological entropy on surface.
(An easier case for this generalization are flows admitting a global compact transverse section.)

\subsection{Conjectures in finite smoothness}\label{ss.comment}

Our techniques yield rather complete qualitative results assuming $C^\infty$ smoothness and positive entropy.
When we only assume finite smoothness (i.e., $C^r$ smoothness with $1<r<\infty$), then we usually need to restrict to entropy larger than $\lambdamin(f)/r$ (though the higher value of $\lambdamax(f)/r$ is sometimes needed for the exploitation of topological transitivity). Let us review these results and discuss their possible optimality. Throughout this section $f$ is a $C^r$ diffeomorphism of a closed surface.

\paragraph{a. {Infinite number of homoclinic classes}.}

By Theorem~\ref{thmEntropyDecay}, if $\{\mu_i\}_{i\in I}$ is a family of ergodic hyperbolic measures such that $\mu_i\hsim\mu_j\implies i=j$, then
 $$
   \text{for any }\chi>\frac{\lambdamin(f)}r \quad \text{ the set }\{i\in I: h(f,\mu_i)\geq \chi\} \text{ is finite} .
 $$
A standard construction produces infinitely many disjoint horseshoes
with topological entropy bounded away from zero, by a $C^r$-perturbation near a homoclinic orbit
(see for instance~\cite{misiurewicz,newhouse1978,newhouse-downarowicz,BuzziNoMax}).
In a forthcoming work \cite{Buzzi-Crovisier-Sarig-forthcoming} we use this technique to show that this value of $\lambdamin(f)/r$ is the best possible
for the finiteness property in the following sense.

\begin{theorem*}
For any $1<r<\infty$,
there is a $C^r$ surface diffeomorphism $f$ with infinitely many pairwise disjoint homoclinic classes $\HC(\cO_n)$, each supporting
a measure $\mu_n$ with $h(f,\mu_n)=h_\top(f,\HC(\cO_n))$ and such that
 $$
   \liminf_{n\to\infty} h_\top(f,\HC(\cO_n)) \geq\frac{\lambdamin(f)}r.
$$
\end{theorem*}

We believe the threshold $\lambdamin(f)/r$ is also sharp for the finiteness of the number of ergodic m.m.e. proved in the Main Theorem Revisited:
\smallskip

\begin{conjecture}
For any $\varepsilon>0$, any $1<r<\infty$,
there is a $C^r$ surface diffeomorphism $f$ with $h_\top(f)>\lambdamin(f)/r-\eps$ and infinitely many ergodic m.m.e.'s.
\end{conjecture}

\paragraph{b. Uniqueness in the transitive case.}
When $f$ is topologically transitive, the Main Theorem Revisited asserts the uniqueness of the m.m.e. when $h_\top(f)>\lambdamax(f)/r$ (it shows the existence of at most two ergodic m.m.e.'s when $h_\top(f)>\lambdamin(f)/r$).
But in fact we know of no example showing that positive topological entropy is not enough to ensure uniqueness in the transitive case:

\begin{question}
Given $1\leq r<\infty$, does there exist a transitive $C^r$-smooth surface diffeomorphism with positive topological entropy that has multiple m.m.e.'s?
\end{question}

Likewise, Corollary~\ref{c.measure-related} says that two homoclinic classes $\HC(\cO),\HC(\cO')$ with $\cO'\not\hsim\cO$ can intersect only in a set with zero topological entropy as soon as $h_\top(f,\HC(\cO))>\lambdamax(f)/r$. But we do not know if this entropy condition is necessary and much less if it is sharp:

\begin{question}
Given $1\leq r<\infty$, does there exist a $C^r$-smooth surface diffeomorphism
admitting two homoclinic classes $\HC(\cO),\HC(\cO')$
such that $h_\top(f,\HC(\cO)\cap H(\cO'))>0$ but $\cO'\not\hsim\cO$?
\end{question}

\paragraph{c. Nonexistence of m.m.e.}
Our results address the finiteness of the number of ergodic m.m.e.'s, but not their existence.
Examples of $C^r$ surface diffeomorphisms $f$ without m.m.e. are constructed in \cite{BuzziNoMax}. However their topological entropy is smaller than (but arbitrarily close to) $\lambdamin(f)/r$.
It has been asked \cite{BuzziNoMax} whether this is optimal as it is for interval maps (see \cite{BuzziRuette2006,BurguetExistence2014}):

\begin{conjecture}
For  any $1<r<\infty$, any $C^r$ surface diffeomorphism $f$ with topological entropy larger than $\lambdamin(f)/r$ has at least one measure maximizing the entropy.
\end{conjecture}

There are two possible scenarios leading to the non-existence of any m.m.e.
In the examples built in~\cite{BuzziNoMax}
the lack of m.m.e.'s occurs despite the existence of a homoclinic class with maximal entropy (see \cite{Buzzi-Crovisier-Sarig-forthcoming} for details):
\smallskip

\noindent
\begin{theorem*}{For any $1<r<\infty$,
there is a $C^r$ surface diffeomorphism $f$ with a homoclinic class $\HC(\cO)$
such that
 \begin{enumerate}[(1)]
 \item $h_\top(f,\HC(\cO))=h_{\top}(f)=\frac{\lambdamin(f)}r,$
 \item $h(f,\mu)<h_{\top}(f)$, for any $\mu\in\Proberg(f)$.
\end{enumerate}}
\end{theorem*}

But we expect that the examples conjectured in paragraph~\ref{ss.comment}.(a) could be modified so that the supremum of the topological entropies of the homoclinic classes is not achieved, providing a different mechanism for nonexistence:
\smallskip

\noindent
\begin{conjecture}
For any $\varepsilon>0$, any $1<r<\infty$,
there is a $C^r$ surface diffeomorphism $f$ such that
 \begin{enumerate}[(1)]
 \item $h_{\top}(f)>(1-\eps)\frac{\lambdamin(f)}r>0,$
 \item $h_{\top}(f,\HC(\cO))<h_{\top}(f)$, for any $\cO\in\HPO(f)$.
\end{enumerate}
\end{conjecture}
\medbreak

\subsection{Previous finiteness results}\label{s.history}

Various classes of dynamical systems have finitely many ergodic m.m.e.'s.
This was  first shown for uniformly hyperbolic systems: subshifts of finite type \cite{Parry-Entropy},  hyperbolic toral automorphisms in dimension two \cite{Adler-Weiss-PNAS}, Anosov diffeomorphisms \cite{Sinai-Gibbs},  Axiom~A diffeomorphisms \cite{Bowen-Periodic-Points}, or flows \cite{Bowen-1974, Parry-Pollicott-Asterisque}.

This was
generalized to nonuniformly expanding maps: topologically transitive countable state Markov shifts \cite{Gurevich-Measures-Of-Maximal-Entropy},  piecewise-monotonic interval maps \cite{Hofbauer1981}, smooth interval maps \cite{BuzziSIM}, skew-products of those \cite{Buzzi-2005},  surface maps with singularities \cite{Lima-Matheus-2018},  maps satisfying suitable expansivity and specification properties (see, e.g.  \cite{Bowen-Unique-1975,Climenhaga-Thompson-2014}), and related symbolic systems \cite{Buzzi-2005,Thomsen-2005,Buzzi-2010,Climenhaga-Thompson-2012,Pavlov-2016}. Nonuniformly hyperbolic invertible dynamics have been considered: piecewise affine surface homeomorphisms \cite{BuzziPWAH}, various derived from Anosov \cite{Newhouse-LSY-1983,Buzzi-Fisher-Sambarino-Vasquez-2012,Ures-2012,Buzzi-Fisher-2013,Climenhaga-Fisher-Thompson-2014} or  partially hyperbolic systems \cite{Hertz-Hertz-Tahzibi-Ures-2012}. The uniqueness of the m.m.e. has also been established for the H\'enon map \cite{Berger}, rational maps of the Riemann sphere \cite{Mane-MME-1983} and endomorphisms of the complex projective plane \cite{Briend-Duval-2001}.

Further generalizations include nonsingular flows \cite{Lima-Sarig} with positive topological entropy  or assuming some type of specification properties \cite{Climenhaga-Thompson-2016,Pavlov-NUspecification-2017}. Flows of special interest have also been analyzed such as the Teichm\"uller flow \cite{Bufetov-Gurevic-2011}, or the geodesic flow on compact surfaces with nonpositive curvature \cite{Burns-Climenhaga-Fisher-Thompson-2017}.

\smallbreak

Let us note that these results rely on two main approaches.
The first one is to observe that the non-uniform hyperbolicity and the homoclinic relations
provide specification properties (going back to \cite{Bowen-Unique-1975}) which may allow to work directly at the level of the manifold (this is the case for instance in the recent works~\cite{Climenhaga-Thompson-2016,Climenhaga-Thompson-2014,Burns-Climenhaga-Fisher-Thompson-2017}).

We will use the maybe more common strategy of building a semi-conjugacy with a Markov shift. Such a symbolic system decomposes into countably many transitive sets (``irreducible components"). Gurevi\v{c} proved uniqueness \cite{Gurevich-Measures-Of-Maximal-Entropy} for these components, reducing the problem to that of counting the large entropy irreducible components.

\smallbreak

Directly relevant to this work,  $C^r$ surface diffeomorphisms with $r>1$ have been shown in \cite{Sarig-JAMS} to have at most  {\em countably} many ergodic m.m.e.'s with positive entropy.

\subsection{Discussion of the techniques}

The proof of the finiteness of the number of ergodic m.m.e.'s in the Main Theorem has three parts:
\begin{enumerate}[(1)]
\item A homoclinic relation between hyperbolic ergodic measures is introduced and a related measurable partition of the manifold is built:
to each ergodic hyperbolic measure $\mu$ is associated a measurable invariant set $H_\mu$
satisfying $\mu(H_\mu)=1$ and
\begin{itemize}
\item[--] $H_\mu=H_\nu$ when $\mu$ is homoclinically related to $\nu$,
\item[--] $H_\mu\cap H_\nu=\emptyset$ when $\mu$ and $\nu$ are not homoclinically related.
\end{itemize}
\item We show that if $(\mu_i)_{i\geq 1}$ is a sequence of hyperbolic ergodic measures such that $H_{\mu_i}\neq H_{\mu_j}$ for $i\neq j$, then $h(f,\mu_i)\xrightarrow[i\to\infty]{}0$. Consequently,
if $h_{top}(f)>0$ then
at most finitely many sets $H_\mu$ can support a m.m.e.
\item We code every set  $H_\mu$  in a finite-to-one way by an irreducible countable state Markov shift. Since irreducible Markov shifts can carry at most one m.m.e. \cite{Gurevich-Measures-Of-Maximal-Entropy}, $H_\mu$ can carry at most one m.m.e.
 \end{enumerate}
(1), (2) and (3) imply that the number of m.m.e.'s is finite.
The other properties (uniqueness of the m.m.e. in the transitive case
and strong mixing in topological mixing case) are obtained in a subsequent step:
\begin{enumerate}[(4)]
\item Any two ergodic measures with ``large entropy" which are carried by the same transitive set are homoclinically related.
``Large entropy" means entropy bigger than $\frac{\lambdamin(f)}r$ for $C^r$ maps and positive entropy for $C^\infty$ maps.
\end{enumerate}

The first step is based on Pesin theory and applies to any $C^r$-diffeomorphism ($r>1$) in any dimension (see Section~\ref{Sec.homoclinic}).

The second and the fourth steps are specific to surface diffeomorphisms. Indeed in this case any ergodic measure with positive entropy is hyperbolic {thanks to Ruelle-Margulis' inequality}.
One may expect that a lower bound on the entropy gives a control of the geometry of the stable and unstable manifolds of points
in a set with positive measure: this would clearly imply that among a large set of ergodic measures $\{\mu_i\}$ with entropy
uniformly bounded away from zero, two of them have to be homoclinically related.
We were not able to follow that approach.

Instead we first use dimension $2$ and especially that we can bound topological disks ($su$-qua\-dri\-la\-te\-rals) by a two stable and two unstable segments. We then use upper-semicontinuity estimates for the entropy map given by Yomdin-Newhouse theory \cite{Yomdin-Volume-Growth-Entropy,Newhouse-Entropy}
({this uses large entropy}), to obtain topological intersections between stable and unstable manifolds of two different measures
$\mu_i,\mu_j$ that are weak-$*$ close (see Section~ \ref{sec-finite}).

A ``dynamical" version of Sard's lemma allows us then to conclude that some of these topological intersections are transverse
(again this {uses large entropy}).
The use of the classical Sard's Lemma already appeared in the analysis of SRB measures in \cite{Rodriguez-Hertz-Squared-Tahzibi-Ures-CMP}.
But because SRB measures have absolutely continuous conditional measures and m.m.e do not, we need a different, non-standard, version of Sard's Lemma, which we develop in  Section~\ref{Sec.Sard}.
\medskip

After the two first steps, we are reduced to showing the uniqueness of the m.m.e. among a set of hyperbolic measures that are homoclinically related.
We use the semiconjugacy with a Markov shift provided by \cite{Sarig-JAMS}. As discussed in Section~\ref{s.history}, it suffices to choose this Markov shift irreducible. However, since the coding in \cite{Sarig-JAMS} is highly non-canonical and has  lots of redundancies,
it could easily happen that a topologically transitive smooth system is coded by a non-transitive Markov shift with many irreducible components.
We show in Section~\ref{Section-Symbolic} that for each set $H_\mu$, there exists an irreducible component of the coding in \cite{Sarig-JAMS}
which codes the {entire} class $H_\mu$, completing the third step.
\medskip

The fourth step again uses that stable and unstable manifolds are one-dimensional.
Given two hyperbolic measures  carried by a common transitive set, we get intersections of their stable and unstable foliations through the use of  $su$-quadrilaterals. If both measures have, e.g., an unstable foliation with large  transverse dimension, then Sard's lemma shows that any stable leaf of one measure must  transversally intersect some leaf of the unstable foliation of the other measure, leading to a homoclinic relation between the measures.
\medskip

\subsubsection*{Outline of the paper}

Section~\ref{Sec.homoclinic} introduces homoclinic classes for measures (generalizing the classical definition for periodic orbits)
and establishes some general properties.
In the case of surfaces, we build $su$-quadrilaterals as mentioned above.

Section~\ref{Section-Symbolic}
codes homoclinic classes of measures by transitive Markov shifts
by selecting a transitive component of the  global symbolic extension built in \cite{Sarig-JAMS}.
Other codings with injectivity properties are obtained together with some properties of equilibrium measures.

Section~\ref{Sec.Sard} establishes a version of Sard Lemma adapted to dynamical foliations and gives
a criterion for a curve to have a transverse intersection with the stable manifold of a periodic orbit inside a horseshoe
whose dimension is large enough.

Using Yomdin-Newhouse theory and $su$-quadrilaterals, Section \ref{sec-finite} proves that measures with large entropy
cannot accumulate without being homoclinically related.

Using Sard Lemma and $su$-quadrilaterals, Section \ref{s.spectral-decomposition} deduces homoclinic relations from topological transitivity and establishes the Spectral Decomposition Theorem.

Finally, Section~\ref{s-proof-main-theorems} concludes by deducing the main results, announced in the introduction, from the results proved so far.

An appendix establishes the required Lipschitz regularity of
the dynamical foliations for horseshoes on surfaces (such results are folklore but we could not find a reference for the precise statements we need).

\subsubsection*{Acknowledgements}
We thank David Burguet for discussions on tail entropy estimates.
We are grateful to Sheldon Newhouse for his comments on a first version of this text.

\section{Homoclinic classes and horseshoes}\label{Sec.homoclinic}
Throughout this section, unless stated otherwise,  $M$ is a closed manifold of any dimension, and  $f:M\to M$ is a $C^1$ diffeomorphism.
The orbit of a point $p$ is denoted by $\cO(p)$.
The {\em transverse intersection} of $C^1$-submanifolds $U,V\subset M$ is
\begin{equation}\label{equTrans}
   U\pitchfork V:=\{x\in U\cap V: T_x V+ T_x U=T_xM\}.
\end{equation}

{A diffeomorphism $f$ is said to be \emph{Kupka-Smale} if all its periodic orbits are hyperbolic and if all homoclinic and heteroclinic intersections of their invariant manifolds are transverse.}

\subsection{Hyperbolic sets}\label{ss.hyperbolic}

A {\em hyperbolic set} for $f$ is a compact $f$-invariant set $\Lambda\subset M$ with a direct sum decomposition $T_x M= E^s(x)\oplus E^u(x)$ for all $x\in \Lambda$ such that for some $C>0$ and $0<\lambda<1$, for all $x\in \Lambda$, $n\geq 0$, $v^s\in E^s(x)$ and $v^u\in E^u(x)$, we have $\|Df^n_x v^s\|\leq C\lambda^n \|v^s\|$, and  $\|Df^{-n}_x v^u\|\leq C\lambda^n \|v^u\|$.

By \cite{Hirsch-Pugh-Stable-Manifold-Thm}, if $\Lambda$ is hyperbolic, then the following sets are injectively immersed sub-manifolds for every $x\in\Lambda$:
\begin{align*}
W^s(x)&:=\{y\in M: \dist(f^n(y),f^n(x))\xrightarrow[n\to+\infty]{}0\},\\
W^u(x)&:=\{y\in M: \dist(f^{-n}(y),f^{-n}(x))\xrightarrow[n\to+\infty]{}0\}.
\end{align*}
We have $T_x W^u(x)=E^s(x)$ and $T_x W^s(x)=E^u(x)$,  and the convergence in the definition of $W^{s/u}$ is uniformly exponential, see \cite{Shub-Book}.
(The definition of $W^{s/u}(x)$ for non-uniformly hyperbolic orbits, such as typical points of hyperbolic measures, is different, see section \ref{s.hyp-measures}.)

It is also shown in \cite{Hirsch-Pugh-Stable-Manifold-Thm} that, {for $\eps>0$ small enough,}
\begin{align*}
W^s_\epsilon(x)&:=\{y\in M: d(f^{k}(x),f^{k}(y))< \epsilon\text{ for all }k\geq 0\},\\
W^u_\epsilon(x)&:=\{y\in M: d(f^{-k}(x),f^{-k}(y))< \epsilon\text{ for all }k\geq 0\},
\end{align*}
 are uniform open neighborhoods of $x$ in $W^s(x)$ and $W^u(x)$. The subsets $W^s_\epsilon(x), W^u_\epsilon(x)$ are called the  {\em local stable} and {\em unstable manifolds} of $x$ (of size $\epsilon$).

{ A hyperbolic set $\Lambda$ is called {\em locally maximal}, if it has an open neighborhood $V$ such that $\Lambda=\bigcap_{n\in\Z}f^n(\ov{V})$.
 It is known that a hyperbolic set is locally maximal iff it has the following property, called {\em local product structure}: There exist $\epsilon,\delta>0$ such that for every $x,y\in\Lambda$, $W^u_\epsilon(x)\cap W^s_\epsilon(y)$ consists of at most one point, and in case $d(x,y)<\delta$ exactly one point; this point belongs to $\Lambda$;  and the intersection of $W^u_\epsilon(x), W^s_\epsilon(y)$ there is transverse. See \cite{Shub-Book}. }

A {\em basic set} $\Lambda$ for $f$ is an $f$-invariant {compact} set
which is transitive, hyperbolic, and locally maximal. A totally disconnected and infinite basic set is called a
 \emph{horseshoe}.

\subsection{Homoclinic classes of hyperbolic periodic orbits}\label{ss.classes-topo}
{We review some facts and definitions from \cite{Newhouse-Homoclinic}.}
Recall that $\HPO(f)$ is the collection of hyperbolic periodic orbits of saddle type.
\begin{definition}[Smale's order]\label{Def-H-Eq-Newhouse}
Let
 $\mathcal O_1, \mathcal O_2\in\HPO(f)$. We say that $\cO_1$ precedes $\cO_2$ in the \emph{Smale preorder} if $W^u(\mathcal O_1)\pitchfork W^s(\mathcal O_2)\neq \emptyset$. We then write $\cO_1\preceq\cO_2$.
\end{definition}
The condition $W^u(\mathcal O_1)\cap W^s(\mathcal O_2)\neq \emptyset$ implies the existence of orbits which are asymptotic to $\cO_1$ in the past and asymptotic to $\cO_2$ in the future (``$\cO_1$ can come before $\cO_2$"). The transversality of the intersection is used to show that $\preceq$ is transitive, see \cite{Newhouse-Homoclinic}. Since $\preceq$ is obviously reflexive on $\HPO(f)$, the following is an equivalence relation on $\HPO(f)$:

\begin{definition}[Homoclinic equivalence relation]
We say that $\cO_1$ and $\cO_2$ are {\em homoclinically related} if   $\cO_1\preceq\cO_2$ and $\cO_2\preceq{\cO_1}$. We then write $\mathcal O_1\hsim \mathcal O_2$.
\end{definition}

\begin{definition}
The \emph{homoclinic class of a hyperbolic periodic orbit $\mathcal O$} is
$$
\HC(\mathcal O):=\ov{\{x\in \mathcal O': \mathcal O'\in\HPO(f),\  \mathcal O'\hsim\mathcal O\}}\; {\overset{!}{=}}\; \ov{W^u(\mathcal O)\pitchfork W^s(\mathcal O)}.
$$
{The class $\HC(\mathcal O)$ is called \emph{trivial} when it coincides with $\cO$.}
\end{definition}
\noindent
See \cite{Newhouse-Homoclinic} for the proof of $\overset{!}{=}$.
The compact set $\HC(\cO)$ is transitive \cite{Newhouse-Homoclinic}.
In general, different homoclinic classes may have non-empty intersection.

\begin{definition}
The \emph{period of the homoclinic class of $\mathcal O$} is
$$
\ell(\mathcal O):=\gcd(\{\operatorname{Card}(\cO'):\cO'\in\HPO(f),\cO'\hsim\cO\}).
$$
\end{definition}
\cite{Abdenur-Crovisier} shows that:
\begin{proposition}\label{p-hc-cyclic}
Let $p$ be a {hyperbolic} periodic point of saddle type.
If $\ell$ is the period of the homoclinic class of $\cO(p)$
and if $A:=\overline{W^u(p)\pitchfork W^s(p)}$,
then $\HC(\cO(p))=\bigcup_{k=0}^{\ell-1} f^k(A)$. Moreover:
\begin{myitemize}
\item[(1)] $f^\ell(A)=A$ and the restriction of $f^\ell$ to $A$ is topologically mixing.
\item[(2)] For each $i\in \mathbb{Z}$, either $f^i(A)=A$, or $f^i(A)\cap A$ has empty relative interior in $\HC(\cO(p))$.
Hence there is a divisor $L$ of $\ell$ such that $f^L(A)=A$ and $\operatorname{int}_{\HC(\cO(p))}(A\cap f^jA)=\emptyset$ for any $0<j<L$.
\item[(3)] For any $p$ and $q=f^k(p)$ in $\cO$, $W^s(p)\pitchfork W^u(q)\ne\emptyset \iff k\in \ell\mathbb{Z}$.
\end{myitemize}
\end{proposition}

A basic set is contained in the homoclinic class of any of its periodic orbits {\cite{Newhouse-Homoclinic}.}
Conversely, a homoclinic class is approximated by horseshoes; The next result is classical and the proof is identical
to Smale's horseshoe theorem~\cite[Theorem 6.5.5]{Katok-Hasselblatt-Book}:

\begin{proposition}\label{prop-approx-HC}
For every $\cO\in \HPO(f)$, there exists an increasing sequence of horseshoes $\Lambda_n\subset \HC(\cO)$ such that
$\HC(\cO)=\ov{\bigcup_{n\geq1}\Lambda_n}$,  and so that  for every $\cO'\in\HPO(f)$, if $\cO'\hsim\cO$ then $\cO'\subset \Lambda_n$ for some $\Lambda_n$.
\end{proposition}

\subsection{Hyperbolic measures}\label{s.hyp-measures}
We refer to \cite[Chapter S]{Katok-Hasselblatt-Book}, \cite{Barreira-Pesin-Non-Uniform-Hyperbolicity-Book} for the theory of non-uniform hyperbolicity.
We summarize the facts that we will use.

\paragraph{\sc Lyapunov exponents}
Suppose $\mu$ is an invariant probability measure.  By the Oseledets theorem, for $\mu$-almost every $x$, the limit $\chi(x,{v}):=\lim\limits_{n\to\infty}\frac 1n\log\|Df^n_x {v}\|$ exists for all non-zero $v\in T_x M$, and takes just finitely many values as $v$ ranges over $T_x M\setminus\{0\}$. These values are  called the {\em Lyapunov exponents} of $x$. We denote them by
{$\lambda_1(f,x)< \dots<\lambda_{m(x)}(f,x)$.  }
 {When $\mu$ is ergodic, the Lyapunov exponents are constant $\mu$-almost everywhere
 and are denoted by $\lambda_1(f,\mu)<\dots<\lambda_{m(\mu)}(f,\mu)$.}

The measure $\mu$ is \emph{hyperbolic} if  $\lambda_i(f,x)\neq 0$ for every $i$ for $\mu$-a.e. $x$.
In this case $T_xM=E^s(x)\oplus E^u(x)$ {where $E^s(x),E^u(x)$ are linear spaces such that $\lim_{n\to\infty} \frac 1n \log\|Df^n_x v\|<0$ on $E^s(x)\setminus\{0\}$, and
$\lim_{n\to\infty} \frac 1n \log\|Df^{-n}_x v\|<0$ on $E^u(x)\setminus\{0\}$.  }

We say that a hyperbolic measure has \emph{saddle type} if $\lambda_1(f,x) <0<\lambda_{m(x)}(f,x)$ for $\mu$-almost every $x\in M$
(equivalently, both $\dim E^{s}(x)>0$, and $\dim E^{u}(x)>0$ at $\mu$-almost every point $x$).

By Ruelle's inequality \cite{Ruelle-Entropy-Inequality}, if $\dim(M)=2$ then every ergodic invariant measure with positive entropy is hyperbolic of saddle type. We denote the negative Lyapunov exponent by  $-\lambda^s(f,\mu)$, and the positive Lyapunov exponent by $\lambda^u(f,\mu)$.

\paragraph{\sc Invariant manifolds / Pesin blocks.} Suppose $f$ is a $C^r$ diffeomorphism $f$ with $r>1$.
Pesin's theory
asserts that there exist a family of compact sets
(called \emph{Pesin blocks})
$(K_n)_{n\in \mathbb{N}}$ and for each $n$
two families of embedded $C^r$-discs
$(W^s_{loc}(x))_{x\in K_n}$ and $(W^u_{loc}(x))_{x\in K_n}$
(called \emph{local stable} and \emph{unstable manifolds}) such that:
\begin{enumerate}[\quad a.]
\item\label{a}
For each $n$, $f(K_n) \cup K_n\cup f^{-1}(K_n)\subset K_{n+1}$.
Hence the measurable set $Y:=\cup K_n$ is invariant.

\item\label{b} For each $n$, there exists a continuous splitting
$TM|_{K_n}=\mathcal{E}^s\oplus \mathcal{E}^u$ and
for each hyperbolic measure $\mu$ and $\mu$-almost every point $x\in K_n$,
we have $E^s(x)=\mathcal{E}^s(x)$ and $E^u(x)=\mathcal{E}^u(x)$.

\item\label{c} For each $n$, and $x\in K_n$, the discs $W^s_{loc}(x)$, $W^u_{loc}(x)$
contain $x$ and are tangent to $\mathcal{E}^s(x)$ and $\mathcal{E}^u(x)$ respectively.
They vary continuously for the $C^r$-topology with $x$:
They are the images of $C^r$-embeddings $\varphi^s_x\colon \mathcal{E}^s(x)\to M$ and
$\varphi^u_x\colon \mathcal{E}^u(x)\to M$ which vary continuously in the compact-open topology.

\item\label{d} Let $Y^\#$ be the set of points having infinitely many forward and backward
iterates in one $K_n$.

Then, for every point $x\in Y^\#$ the following sets, called  the \emph{stable} and \emph{unstable manifolds of $x$},
 $$\begin{aligned}
    &W^s(x) :=\{y\in M:\limsup_{k\to\infty} \frac1k\log d(f^ky,f^kx)<0\},\\
    &W^u(x) :=\{y\in M:\limsup_{k\to\infty} \frac1k\log d(f^{-k}y,f^{-k}x)<0\},
 \end{aligned}$$
 are injectively immersed $C^r$-submanifolds.

 Moreover for any $n_k\to +\infty$ satisfying $f^{n_k}(x)\in K_n$,
 one has $W^s(x)=\cup_k f^{-n_k}(W^s_{loc}(f^{n_k}(x))$.

 Similarly, for any $n_k\to +\infty$ satisfying $f^{-n_k}(x)\in K_n$,
 one has $W^u(x)=\cup_k f^{n_k}(W^u_{loc}(f^{-n_k}(x))$.

 \item\label{e} Let $Y'$ be the set of points $y$ admitting sequences of forward iterates $f^{n_k}(y)$
and backward iterates $f^{-m_k}(y)$ in the same Pesin block $K_n$ which both converge to $y$.
Then the following property holds:

\begin{lemma}[Inclination Lemma]\label{l.density-submanifolds}
For any $y\in Y'$, any disc $D\subset W^u(y)$, and  any embedded disc $\Delta\subset M$ having a transverse intersection point with $W^s(y)$, there exist discs $D_k\subset \Delta$ and times $n_k\to +\infty$
such that $f^{n_k}(D_k)\to D$ in the $C^1$ topology.

A similar property holds for backward iterates $f^{-n_k}(D_k)$ of discs transverse to $W^u(y)$.
\end{lemma}

\item\label{f} The measurable invariant sets $Y'\subset Y^\#\subset Y$ have full measure for any hyperbolic measure $\mu$.
Moreover, each uniformly hyperbolic set is contained in one of the $K_n$'s.
\end{enumerate}
\medskip

\noindent
The sets $K_n$ are defined as in~\cite[Theorem 1.3.1]{Pesin-Izvestia-1976} or~\cite{KatokIHES}.\footnote{To be precise, one considers a function $N\colon \mathbb{N}\to \mathbb{N}$ that grows fast enough
and one defines $K_n:=\Lambda_{1/N(n), N(n)}$ where $\Lambda_{\chi,\ell}$ is the set defined in~\cite[section 2]{KatokIHES}.}
Property~(\ref{a}) is immediate.
Properties~(\ref{b}) and~(\ref{f}) follow from the Oseledets theorem.
Property~(\ref{c}) is the Pesin stable manifold theorem~\cite[Theorem 2.2.1]{Pesin-Izvestia-1976}.

The global stable manifold (property~(\ref{d})) follows from the properties of the local stable manifold obtained in~\cite{Pesin-Izvestia-1976} (see also~\cite{Katok-Hasselblatt-Book} and~\cite{Barreira-Pesin-Non-Uniform-Hyperbolicity-Book}). The inclusion $\bigcup_k f^{-n_k}W^s_{loc}(f^{n_k}(x))\subseteq W^s(x)$ is immediate.
For the other inclusion, one first chooses $y\in W^s(x)$ and $\chi>0$ such that
$\limsup_{n\to\infty} \frac1n\log d(f^ny,f^nx)<-\chi$.
Then if $n$ is large enough, the point $x$ has arbitrarily large iterates in $K_n$, and
there exists $C>0$ such that for any $z\in K_n$,
 $$
   W^s_{loc}(z)\supset \{\zeta\in M, \; \forall k\geq 0, d(f^k(\zeta),f^k(z))\leq C\exp(-k\chi)\}.
 $$
One deduces that $f^{n_k}(y)$ belongs to $W^s_{loc}(f^{n_k}(x))$ for $n_k$
large enough such that $f^{n_k}(x)\in K_n$. Hence $W^s(x)$ is contained in the union $\bigcup_k f^{-n_k}(W^s_{loc}(f^{n_k}(x))$.

The inclination lemma (property~(\ref{e})) is proved using a graph transform argument.
An embedded disc is an admissible manifold {(see \cite[Section S.4]{Katok-Hasselblatt-Book})} if it belongs to
a small $C^1$ neighborhood of $W^u_{loc}(y)$. For sufficiently small neighborhoods, \cite{Katok-Hasselblatt-Book}
 proves that for any
$n$ large enough, if $f^{n_k}(y_k)$ belongs to $K_n$ and is close enough to $y$, then
the image $f^{n_k}(\Delta)$ contains an admissible manifold. Pulling it back to $\Delta$ gives $D_k$.

\paragraph{\sc $\chi$-hyperbolicity.} Let us fix $\chi>0$.
An ergodic probability measure is \emph{$\chi$-hyperbolic} if it is hyperbolic of saddle type and all its Lyapunov exponents belong to $\mathbb{R}\setminus [-\chi,\chi]$.
A compact invariant set $K$ is called  \emph{$\chi$-hyperbolic}, if it is hyperbolic and if all the ergodic probability measures are $\chi$-hyperbolic.
\medskip

We use the following simple characterization of $\chi$-hyperbolicity:

\begin{proposition}\label{p.chihypset}
An invariant hyperbolic compact set $K$ with hyperbolic splitting $T_KM=E^s\oplus E^u$ is $\chi$-hyperbolic if and only if it admits numbers $C>0$ and $\kappa>\chi$ such that, for all $v^s\in E^s$ and $v^u\in E^u$,
 \begin{equation}\label{eq-chi-uniform}
     \forall n\geq0\quad
       \|Df^n.v^s\|\leq C e^{-\kappa n}\|v^s\| \text{ and }
       \|Df^{-n}.v^u\|\leq C e^{-\kappa n}\|v^u\|.
 \end{equation}
\end{proposition}
\begin{proof}
Obviously, \eqref{eq-chi-uniform} implies $\chi$-hyperbolicity.

For the converse, we assume that \eqref{eq-chi-uniform} does not hold. Then  for every $\kappa>\chi$, there exist arbitrarily large integers and unit vectors $v$ such that
either $v\in E^s$ and $\|Df^n.v\|\geq e^{-\kappa n}\|v\|$ or $v\in E^u$ and $\|Df^{-n}.v\|\geq e^{-\kappa n}\|v\|$.
Assume without loss of generality that the first case holds.

Let $T^s K:=\{v\in TM: x\in K, v\in E^s(x), \|v\|=1\}$ and let $F: T^s K\to T^s K$,  $F(v)=Df(v)/\|Df(v)\|$. This is a homeomorphism, because  $x\mapsto E^s(x)$ is continuous and $f$ is a diffeomorphism.  Let $\varphi:=\log\|Df(v)\|$, then $\vf: T^s K\to\R$ is continuous.

By our assumptions, there are $v_k\in T^s K$ and $n_k\to\infty$  such that $\|Df^{n_k}.v\|\geq e^{-\kappa n_k}\|v\|$ for all $k$.  Every accumulation point of $\tfrac1n\sum_{j=0}^{n_k-1}\delta_{F^j(v_k)}$  is an $F$-invariant probability measure $\widehat \mu$
on $T^s K$  such that $\int_{T^s K} \varphi(v) d\widehat\mu(v)\geq -\chi$. A calculation shows that $\sum_{j=0}^{n-1}\vf\circ F^j=\log\|Df^n\|$, whence
\begin{equation}\label{e.hyp-fails}
\int_{T^s K} \log\|Df^n(v)\|  d\widehat\mu(v)\geq -n\chi\text{ for all }n.
\end{equation}
The measure $\widehat \mu$ projects to an invariant probability measure $\mu$ on $K$,
which by \eqref{e.hyp-fails} must satisfy $\frac{1}{n}\int_K \log\|Df^n|_{E^s(x)}\|d\mu\geq-\chi$ for all $n$. A standard argument now shows that
the largest negative Lyapunov exponent of $\mu$  is larger or equal to $-\chi$.
Consequently, $K$ is not $\chi$-hyperbolic.
\end{proof}

 \subsection{Homoclinic classes of hyperbolic measures}\label{s.classes-measure}
Suppose $M$ is a closed manifold of any dimension and let  $f\in\Diff^r(M)$,  $r>1$. We extend the definition of Smale's pre-order $\preceq$ and the homoclinic relation $\hsim$ to the set of  ergodic hyperbolic  measures of saddle type $\HPM(f)$:
\begin{definition}\label{defHomoclinicMeasures}
For $\mu_1,\mu_2\in\HPM(f)$, we write $\mu_1\preceq \mu_2$ iff there are measurable sets $A_1,A_2\subset M$ with $\mu_i(A_i)>0$ such that for all $(x_1,x_2)\in A_1\times A_2$, the manifolds
$W^u(x_1)$ and $W^s(x_2)$ have a point of transverse intersection.

\end{definition}
\begin{definition}\label{defi-h-rel-meas}
$\mu_1,\mu_2$ are \emph{homoclinically related} if $\mu_1\preceq\mu_2$ and $\mu_2\preceq\mu_1$.
We write $\mu_1\hsim\mu_2$. The set of measures homoclinically related to a measure $\mu$ is called
the \emph{measured homoclinic class of $\mu$}.
\end{definition}

Every hyperbolic periodic orbit of saddle type $\cO$ carries a unique hyperbolic invariant probability measure $\mu_\cO$.
Smale's order and the  homoclinic relation for hyperbolic measures are compatible with Smale's order and the  homoclinic relation for hyperbolic periodic orbits in the following sense:
 $$
\bigl(\cO\preceq\cO'
\quad\Longleftrightarrow\quad \mu_{\cO}\preceq\mu_{\cO'}\bigr)\quad\text{ and }\quad    \bigl(\cO\hsim \cO' \quad \Longleftrightarrow \quad \mu_\cO\hsim\mu_{\cO'}\bigr).
 $$

\begin{proposition}\label{Prop-Equiv-Rel} When $f\in \Diff^r(M)$, $r>1$,
Smale's pre-order {on measures} is reflexive and transitive. Consequently, the homoclinic relation is an equivalence relation {for ergodic hyperbolic measures of saddle type}.
\end{proposition}

\begin{proof}
Reflexivity is clear, we show transitivity: If $\mu_1,\mu_2,\mu_3\in\HPM(f)$ satisfy $\mu_1\preceq \mu_2$ and $\mu_2\preceq \mu_3$, then $\mu_1\preceq\mu_3$.
The following proof is based
on Newhouse's proof that  Smale's pre-order for hyperbolic periodic orbits is transitive \cite{Newhouse-Homoclinic}.

Let $X_i$ $(1\leq i\leq 3)$ be sets  such that (a) $\mu_i(X_i)>0$;
(b) $X_i$ is a Pesin block for $\mu_i$ in the sense of Section~\ref{s.hyp-measures};
(c) Every $y\in X_i$ satisfies the conclusion of the Inclination Lemma~\ref{l.density-submanifolds}; (d) Every neighborhood of every $x_i\in X_i$ has positive $\mu_i$ measure.

Since $\mu_1\preceq \mu_2$ and $\mu_2\preceq \mu_3$, there exist points $x_1,x_2,{x_2'},x_3$ having infinitely many backward iterates in $X_1,X_2,X_2,X_3$ respectively,
which satisfy the conclusion of Lemma~\ref{l.density-submanifolds}
and such that\begin{enumerate}[\quad $\circ$]
  \item $W^u(x_1)\pitchfork W^s(x_2)\neq\emptyset$,
  \item $W^u({x_2'})\pitchfork W^s(x_3)\neq\emptyset$.
\end{enumerate}
Since $\mu_2$ is ergodic, it is possible to arrange $x_2'=x_2$.

Pick a disc $D\subset W^u(x_2)$ such that $D\pitchfork W^s(x_3)\neq\emptyset$.
Then there are discs $D_k\subset W^u(x_1)$ and times $m_k\to +\infty$
such that $f^{m_k}(D_k)$ converges to $D$ in the $C^1$ topology.
This proves that $f^{m_k}(D_k)\pitchfork W^s({x_3})\neq\emptyset$ for $k$ large. We have thus proved that
$f^m(W^u(x_1))$ has a transverse intersection point with $W^s(x_3)$ for some arbitrarily large $m\geq 0$.

There exist $n_1$ and $n_3$ arbitrarily large such that
$f^{-n_1}(x_1)\in X_1$ and $f^{n_3}(x_3)\in X_3$.
One deduces that $f^{m+n_1}(W^u_{loc}(f^{-n_1}(x_1))$ has a transverse intersection point with
$f^{-n_3}(W^s_{loc}(f^{n_3}(x_3))$. Since the local stable and unstable manifolds vary continuously
on Pesin blocks, for any point $x$ in $X_1$ close to $f^{-n_1}(x_1)$ and any point $y$ in $X_3$ close to $f^{n_3}(x_3)$,
the manifolds $f^{m+n_1-n_3}(W^u(x))$ and $W^s(y)$ have a transverse intersection point.
The neighborhoods of $f^{-n_1}(x_1)$ in $X_1$ and of $f^{n_3}(x_3)$ in $X_3$ have positive measure for $\mu_1$
and $\mu_3$ respectively. This shows that $\mu_1\preceq \mu_3$.
\end{proof}

\paragraph{Notation.} For $\cO\in \HPO(f)$ and $\mu\in \HPM(f)$, one writes $\cO\hsim \mu$ when $\mu_\cO \hsim \mu$. This is easily seen to be
equivalent to requiring both $W^u(x)\pitchfork W^s(\cO)\neq\emptyset$ and $W^s(x)\pitchfork W^u(\cO)\neq\emptyset$
for $\mu$-almost every point $x$ as mentioned in the introduction.

One says that a transitive hyperbolic set $\Lambda$ is homoclinically related to $\mu$ and writes $\Lambda\hsim\mu$ if, for some $\nu\in\HPM(f|_\Lambda)$,
we have $\nu\hsim\mu$.
{For any two points $x,y\in \Lambda$, there exists $n\in \mathbb{Z}$ such that $W^u(f^n(x))\pitchfork W^s(y)\neq\emptyset$,
hence any two ergodic measures in $\Lambda$ are homoclinically related.}

\medskip

The following is a slight improvement of  Katok's Horseshoe Theorem \cite{KatokIHES}:

\begin{theorem}\label{t.katok}
For any $f\in \Diff^r(M)$, $r>1$,
any atomless $\mu\in \HPM(f)$, any weak-$*$ neighborhood $U$ of $\mu$ and any $\varepsilon>0$,
there exists a horseshoe $\Lambda$ such that
\begin{enumerate}[(1)]
\item $\HPM(f|_\Lambda)\subset U$,
\item $h_{\top}(f|_\Lambda)>h(f,\mu)-\varepsilon$,
\item if $\chi_1\leq \dots\leq \chi_d$ are the Lyapunov exponents of $\mu$,
counted with multiplicities, then the Lyapunov exponents of any $\nu\in \Proberg(f|_\Lambda)$
satisfy $|\chi_i(f,\nu)-\chi_i|<\varepsilon$,
\item $\Lambda\hsim \mu$,
\item if $(f,\mu)$ is mixing, then $\Lambda$ can be assumed to be topologically mixing.
\end{enumerate}
\end{theorem}
\begin{proof}
The two first items are a folklore strengthening of \cite{KatokIHES} and
are proved in~\cite[Thm. S.5.9]{Katok-Hasselblatt-Book} for the case of a surface $M$.
The three first items in arbitrary dimension are proved in~\cite[Appendix]{ACW}).

Note that the proof in~\cite{{Katok-Hasselblatt-Book}} shows that
for $\mu$-almost every $x$, one can choose $\Lambda$ which contains a point $y$ whose local stable and unstable manifolds
are arbitrarily $C^1$-close to the local stable and unstable manifolds of $x$.
The fourth item follows.

The last item is proved for instance in \cite{Buzzi-Lausanne} - it is in fact sufficient to assume that $\mu$ is totally ergodic, i.e., ergodic with respect to all positive iterates.
\end{proof}

\begin{definition}
The \emph{topological homoclinic class}
of $\mu\in\HPM(f)$ is the set
$$\HC(\mu):=\ov{\bigcup\{\supp\nu:\nu\in\HPM(f)\;, \nu\hsim\mu\}}.$$
\end{definition}
In a horseshoe $\Lambda$, the periodic measures are dense in $\Proberg(f|_\Lambda)$ {\cite[Cor. 6.4.19]{Katok-Hasselblatt-Book}}, hence:
\begin{cor}\label{Cor-HC}
For any $f\in \Diff^r(M)$, $r>1$, and any $\mu\in \HPM(f)$, the set of measures supported by periodic orbits is dense in the set of hyperbolic ergodic measures homoclinically
related to $\mu$, endowed with the weak-$*$ topology. In particular, there exists $\cO\in\HPO(f)$ such that $\cO\hsim \mu$ and
$$\HC(\mu)=\ov{\bigcup\{\supp\nu:\nu\in\HPM(f)\;, \nu\hsim\cO\}}=\HC(\cO).$$
\end{cor}
\noindent
So the definitions of topological homoclinic classes of orbits and measures are consistent.
\medskip

To each measured homoclinic class, we associate a subset of the manifold. Recall the set $Y'$ of regular points introduced in Section~\ref{s.hyp-measures}.

\begin{proposition}\label{prop-HcO}
Suppose $r>1$, $f\in \Diff^r(M)$,  and  $\cO\in\HPO(f)$. The set
 $$H_\cO:=\{x\in Y': W^u(x)\pitchfork W^s(\cO)\neq\emptyset\text{ and }W^s(x)\pitchfork W^u(\cO)\neq\emptyset\},$$ is invariant and  measurable.
 It contains every transitive uniformly hyperbolic set $\Lambda\hsim\cO$, and
\begin{equation}\label{eq-HO}
\forall \mu\in\Proberg(f), \; \mu(H_\cO)=1\iff(\mu\in\HPM(f)  \text{ and } \mu\hsim \cO).
\end{equation}
\end{proposition}

\begin{proof}
The properties of Pesin blocks recalled in Section~\ref{s.hyp-measures} give,
for any $x\in H_\cO$, positive integers $m^+,m^-$ and a Pesin block $P_n$
such that
\begin{equation}\label{e.defHO}
f^{m^+}(x),f^{-m^-}(x)\in K_n,\quad
W^s_{loc}(f^{m^+}(x))\pitchfork W^u(\cO)\neq\emptyset
\text{ and } W^u_{loc}(f^{-m^-}(x))\pitchfork W^s(\cO)\neq\emptyset.
\end{equation}
For any $m^+,m^-,n$, the set of points $x$ satisfying~\eqref{e.defHO} is measurable
(since the local manifolds vary continuously for the $C^1$-topology on each Pesin block).
Hence $H_\cO$ is measurable.

By definition of the homoclinic relation of measures, if an hyperbolic ergodic measure $\mu$ satisfies $\mu\hsim \cO$,
then $\mu(H_\cO)$ is positive. Since $H_\cO$ is invariant $\mu(H_\cO)=1$.
Conversely if $\mu$ is ergodic and satisfies $\mu(H_\cO)=1$, it is hyperbolic of saddle type
(from the properties of the Pesin blocks) and $\mu\hsim \cO$ by definition of $H_\cO$.
\end{proof}

\begin{remark}
We note that if two sets $H_\cO$ and $H_{\cO'}$ intersect, then $\cO\hsim \cO'$ and the two sets coincide.
Indeed, if $x\in H_\cO\cap H_{\cO'}$, then the inclination lemma~\ref{l.density-submanifolds} implies that
the stable manifolds of $\cO$ and $\cO'$ contain discs that converge towards the stable manifold
of $x$ for the $C^1$-topology; the same hold for the unstable manifolds, implying the homoclinic relation between $\cO$ and $\cO'$.
\end{remark}
\medskip

Let us recall that $n\geq 1$ is a \emph{period of an invariant measure} $\mu\in \mathbb P(f)$
if there exists a measurable set $A$ such that $f^i(A)\cap A=\emptyset$ for $0<i<n$, $f^n(A)=A$
and $\mu(A\cup\dots\cup f^{n-1}(A))=1$. When it exists, the largest period is called \emph{the} period of $\mu$.

\begin{proposition}\label{p.period}
For any $f\in \Diff^r(M)$, $r>1$, and any $\cO\in\HPO(f)$
whose homoclinic class has period $\ell=\ell(\cO)$,
there exists a partition
$B_1,\dots,B_\ell$ of the set $H_\cO$ into $\ell$ measurable sets
that are cyclically permuted by the dynamics.
Hence $\ell$ is a period of any measure $\mu\in \HPM(f)$ that is homoclinically related to $\cO$.
Moreover, for any two points $x,y\in H_{\cO}$, we have
 \begin{equation}\label{eq-equiv-period}
 \left(\exists i,\; \{x,y\}\subset B_i\right)\;\Leftrightarrow\; W^s(x)\pitchfork W^u(y)\neq\emptyset
\;\Leftrightarrow\; W^u(x)\pitchfork W^s(y)\neq\emptyset.
 \end{equation}
\end{proposition}

\begin{proof}
Proposition~\ref{p-hc-cyclic}(3) {gives} a partition $A_1,\dots,A_\ell$ of $\cO$ into $\ell$
subsets such that for any $p,q\in \cO$, if $p,q$ belong to the same $A_i$ then
$W^s(p)\pitchfork W^u(q)$ and $W^u(p)\pitchfork W^s(q)$ are both non-empty, and if $p,q$ do not belong to the same $A_i$, then $W^s(p)\pitchfork W^u(q)$ and $W^u(p)\pitchfork W^s(q)$ are both empty.
Note that $A_i$ are  cyclically permuted by the dynamics.

By the definition of $H_\cO$,
for every point $x\in H_\cO$, there exists $p\in \cO$ such that $W^u(x)\pitchfork W^s(p)\neq\emptyset$.
The same property holds for any point $q$ in the set $A_i$ containing $p$,
since the stable manifold of $q$ accumulates on the stable manifold of $p$ (from the inclination lemma~\ref{l.density-submanifolds} applied to $f^{|\cO|}$ at $p$).

$W^u(x)$ does not intersect transversally the stable manifold of a point $q$ in another set $A_j$;
Otherwise,
since $x\in H_\cO\subset Y'$, there is $r\in \cO$ such that $W^u(r)\pitchfork W^s(x)\neq \emptyset$ and
the inclination lemma applied to $f$ at $x$ yields an iterate of $W^u(r)$ which intersects transversally the stable manifolds of points in $A_i$ and $A_j$,
a contradiction.

We have thus associated $\mu$-almost every point $x$ to a unique set $A_i$. The association  is equivariant
and induces a measurable partition into $\ell$ sets
cyclically permuted by the dynamics:
 $$
   B_i := \{x\in H_\cO: \forall p\in A_i,\; W^u(x)\pitchfork W^s(p)\ne\emptyset\}.
 $$

Let us consider any $x\in B_i$ and some point $p\in A_i$.
Arguing with $f^{-1}$ instead of $f$, one gets some set $A_j$ such that
for any $q\in \cO$, the set $W^s(x)\pitchfork W^u(q)$ is non empty if and only if $q$ belongs to $A_j$.
But the inclination lemma at $x$ implies that there exist some iterate $f^{n_k}(W^u(q))$ which
contain a disc close to the local unstable manifold of $x$. In particular $f^{n_k}(W^u(q))$ intersects transversally
$W^s(p)$, so that $f^{n_k}(q)\in A_i$. Note also that $f^{n_k}(W^u(q))$ intersects transversally
$W^s(x)$, which implies that $f^{n_k}(q)\in A_j$. Hence $A_i=A_j$.
As a consequence, the symmetric characterization holds:
 $$
   B_i= \{x\in H_\cO: \forall p\in A_i,\; W^s(x)\pitchfork W^u(p)\ne\emptyset\}.
 $$
The previous argument also shows that for any $x\in B_i$ and $q\in A_i$, the unstable manifold $W^u(q)$
accumulates on $W^u(x)$: this has been obtained for some iterates $f^{n_k}(W^u(q))$ with $f^{n_k}(q)\in A_i$,
but $f^{n_k}(W^u(q))$ and $W^u(q)$ have the same accumulation sets since both $f^{n_k}(q)$ and $q$ belong to $A_i$.

Conversely the unstable manifold $W^u(x)$ accumulates on $W^u(q)$:
indeed, since $x\in Y'$, one can consider a large backward iterate $f^{-n_k}(x)$ close to $x$ in the same Pesin block
(remember item (e) of section~\ref{s.hyp-measures}); as a consequence the local unstable manifold of $x$ and $f^{-n_k}(x)$ are close and intersect
transversally $W^s(q)$ at two points close {to each other}. The inclination lemma, then implies that $W^u(x)$ accumulates on the unstable
manifold $W^u(f^{n_k}(q))$, hence  {on} $W^u(q)$ (since $q$ and $f^{n_k}(q)$ {must} both belong to $A_i$).

If $y\in B_i$, $x\in H_\cO$ and $W^u(x)\pitchfork W^s(y)\neq\emptyset$, then $x\in B_i$:
indeed we deduce from the previous paragraphs that $W^s(p)$ accumulates on $W^s(y)$
for $p\in A_i$; then $W^s(p)$ intersects $W^u(x)$ transversally.

Conversely let us assume that $x,y\in B_i$ and let us fix $p\in A_i$.
We have shown previously that $W^u(x)$ accumulates on $W^u(p)$, hence $W^u(x)\pitchfork W^s(y)\neq\emptyset$.
This concludes the proof of the equivalence, {eq.~\eqref{eq-equiv-period}}.
\end{proof}

\subsection{Quadrilaterals associated to hyperbolic measures}\label{ss-quadrilaterals}

In this section, we assume that $M$ is a closed surface, and use two-dimensionality to associate to each hyperbolic measure topological discs with positive measure of the following type.
\medskip

\begin{definition}\label{def-quadrilateral}
Let $\cO$ be a hyperbolic periodic orbit contained in a horseshoe $\Lambda$.
An {\em $su$-quadrilateral associated to   $\cO$} is an open disc $Q\subset M$ whose boundary is a union of four compact curves
 $$
\partial Q=\partial^{u,1}Q\cup\partial^{s,1} Q\cup\partial^{u,2}Q\cup \partial^{s,2}Q,
 $$
satisfying $\partial^\sigma Q:=\partial^{\sigma,1}Q\cup\partial^{\sigma,2}Q\subset W^\sigma(\cO)$ for $\sigma=s,u$.
We call $\partial^{s,i}Q$, $i=1,2$, the {\em  stable sides} and $\partial^{u,i}Q$ the {\em unstable sides} of $Q$.
\end{definition}

One will get such sets by applying the following proposition.
A measure is {\em atomless} if $\mu(\{x\})=0$ for all $x\in M$.

\begin{prop}\label{p.rectangles}
Let $f$ be a $C^r$ diffeomorphism, $r>1$, of a closed surface and let $\mu$ be an atomless invariant probability measure which is hyperbolic of saddle type.
Then for every $0<t<1$ there are  $su$-quadrilaterals $Q_1,\dots,Q_N$ {associated to $\cO_1, \ldots,\cO_N\in\HPO(f)$}  such that $\diam(Q_i)<t$ and $\mu(\bigcup Q_i)>1-t$. {Given an ergodic decomposition $\mu=\int \mu_x\, d\mu$, one can ensure that for each $i$, $\mu\{x:\mu_x\hsim \cO_i\}\neq 0$.}
\end{prop}

\begin{proof}
Let us  assume first that $\mu$ is ergodic and consider a Pesin block $X$
satisfying $\mu(X)>1-t/2$. {Without loss of generality, $X\subset Y'\cap\supp(\mu)$ where $Y'$ is the set of full measure in the Inclination Lemma (Lemma \ref{l.density-submanifolds}).

{Fix $\tau>0$ small to be determined later.} One can reduce the size of local stable and unstable manifolds
at the points of $X$ in order to satisfy the following property:
 There
 exists $\eps>0$ such that for every $\eps$-close $x,y\in X$,
 the manifolds $W^u_{loc}(x)$ and $W^s_{loc}(y)$ intersect at a unique point $[x,y]$, this intersection is transverse, and $d(x,[x,y])\leq \tau$, $d(y,[x,y])\leq \tau$.

Without loss of generality,    $X$ is the finite union of disjoint compact sets $X_1,\dots,X_\ell$
with {diameter less than $\epsilon$ and such that the $2\tau$-neighborhood of $X_i$ is contained in an open disc $U_i\subset M$}.

Fix a point $p_i$ in $X_i$.
By the compactness of $X_i$ and the continuity of the local unstable manifolds,
there exist two points $p^{s,+}_i,p^{s,-}_i \in X_i\cap W^s_{loc}(p_i)$
such that for any $x\in X_i$,  $[x,p_i]$
belongs to the subcurve  of  $W^s_{loc}(p_i)$ bounded by
$p_i^{s,-}$ and $p_i^{s,+}$.
Similarly there exists two points $p^{u,+}_i,p^{u,-}_i \in X_i$
such that for any $y\in X_i$,  $[p_i,y]$
belongs to the subcurve of  $W^u_{loc}(p_i)$ bounded by
$p_i^{u,-}$ and $p_i^{u,+}$.
See Figure \ref{figure-build-quad}.

Let $\gamma^{s,+}_i$ denote the piece of $W^s_{loc}(p_i^{u,+})$ bounded by $[p^{s,-}_i,p^{u,+}_i]$ and $[p^{s,+}_i,p^{u,+}_i]$, and let  $\gamma^{s,-}_i$ denote the piece of $W^s_{loc}(p_i^{u,-})$ bounded by $[p^{s,-}_i,p^{u,-}_i]$ and $[p^{s,+}_i,p^{u,-}_i]$. Similarly, let $\gamma^{u,+}_i$ denote the piece of $W^u_{loc}(p_i^{s,+})$ bounded by $[p^{s,+}_i,p^{u,-}_i]$ and $[p^{s,+}_i,p^{u,+}_i]$, and let  $\gamma^{u,-}_i$ denote the piece of
 $W^u_{loc}(p_i^{s,-})$ bounded by $[p^{s,-}_i,p^{u,-}_i]$ and $[p^{s,-}_i,p^{u,+}_i]$.

Two {pieces} of local stable manifolds $W^s_{loc}(x), W^s_{loc}(y)$, $x,y\in X_i$, are either disjoint, or the  union { of their closures}  is a $C^1$-curve;
the same holds for local unstable manifolds. Consequently, the four curves
{ $\gamma^{s,+}_i, \gamma^{u,-}_i, \gamma^{s,-}_i, \gamma^{u,+}_i$ form a closed simple curve $\gamma_i$.

By construction, $\gamma_i$  is inside the $2\tau$-neighborhood of $p_i$,  therefore inside $U_i$. By Jordan's Theorem (in the disc $U_i$), $\gamma_i$ bounds an open disc $D_i\subset U_i$ which contains $X_i$ in its ``interior", formally defined to be  the connected piece of $U_i\setminus\gamma_i$ which does not contain $\partial U_i$ in its closure. }

{By choice of $\epsilon$,  $\gamma_i$ has diameter less than $4\tau$. For every closed smooth surface, there exists $\tau_0$ such that every simple closed curve $\gamma$ with diameter less than $\tau_0$ bounds a topological disc with diameter less than $\min\{t,\frac 12\diam(M)\}$. Choosing $\tau:=\frac 14 \tau_0$, we obtain that $\diam(D_i)<t$.   }

Since $\mu$ is an invariant atomless measure, every local stable or unstable manifold has zero measure {(because its forward or backward images have diameters tending to zero)}.} So the boundary of $D_i$ has zero measure with respect to $\mu$.

Theorem~\ref{t.katok} gives a hyperbolic periodic orbit $\cO$ inside a horseshoe homoclinically related to $\mu$.
Since $p_i^{u,\pm}, p_i^{s,\pm}\in X_i\subset X\subset Y$, we have by the  Inclination Lemma~\ref{l.density-submanifolds} that
the four curves $W^s_{loc}(p^{s,+}_i)$, $W^u_{loc}(p_i^{u,+})$, $W^s_{loc}(p_i^{s,-})$, $W^u_{loc}(p_i^{u,-})$
can be $C^1$-approximated by curves {contained} in the manifolds $W^s(\cO)$ and $W^u(\cO)$. These curves bound
a $su$-quadrilateral $Q_i$ {associated to a horseshoe homoclinically related to $\cO$, whence to $\mu$.} By construction, its diameter is  smaller than $t$. Since the boundary of $D_i$
has zero $\mu$-measure, the symmetric difference between $D_i$ and $Q_i$ has arbitrarily small $\mu$-measure.
One deduces that the union of the $D_i$ has measure close to the measure of $X$,
hence larger than $1-t$.

\begin{figure}
\begin{center}
\begin{tikzpicture}[scale=2]
\draw (0,0) node {$\bullet$}; \draw (0.12,-0.12) node {$p_i$};
\draw [very thick] plot[domain=-1.5:1.5] ({0.1*sin(deg(2*\x))},\x) node[above] {$W^u_{loc}(p_i)$};
\draw [very thick] plot[domain=-1.5:1.5] (\x,{0.1*sin(deg(2*\x))}) node[right] {$W^s_{loc}(p_i)$};

\draw  plot[domain=-1.21:1.29] ({0.1*sin(deg(2*\x))+1},\x);
\draw (1.02,{0.1*sin(deg(2*1))}) node{$\bullet$} node[above right] {$p_i^{s,+}$};
\draw  plot[domain=-1.38:1.11] ({0.1*sin(deg(2*\x))-1.1},\x);
\draw (-1.12,{0.1*sin(deg(-2*1.1))}) node{$\bullet$} node[above left] {$p_i^{s,-}$};
\draw  plot[domain=-1.02:1.05] (\x,{0.1*sin(deg(2*\x))+1.2}) node[above right] {$[p_i^{s,+},p_i^{u,+}]$};
\draw (-1.02,{0.1*sin(deg(-2*1.02))+1.2}) node [above left] {$[p_i^{s,-},p_i^{u,+}]$};
\draw (0.07,{0.1*sin(deg(2*0.07))+1.2}) node {$\bullet$} node[above right] {$p_i^{u,+}$};
\draw  plot[domain=-1.14:0.93] (\x,{0.1*sin(deg(2*\x))-1.3}) node[below right] {$[p_i^{s,+},p_i^{u,-}]$};
\draw (-0.05,{0.1*sin(deg(-2*0.05))-1.3}) node {$\bullet$} node[below right] {$p_i^{u,-}$};
\draw (-1.14,{0.1*sin(deg(-2*1.14))-1.3}) node [below left] {$[p_i^{s,-},p_i^{u,-}]$};

\draw [very thin] plot[domain=-1.7:1.8]  ({0.1*sin(deg(2*\x))+0.8},\x) node[above] {$W^u(\cO)$};
\draw [very thick] plot[domain=-0.9:1.0]  ({0.1*sin(deg(2*\x))+0.8},\x);
\draw [very thin] plot[domain=1.7:-2.1] (\x,{0.1*sin(deg(2*\x))+0.9}) node[left] {$W^s(\cO)$};
\draw [very thick] plot[domain=-0.8:0.9] (\x,{0.1*sin(deg(2*\x))+0.9});
\draw  [very thin] plot[domain=-1.7:1.75] ({0.1*sin(deg(2*\x))-0.9},\x) node[above] {$W^u(\cO)$};
\draw  [very thick] plot[domain=-1.1:0.8] ({0.1*sin(deg(2*\x))-0.9},\x);
\draw [very thin] plot[domain=1.7:-2.1] (\x,{0.1*sin(deg(2*\x))-1.0}) node[left] {$W^s(\cO)$};
\draw [very thick] plot[domain=-0.97:0.71] (\x,{0.1*sin(deg(2*\x))-1.0});

\draw [very thick,->,->=latex] (-2.2,0.5) to [out=0,in=180] (-1.05,-0.5);
\draw (-2.2,0.5) node[left] {$\partial Q_i$};
\draw [->,->=latex] (2.2,1.3) to [out=180,in=0] (1.1,1.1);
\draw (2.2,1.3) node[right] {${\partial D_i=\gamma^{s,+}\cup\gamma^{u,-}\cup\gamma^{s,-}\cup\gamma^{u,+}}$};


\end{tikzpicture}
\caption{
The construction of the $su$-quadrilateral $Q_i$. All lines are stable or unstable manifolds (the nearly horizontal ones being the stable ones). }\label{figure-build-quad}
\end{center}
\end{figure}
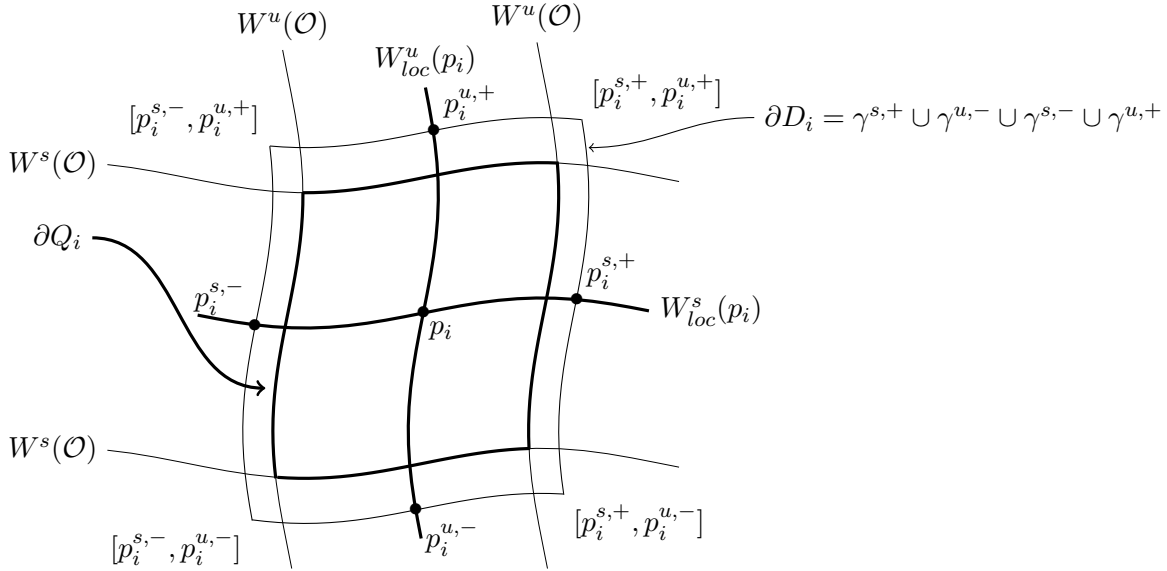

\bigskip

Let us now consider the  case of non-ergodic $\mu$ which are hyperbolic of saddle type.
Let $\Proberg(f)$ denote the collection of
ergodic $f$-invariant measures, and let $\Prob(\Proberg(f))$ denote the collection of Borel probability measures on $\Proberg(f)$.
The ergodic decomposition of $\mu$ gives $\lambda\in\Prob(\Proberg(f))$ such that
 $$
   \mu = \int_{\Proberg(f)} \nu\, d\lambda(\nu).
 $$
Since $\mu$ is hyperbolic of saddle of type, $\lambda$-a.e. $\nu$ belongs to $\HPM(f)$.

Any ergodic hyperbolic  measure is either {atomless} or supported on a hyperbolic periodic orbit.
Since there are at most countably many hyperbolic periodic orbits, and since $\mu$ {is atomless},
$\lambda$-almost every measure $\nu$ {is atomless}.
By the regularity of $\lambda$, there is a compact set $K\subset \Proberg(f)$ with $\lambda(K)>1-t/2$
such that  every measure in $K$ {is atomless}.

By the first part of the proof, for each $\nu\in K$ there is a finite collection $Q_1^\nu,\dots,Q_{N(\nu)}^\nu$ of $su$-quadrilaterals with diameter smaller than $t$ and
satisfying $\nu(\bigcup Q_i^{\nu})>1-t/2$. {Since $\nu(\partial Q_i^\nu)=0$, }
 each $\nu\in K$ has a weak-$*$ neighborhood $\mathcal V(\nu)$ such that
 $$
   \nu'\in \mathcal V(\nu) \implies \nu'\left(\bigcup_{i=1}^{N(\nu)} Q^\nu_i\right)>1-t/2.
 $$
Since  $K$ is compact, it has a finite sub-cover: $K\subset \bigcup_{j=1}^M \mathcal V(\nu_j)$. The collection
 $
    \{Q^{\nu_j}_i:j=1,\ldots,M; \;  1\leq i\leq N(\nu_j)\}
 $
covers a set with $\mu$-measure larger than $(1-\tfrac{t}{2})^2$, and  $\lambda\{\nu:\nu(Q_i^{\nu_j})>0\}\neq 0$ for all $i,j$.
\end{proof}

\section{Symbolic dynamics of homoclinic classes}\label{Section-Symbolic}

The goal of this section is to prove the following theorem. As in the introduction, $\Sigma^\#$ denotes the regular part of a
countable state Markov shift $\sigma\colon \Sigma\to \Sigma$ (see Section~\ref{s.def-symbolic}).

\begin{theorem}\label{Theorem-Symbolic-Dynamics-C^r}
Let $f$ be a $C^r$ diffeomorphism, $r>1$, on a closed surface $M$. Let $\mu$ be an ergodic hyperbolic measure for $f$. For every $\chi>0$ there are  a {locally compact} countable state Markov shift $\Sigma$ and {a H\"older-continuous map} $\pi:\Sigma\to M$ such that  {${\pi}\circ\sigma=f\circ{\pi}$ and}:
\begin{enumerate}[(a)]
\item[\rm (C0)]
$\Sigma$ is irreducible.

\item[\rm (C1)]
$\pi:\Sigma^\#\to M$ is finite-to-one; more precisely if $x_i=a$ for infinitely many $i<0$ and $x_i=b$ for infinitely many $i>0$, then $\#\{\un{y}\in\Sigma^\#: \pi(\un{y})=\pi(\un{x})\}$ is bounded by a constant $C=C(a,b)$.

\item[\rm (C2)]
{\rm (a)} $\nu[\pi(\Sigma^\#)]=1$ for every $\chi$-hyperbolic $\nu\in\HPM(f)$ such that $\nu\hsim\mu$. Moreover,  there is an ergodic measure $\bar{\nu}$ on $\Sigma$ such that $\pi_\ast(\bar{\nu})=\nu$.

{\rm (b)}  Conversely, for every ergodic $\bar{\nu}$ on $\Sigma$, the projection $\nu:=\pi_\ast(\bar\nu)$ is ergodic, hyperbolic, homoclinically related to $\mu$, and $h(f,\nu)=h(\sigma,\bar\nu)$.

\item[\rm (C3)]
For any $x\in\pi(\Sigma)$ there is a unique splitting $T_x M=E^s(x)\oplus E^u(x)$ such that
\begin{enumerate}[(i)]
\item $\limsup_{n\to\infty}\frac 1 n \log\|Df_x^n|_{E^s(x)}\|\leq -\frac \chi 2$;
\item $\limsup_{n\to\infty}\frac 1 n\log\|Df_x^{-n}|_{E^u(x)}\|\leq -\frac \chi 2$.
\end{enumerate}
Moreover the maps $\un{x}\mapsto E^{s/u}(\pi(\un{x}))$ are H\"older continuous on $\Sigma$.
\end{enumerate}
\end{theorem}

\begin{remarks}\label{r.potential}
By (C3), if  $\phi:M\to\R$ is an admissible potential, then  $\phi\circ\pi:\Sigma\to\R$ is H\"older continuous.

The image $\pi(\Sigma)$ is contained in the topological homoclinic class $\HC(\mu)$ (this follows from item {(C2.b)}). For more information on the  measures $\nu\in\HPM(f)$ such that $\nu[\pi(\Sigma^\#)]=1$, see Section \ref{Section-Measure/Topological-HC}.

The inclusion $\pi(\Sigma)\subset \HC(\mu)$ may be strict: this is the case for instance when the homoclinic class $\HC(\mu)$
contains a homoclinic tangency: at such a  point property (C3) cannot be satisfied.
\end{remarks}

\medskip

It is worthwhile to compare Theorem \ref{Theorem-Symbolic-Dynamics-C^r} to the main result of  \cite{Sarig-JAMS}. There, the third author  built a coding $\wh \Sigma\to M$ which
captures all ergodic $\chi$-hyperbolic measures. But this coding is not necessarily irreducible.
The point of Theorem \ref{Theorem-Symbolic-Dynamics-C^r} is that we can obtain irreducibility by localizing to homoclinic classes.
For each class of homoclinically related measures,
we  select an \emph{irreducible} component $\Sigma$ of $\wh \Si$ which captures all the $\chi$-hyperbolic measures in the class.
\medskip

In Section~\ref{ss-lifting-hcm}, we state some further properties for future reference, including lifting transitive compact hyperbolic subsets, the Bowen property and almost everywhere injective coding with respect to any given ergodic measure.
In Section~\ref{equilibrium1}, we deduce the following properties of hyperbolic equilibrium measures.
\begin{corollary}\label{c-local-uniqueness}
Let $r>1$ and let $f$ be a $C^r$ diffeomorphism of a closed surface $M$. Suppose $\phi:M\to\RR\cup\{-\infty\}$ is an admissible potential,
and $\mu$ is an ergodic hyperbolic equilibrium measure for $\phi$. Then:
\begin{myitemize}
\item[--] Any ergodic, hyperbolic, equilibrium measure for $\phi$ which is homoclinically related to $\mu$ is equal to $\mu$.
\item[--] The support of $\mu$ is $\HC(\mu)$.
\item[--] The measure $\mu$ is isomorphic to the product of a Bernoulli scheme with a cyclic permutation of order
$\gcd\{\operatorname{Card}(\cO)\; :\; \cO\hsim \mu\}$.
\end{myitemize}
\end{corollary}

\subsection{Definitions}\label{s.def-symbolic}
We begin with some definitions from symbolic dynamics.
Let $\mathfs G$ denote a countable (possibly finite) directed graph with set of vertices $V$ and set of directed edges $E\subset V\times V$.
We write $v\to w$, when $(v,w)\in E$.  We will always assume that every vertex $v$ has an incoming edge $u\to v$ and an outgoing edge $v\to w$.

The {\em countable state Markov shift} {(or simply \emph{Markov shift} for short)} associated to $\mathfs G$ is the dynamical system
$\sigma:\Sigma\to\Sigma$ defined on the metric space $(\Sigma,d)$ where
$$
\Sigma=\Sigma(\mathfs G):=\{(v_i)_{i\in\mathbb Z}\in V^\Z: v_i\to v_{i+1}\text{ for all }i\},
$$
$d(\un{v},\un{w}):=\exp[-{\inf}\{|i|:v_i\neq w_i\}]$ and $\sigma$ is the {\em left shift map}
$$
\sigma[(v_i)_{i\in\Z}]:=(v_{i+1})_{i\in\Z}.
$$
The set $V$ is called the {\em alphabet} of the shift. Words $(v_1,\ldots,v_n)\in V^n$ such that $v_i\to v_{i+1}$ for all $i$ are called {\em admissible}. Admissible words are also called  {\em paths on $\mathfs G$}. Note that any subshift of finite type is topologically conjugate to a countable state Markov shift.

 It is easy to see that $(\Sigma,d)$ is locally compact iff every vertex has finite valency: the numbers of incoming and outgoing edges at a vertex are finite.
Unless the graph $\mathfs G$ is finite, $\Sigma$ is not compact. Gurevi\v{c} has extended the usual topological entropy to non compact Markov shift by  setting:
 \begin{equation}\label{eq-hG}
    h(\Sigma) := \sup_{\mu\in\Prob(\sigma,\Sigma)} h(f,\mu).
  \end{equation}

It is also not difficult to see that $\sigma:\Sigma\to\Sigma$ is topologically transitive iff  for any $v,w\in V$
\begin{equation}\label{e.transition}
\exists (v_i)_{\in \mathbb{Z}}\in \Sigma,\;n\geq 1,\quad  v_0=v \text{ and }v_n=w.
\end{equation}
When this happens we say that  $\Sigma$ is  \emph{transitive} (or \emph{irreducible}).
In this case, one defines the \emph{period} of $\Sigma$ as the greatest common divisor of the lengths of the loops in $\mathfs G$.

In the non-transitive case, one introduces the subsets $V'\subset V$ which satisfy~\eqref{e.transition} for any $v,w\in V'$
and which are maximal for the inclusion. They induce subshifts $\Sigma'\subset \Sigma$ that are called \emph{irreducible components}
of $\Sigma$.

\begin{definition}
The regular part of $\Sigma$ is the subset
$$
\Sigma^\#:=\{(v_i)\in \Sigma: \text{both } (v_i)_{i>0}, (v_i)_{i<0}\text{ contain constant subsequences}\}.
$$
\end{definition}
It has full measure with respect to every shift invariant measure on $\Sigma$.

\subsection{Global symbolic dynamics}\label{ss.global-symbolic}
We recall the results of~\cite{Sarig-JAMS}:

\begin{thm}[\cite{Sarig-JAMS}]\label{thm.sarig}
Let $r>1$ and let $f$ be a $C^{r}$ diffeomorphism on a closed surface $M$. For every
$\chi>0$, there are a locally compact countable state Markov shift $\wh \Sigma$ and a H\"older continuous
map ${\wh \pi}:{\wh \Sigma}\to M$  such that ${\wh \pi}\circ\sigma=f\circ{\wh \pi}$ and:
\begin{enumerate}[(a)]
\item[\rm (C1)]\label{symbolic2}  $\wh \pi:\wh \Sigma^\#\to M$ is finite-to-one; more precisely,
if $x_i=a$ for infinitely many $i<0$ and $x_i=b$ for infinitely many $i>0$, then $\#\{\un{y}\in\wh \Sigma^\#:\wh \pi(\un{y})=\wh \pi(\un{x})\}$ is bounded by a constant $C(a,b)$.
\item[\rm ($\rm \widehat{C}$2)]\label{symbolic3} $\nu({\wh \pi}({\wh \Sigma^\#}))=1$ for every $\chi$-hyperbolic measure $\nu\in \HPM(f)$.
Moreover, there exists an ergodic measure $\bar \nu$ on $\wh \Sigma$
such that $\wh \pi_*(\bar \nu)=\nu$. Conversely, if $\bar\nu$ is a $\sigma$-ergodic measure on $\Sigma$, then $\wh\pi_\ast\bar\nu$ is $f$-ergodic, hyperbolic, and $h(f,\nu)=h(\sigma,\bar\nu)$.
\item[\rm (C3)]\label{symbolic7} for any $x\in \wh \pi(\wh \Sigma)$, there is a splitting
$T_x M=E^s(x)\oplus E^u(x)$ where:
\begin{enumerate}[(i)]
\item $\limsup_{n\to +\infty} \frac 1 n \log \| Df_x^n|_{E^s(x)}\|\leq -\frac \chi 2$,
\item $\limsup_{n\to +\infty} \frac 1 n \log \| Df_x^{-n}|_{E^u(x)}\|\leq -\frac \chi 2$.
\end{enumerate}
Moreover, the maps ${\un x}\mapsto E^{s/u}(\wh \pi({\un x}))$ are H\"older continuous on $\wh \Sigma$.
\item[\rm (C4)]\label{symbolic8} For every $\un{x}\in\wh\Sigma$ there are two $C^1$ sub-manifolds $V^u(\un{x}),V^s(\un{x})$ passing through $x:=\wh{\pi}(\un{x})$ s.t.:
\begin{enumerate}[(i)]
\item $\forall y\in V^s(\un{x})$, $\forall n\geq0$, $d(f^{n}(y), f^{n}(x))\leq e^{-\frac{n\chi}{2}}$ and $T_x V^s(\un{x})=E^s(x)$,
\item $\forall y\in V^u(\un{x})$, $\forall n\geq0$, $d(f^{-n}(y), f^{-n}(x))\leq e^{-\frac{n\chi}{2}}$ and $T_x V^u(\un{x})=E^u(x)$.
\end{enumerate}
\end{enumerate}
\end{thm}

Properties (C1) and ($\rm \widehat {\text{C}}$2)
 correspond to Theorems 1.3, 1.4, 1.5 of \cite{Sarig-JAMS}. There the theorems are stated under the stronger condition that $h(f,\nu)>\chi$ but in fact $\chi$-hyperbolicity is all that is used. Property (C3) is Proposition 12.6 in \cite{Sarig-JAMS}. Property (C4) follows from Proposition 6.3 and Definitions 10.3 and 11.4 in \cite{Sarig-JAMS}.
\medskip

We need another property
constructed in \cite{Sarig-JAMS}, which was identified and brought to the fore in \cite{Boyle-Buzzi}:
\begin{enumerate}[{\it (f)}]
\item[(C5)] {\em  {\em Locally finite Bowen Property.\/} There is a symmetric binary relation $\sim$ on  the alphabet $V$ of {$\wh\Sigma$} satisfying}
 $$\begin{aligned}
  &\forall x,y\in\wh{\Sigma}^\#,\;\; \wh{\pi}(x)=\wh{\pi}(y) \iff (\forall n\in\Z\;\; x_n\sim y_n),\\
  & \forall b\in V,\  \#\{a\in V:a\sim b\}<\infty.
 \end{aligned}$$
\end{enumerate}
The binary relation $\sim$ is the {\em affiliation relation} defined in \cite[\S 12.3]{Sarig-JAMS}. Since we do not need its precise definition,  we will not repeat it.

\begin{proposition}\label{p-proj-measure}
{Let $r>1$ and let $f$ be a $C^{r}$ diffeomorphism on a closed manifold $M$. Let $\widehat \Sigma$ be a locally compact countable state Markov shift and suppose ${\wh \pi}:{\wh \Sigma}\to M$ is a H\"older continuous map satisfying ${\wh \pi}\circ\sigma=f\circ{\wh \pi}$
and (C3) for some $\chi>0$. Then the following property holds:
\begin{enumerate}[{\it (f)}]
\item[\rm (C6)]
For any ergodic measure $\bar \nu$ on $\wh \Sigma$, $\nu:=\wh \pi_*(\bar \nu)$ is $\chi/3$-hyperbolic. Moreover, if $\bar\nu_1,\bar\nu_2$ are two such measures on a common irreducible component of $\wh\Sigma$, their projections $\nu_1,\nu_2$ are homoclinically related.
\end{enumerate}}
 \end{proposition}
\begin{proof}
The first assertion follows from (C3). For the second assertion, suppose $\bar \nu_1,\bar \nu_2$ are supported inside the same irreducible component.
There exist two measurable sets $A, B\subset \wh \Sigma$ with $\bar \nu_1(A)>0$,
$\bar \nu_2(B)>0$ such that the elements $(v_i)_{i\in \mathbb Z}$ in $A$ have the same
symbol $v_0=v^A$,  the elements $(v_i)_{i\in \mathbb Z}$ in $B$ have the same
symbol $v_0=v^B$, and there exists a sequence of symbols
$v^A\to v^1\to\dots\to v^{k-1}\to v^B$. Let $B':=\sigma^{-k}(B)$, then $\bar{\nu}_2(B')>0$ and every $(w_i)_{i\in\mathbb Z}\in B'$ satisfies $w_k=v^B$.

By Pesin's Stable Manifold Theorem \cite{Barreira-Pesin-Non-Uniform-Hyperbolicity-Book}, there is a subset $A''\subset \wh{\pi}(A)$ with positive $\nu_1$-measure  such that every $x\in A''$ has a well-defined  unstable manifold $W^u(x)$,
 and there is a subset $B''\subset \wh{\pi}(B')$ with positive $\nu_2$-measure such that every $y\in B''$ has a well-defined  stable manifold $W^s(y)$. One can also assume that points in $A''\cup B''$ have an Oseledets splitting.
 This splitting must coincide with the splitting $E^s\oplus E^u$ given by condition (C3).

 We claim
 that for every $x\in A''$ and $y\in B''$,  $W^u(x)\pitchfork W^s(y)\neq\emptyset$. Indeed, write $x=\wh{\pi}(\un{x})$ with $x_0=v^A$ and $y=\wh{\pi}(\un{y})$ with $y_k=v^B$, and consider
 $
 z:=\wh\pi(\un{z})
 $ where $z_i=x_i$ for $i\leq 0$, $z_i=v_i$ for $1\leq i{<}k$, and $z_i=y_i$ for $i\geq k$. Then $\un{z}\in \wh{\Sigma}$, $d(\sigma^{-n}\un{z},\sigma^{-n}\un{x})\leq e^{-n}$, and
 $d(\sigma^{n}\un{z},\sigma^{n}\un{y})\leq e^{k-n}$. By the H\"older continuity of $\wh{\pi}$, $d(f^n(z),f^n(y)), d(f^{-n}(z),f^{-n}(x))\xrightarrow[n\to\infty]{}0$ exponentially fast, so
   $z\in W^u(x)\cap W^s(y)$.
The exponential convergence also implies  that $V^u(\un{z})\subset W^u(x)$, $V^s(\un{z})\subset W^s(y)$ where $V^{u/s}(\un{x})$ are as in
(C4). So $T_z W^{u/s}(z)=T_z V^{u/s}(\un{x})=E^{u/s}(z)$ with $E^{u/s}(z)$ as in (C3). Since $E^u(z)\cap E^s(z)=\{0\}$,
the manifolds $W^u(x), W^s(y)$ are transverse at $z$.
This proves $\nu_1\preceq\nu_2$.
By symmetry $\nu_2\preceq\nu_1$, whence  $\nu_1\hsim\nu_2$.
\end{proof}

As usual we endow the space of Borel probability measures on $\wh \Sigma$ with the weak-$*$ topology.

\begin{proposition}\label{p-contExp}
{Let $r>1$ and let $f$ be a $C^{r}$ diffeomorphism on a closed manifold $M$. Let $\widehat \Sigma$ be a locally compact countable state Markov shift and suppose ${\wh \pi}:{\wh \Sigma}\to M$ is a continuous map satisfying ${\wh \pi}\circ\sigma=f\circ{\wh \pi}$
and (C3). Then the following property hold.
\begin{enumerate}[{\it (f)}]
\item[\rm (C7)]
For any $\chi'>0$, the set of ergodic measures $\bar \nu$ whose projection
$\wh\pi_*(\bar \nu)$ is $\chi'$-hyperbolic is open
 for the {relative} weak-$*$ topology of $\Proberg(\wh\Sigma)$.
\end{enumerate}}
\end{proposition}
\begin{proof}
Let us consider an ergodic measure $\bar \nu_0$ such that $\wh\pi_*(\bar \nu_0)$ is $\chi'$-hyperbolic:
there exists $\lambda_0<-\chi'$ such that all its Lyapunov exponents along $E^s$ are smaller than $\lambda_0$.
For any $\delta>0$, there exists an integer $N\geq 1$ such that the open set
$$A:=\{x\in \widehat \Sigma\; :\; \forall v\in E^s(\wh\pi(x)),\;\|Df^N(v)\|< \exp(-N\lambda_0)\|v\|\}$$
has measure for $\bar \nu_0$ larger than $1-\delta$.
Since $E^s$ is continuous over $\wh \Sigma$ by (C3), this is still the case for any ergodic measure
$\bar \nu$ close to $\bar \nu_0$.
Having chosen $\delta>0$ small enough, this implies that for such $\bar \nu$,
$$\frac 1 N \int \log \|Df^N|_{E^s}\|d\bar \nu<-\chi'.$$
A  sub-additivity argument shows that $\lim \frac{1}{n}\int\log\|Df^n|_{E^s}\|d\bar\nu=\inf \frac{1}{n}\int\log\|Df^n|_{E^s}\|d\bar\nu<-\chi'$. In particular all the Lyapunov exponents of $\bar \nu$ along $E^s$ are smaller than $-\chi'$.
Arguing in the same way for the unstable exponents, one concludes that $\bar \nu$ is $\chi'$-hyperbolic.
\end{proof}

Note that a set $S\subset\wh\Sigma$ is relatively compact if and only if $\{x_k:x\in S\}$ is finite for each $k\in\mathbb Z$.
If $S\subset \wh\Sigma$ is relatively compact in $\wh\Sigma$, then $\wh{\pi}(S)\subset M$ is relatively compact (but the converse is not necessarily true).

\begin{proposition}[Compactness]\label{p.properness}
{Let $f$ be a homeomorphism of a compact metric space $M$, let $\widehat \Sigma$ be a locally compact countable state Markov shift and let ${\wh \pi}:{\wh \Sigma}\to M$ be a continuous map satisfying ${\wh \pi}\circ\sigma=f\circ{\wh \pi}$
and (C5). Then the following property holds:
\begin{enumerate}[{\it (f)}]
\item[\rm (C8)]
Consider a sequence $\un{x}^1,\un{x}^2,\un{x}^3,\dots\in \wh \Sigma^\#$
which is relatively compact in $\wh\Sigma$.
Then, any sequence $\un{y}^1,\un{y}^2,\un{y}^3,\dots\in \wh \Sigma^\#$
such that $\wh \pi(\un{x}^i)=\wh \pi(\un{y}^i)$ for each $i\geq 1$
is also relatively compact.
\end{enumerate}}
\end{proposition}

\begin{proof}
 Since  $(\underline x^i)_{i\geq1}$ is relatively compact, $\{x^i_k:i\geq 1\}$ is finite for all $k$. So $\mathcal A_k:=\{a:\exists i\geq1$ $a\sim x^i_k\}$  is finite for all $k$, where $\sim$ is the relation in (C5). Since $\wh\pi(\underline y^i)=\wh\pi(\underline x^i)$,
 $y^i_k\in \mathcal A_k$ for all $i,k$.  The finiteness of $\mathcal A_k$ implies that  $(\underline y^i)_{i\geq 1}$ is  relatively compact.
\end{proof}
\medskip

\subsection{Lifting transitive $\chi$-hyperbolic compact sets}\label{ss.additional}
In this section we prove the following:
\begin{proposition}\label{p.uniform}
In the setting of Theorem \ref{thm.sarig},
the coding $\widehat \pi\colon \widehat\Sigma\to M$ can be chosen to satisfy the following additional property.
\begin{enumerate}[{\it (f)}]
\item[\rm ($\widehat{\rm C}$9)]
For any transitive invariant $\chi$-hyperbolic compact set $K\subset M$,
there exists a \emph{transitive} invariant compact set $X\subset \wh \Sigma$
such that $\wh \pi(X)=K$.
\end{enumerate}
\end{proposition}
\noindent
We check that the construction performed in \cite{Sarig-JAMS} satisfies this property.
In the next section we summarize the steps of the proof of Theorem \ref{thm.sarig} and give the precise corresponding references to \cite{Sarig-JAMS}.

\subsubsection{The global coding: some properties of the construction}

The $\chi$-hyperbolic dynamics of $f$ on $M$ is analyzed through the following construction:

\begin{itemize}
\item[--] \textbf{An invariant measurable set $\NUH^\#_\chi$.}
Its construction is explained below. It satisfies:
\begin{itemize}
\item[(P1)] \it The set $\NUH^\#_\chi$ has full measure for any ergodic $\chi$-hyperbolic measure of saddle type.
\end{itemize}

\item[--] \textbf{A map $q_\varepsilon \colon \NUH^\#_\chi\to (0,+\infty)$.} It measures the quality of the hyperbolicity.
It is uniformly bounded from above, but  it could have arbitrarily small positive values.  {However:}
\begin{itemize}
\item[(P2)] \it If $K$ is a $\chi$-hyperbolic compact invariant set, then $q_\varepsilon$ has a positive lower bound over $K\cap\NUH_\chi^\#$.
\end{itemize}

\item[--] \textbf{A countable collection $\mathcal{R}$ of pairwise disjoint Borel sets.}
These sets are obtained as rectangles of a Markov partition and satisfy:
$$\bigcup_{R\in\mathcal{R}} R\supset \NUH^\#_\chi.$$
Moreover the following finiteness property holds:
\begin{itemize}
\item[(P3)] \it For each $t>0$, the set $\{x\in \NUH^\#_\chi, q_\varepsilon(x)>t\}$ meets only finitely many $R\in \mathcal R$.
\end{itemize}

\item[--] \textbf{A Markov shift $\widehat \Sigma$ together with a map $\widehat \pi\colon \widehat \Sigma\to M$.}
The shift $\wh \Sigma$ is associated to the collection of vertices $\mathcal{R}$ and to the collection
of edges $R\to S$ such that $f(R)\cap S\neq \emptyset$. It codes the set $\NUH^\#_\chi$:
\begin{itemize}
\item[(P4)] \emph{Given any $x\in\NUH^\#_\chi$, consider the sequence $\underline R\in\wh\Sigma$ such that, for all $n\in\Z$, $f^n(x)\in R_n$.
Then $\wh\pi(\underline R)=x$.}
\end{itemize}
\end{itemize}

\paragraph{Comments on the constructions.}
Let us briefly explain where and how these constructions are performed in~\cite{Sarig-JAMS}.
\smallskip

\paragraph{\it Step 1. Non-uniformly hyperbolic dynamics.}
The set $\NUH^\#_\chi$  is described in~\cite[Section~2.5]{Sarig-JAMS}.
It consists of points $x\in M$ such that:
\begin{enumerate}[(a)]
\item $T_xM$ admits a splitting $T_xM=E^s(x)\oplus E^u(x)$ satisfying
\begin{enumerate}[(1)]
\item $E^s(x)=\Span\{\un{e}^s(x)\}$, $\|\un{e}^s(x)\|=1$, $\lim\limits_{n\to\pm\infty}\frac{1}{n}\log\|(Df^n)\un{e}^s(x)\|<-\chi$;
\item $E^u(x)=\Span\{\un{e}^u(x)\}$, $\|\un{e}^u(x)\|=1$, $\lim\limits_{n\to\pm\infty}\frac{1}{n}\log\|(Df^n)\un{e}^u(x)\|>\chi$;
\item $\lim\limits_{n\to\pm\infty}\frac{1}{n}\log|\sin\alpha(f^n(x))|=0$, where
$\alpha(x):=\measuredangle(E^s(x),E^u(x))$.
\end{enumerate}
To these points one associates positive numbers $u(x)$ and $s(x)$ defined by
$$u(x)^2:=2\sum_{n=0}^\infty e^{2n\chi}\|Df^{-n} e^u(x)\|^2, \quad s(x)^2:=2\sum_{n=0}^\infty e^{2n\chi}\|Df^{n} e^s(x)\|^2$$
and a number $Q_\eps(x)>0$, which is an explicit continuous function of $\alpha(x)$, $1/u(x)$, $1/s(x)$, and which goes to zero as any
of these three quantities goes to zero, see \cite{Sarig-JAMS}, Section 2.3.
\item[(b)] $\frac 1 n \log Q_\epsilon(f^n(x))\to 0$ as $|n|\to \infty$.

\item[(c)]
 $\underset{n\to +\infty}{\limsup}\; q_\varepsilon(f^n(x))\neq 0$ and $\underset{n\to -\infty}{\limsup}\; q_\varepsilon(f^n(x))\neq 0$, where
 $$q_\varepsilon(x)^{-1}:=\varepsilon^{-1}\sum_{k\in\mathbb Z} e^{-\varepsilon|k|/3} Q_{\varepsilon}(f^k(x))^{-1}.$$

\end{enumerate}
Property (P1) follows from the Oseledets Theorem, see \cite[Sec. 2.5]{Sarig-JAMS}.

\paragraph{\it Uniformly hyperbolic dynamics.}
Let $K$ be a $\chi$-hyperbolic compact invariant set.
At any point $x$ in $K\cap \NUH^\#_\chi$,
the splitting given by the definition of $\NUH^\#_\chi$ corresponds
to the hyperbolic splitting.
By Proposition~\ref{p.chihypset}, the $\chi$-hyperbolicity of $K$ implies
that {the functions $\alpha(x)$, $1/u(x)$, $1/s(x)$ which appear in the definition of $\NUH^\#_\chi$}
are uniformly bounded away from zero.
In particular $Q_\varepsilon$ and therefore $q_\varepsilon$ are also bounded away from zero, proving (P2).

\medbreak

\paragraph{\it Step 2. A first  Markov shift $\Sigma$.}
A first countable collection $\mathcal A$ of Pesin charts  $\Psi_x^{p^u,p^s}$ is built in~\cite[Proposition 3.5]{Sarig-JAMS}:
$\Psi_x^{p^u,p^s}$ is a pair of two concentric Pesin charts with the same center $x$, but different sizes $p^u, p^s$.
The set $\{\Psi_x^{p^u,p^s}\in \mathcal A: \min(p^s,p^u)>t\}$
is finite for each $t>0$. Definition 4.3 in~\cite{Sarig-JAMS} introduces a directed graph with vertices $\mathcal A$,
the associated shift $\Sigma$, and a projection $\pi\colon \Sigma\to M$ (\cite[Theorem 4.16]{Sarig-JAMS}).
For each $v\in \mathcal A$, one sets $Z(v)=\{\pi(\underline u): \underline u\in \Sigma^\# \text{ and } u_0=v\}$.
Each point $x\in \NUH_\chi^\#$ lifts by $\pi$ to a sequence $(\Psi_{x_n}^{p^u_n,p^s_n})\in \Sigma^\#$
satisfying (\cite[proof of Proposition 4.5]{Sarig-JAMS}):
\begin{equation}\label{e.lowerbound}
\forall n\in \mathbb Z, \quad \min(p^u_n,p^s_n)\geq q_\varepsilon(f^n(x))e^{-\varepsilon/3}.
\end{equation}
The sets $Z(v)$ define a covering $\mathcal Z$ (\cite[Section 10.1]{Sarig-JAMS}) of
an invariant set containing $\NUH_\chi^\#$. The covering is {\em locally finite} \cite[Theorem 10.2]{Sarig-JAMS}: For every $Z\in\mathcal Z$, $\#\{Z'\in\mathcal Z: Z'\cap Z\neq \emptyset\}<\infty$.

\paragraph{\it Step 3. The subshift $\wh\Sigma$.} The local finiteness of the cover $\mathcal Z$ is used in  \cite[Section 11]{Sarig-JAMS}  to construct a collection $\mathcal R$ of
{\em pairwise disjoint} sets such that $\bigcup \mathcal R=\bigcup\mathcal Z$, $\forall R\in\mathcal R$ $(\exists Z\in\mathcal Z$ such that $R\subset Z)$, and $\forall Z\in\mathcal Z$ $(\#\{R\in\mathcal R: Z\supset R\}<\infty)$.

Property (P3) can be checked as follows. By step 2,  the set $\{\Psi_x^{p^u,p^s}\in\mathcal A: \min(p^u,p^s)>t\}$ is finite for every $t>0$.
By \eqref{e.lowerbound},  $\{x\in\NUH_\chi^\#: q_\epsilon(x)>t\}$ can be convered by finitely many $Z\in\mathcal Z$. The local finiteness implies that it
meets only finitely many $Z\in\mathcal Z$.
Each $Z$ contains at most finitely many $R\in\mathcal R$. So (P3) follows.

{Lemmas 12.1 and 12.4 in~\cite{Sarig-JAMS} prove that for any $\underline R\in \widehat \Sigma$,
the sequence $\operatorname{Closure}\left( \cap_{k=-n}^n f^{-k}(R_k)\right)$ decreases to a singleton: by definition, this is
$\widehat \pi(\underline R)$.
Property (P4) immediately follows from the definitions of $\wh \Sigma$ and $\wh \pi$.}

\subsubsection{Proof of Proposition~\ref{p.uniform}}\label{s-p-uniform}

First, we substitute for $K$ a larger set $\widetilde K$ with a fully supported invariant probability measure as follows. Since $K$ is hyperbolic and transitive, for any neighborhood $V$ of $K$, there is a closed invariant set $K\subset\widetilde K\subset V$, equal to  a continuous factor of a transitive shift of finite type (built from specification and expansiveness).  Thus there is an ergodic invariant probability measure $\nu$ on $\widetilde K$ with full support. Since hyperbolicity is an open property,
the set $\widetilde K$ is still $\chi$-hyperbolic (choosing the neighborhood $V$ small enough) and from property (P1) we get $\nu(\widetilde K\cap \NUH_\chi^\#)=1$.
Since $\supp(\nu)=\widetilde K$, the set $\widetilde K\cap\NUH_\chi^\#$ is dense in $\widetilde K$.

Let us consider $q_\eps:\NUH_\chi^\#\to(0,+\infty)$ from the preceding section.
According to (P2), this function has a nontrivial lower bound on $\widetilde K\cap\NUH_\chi^\#$.
For each $x\in K\cap\NUH_\chi^\#$, Property (P4) defines a sequence $\underline{R}(x)\in\wh\Sigma$.
Property (P3) implies that all  sequences $\underline{R}(x)$, $x\in K\cap\NUH_\chi^\#$, only use finitely many symbols hence are contained in some invariant compact set $X_0\subset \wh\Sigma$. As $\wh\pi$ is continuous, $\wh \pi(X_0)$ is a compact set which contains $\widetilde K\cap \NUH_\chi^\#$, hence $K$.

Let $X\subset X_0$ be an invariant compact set such that $\wh \pi(X)\supset K$,
and which is minimal for the inclusion (such a set exists by Zorn's Lemma).
By assumption,  there exists $z\in K$ having a dense forward orbit
in $K$. Consider a lift $x\in X$ of $z$.
The $\omega$-limit set of the forward orbit of $x$ is an invariant compact subset
of $X$, which projects on $K$ by $\wh \pi$ since the forward orbit of $z$ is dense in $K$.
{Since $X$ has been chosen minimal}, this limit set coincides with $X$, hence the forward orbit of $x$ is dense in $X$.
This shows that $X$ is transitive. The proof of Proposition~\ref{p.uniform} is complete.
\qed

\subsection{Lifting homoclinic classes of measures}\label{ss-lifting-hcm}
In this section, we prove Theorem~\ref{Theorem-Symbolic-Dynamics-C^r} as a consequence of the following:

{\begin{proposition}\label{SymbolicDynamics}
Let $r>1, \chi>0$ and let  $f$ be a $C^{r}$ diffeomorphism on a closed manifold $M$.  Consider a locally compact countable state Markov shift $\widehat \Sigma$, and a continuous map ${\wh \pi}:{\wh \Sigma}\to M$ satisfying ${\wh \pi}\circ\sigma=f\circ{\wh \pi}$
and (C1), ($\widehat{ C}$2), (C5), (C6), (C7), ($\widehat{ C}$9). Then the following property holds.

For any hyperbolic ergodic measure $\mu$
there is an irreducible component $\Sigma\subset \widehat \Sigma$ satisfying (C2).
\end{proposition}}

\begin{proof}
Let us fix $\mu\in \HPM(f)$.
One can assume that the homoclinic class of $\mu$ contains a $\chi$-hyperbolic measure,
since otherwise the conclusion of the theorem holds trivially.
Recall that a basic set is  a compact, invariant, transitive, uniformly hyperbolic, locally maximal set.

\begin{lemma}\label{Lemma-K_n}
Suppose $\cO_1,\cO_2,\ldots,\cO_N$ are homoclinically related $\chi$-hyperbolic periodic orbits. Then there exists a $\chi$-hyperbolic basic set $K$ which contains every $\cO_i$.
\end{lemma}
\begin{proof}
Since {the orbits} $\cO_i$ are homoclinically related, for every $1\leq i,j\leq N$ there are $t_{ij}\in M$, $x_k\in\cO_k$, $0\leq a_{ij}<|\cO_j|$ such that
$d(f^{-n}(t_{ij}), f^{-n}(x_i))\xrightarrow[n\to\infty]{}0$, $d(f^{n}(t_{ij}), {f^{n+a_{ij}}(x_j)})\xrightarrow[n\to\infty]{}0$ exponentially fast.
 When $i=j$, let $t_{ii}:=x_i$.

 Let
 $
 L:=\bigcup_{i,j=1}^N \cO(t_{ij}).
 $
This set is compact, invariant, and uniformly $\chi'$-hyperbolic for some $\chi'>\chi$.
By the shadowing lemma and expansivity, there are $\epsilon,\delta>0$ so small that:
\begin{enumerate}[(a)]
\item Every $\epsilon$-pseudo-orbit in $L^\Z$ is $\delta$-shadowed by at least one orbit \cite[Theorem 18.1.2]{Katok-Hasselblatt-Book};
\item Every $\epsilon$-pseudo-orbit in $L^\Z$ is $2\delta$-shadowed by at most one orbit, see \cite[Theorem 18.1.3]{Katok-Hasselblatt-Book} (in particular the orbit in (a) is unique);
\item Let $U$ be a $\delta$-neighborhood of $L$, then $\bigcap_{n\in\Z}f^n(U)$ is  uniformly $\chi$-hyper\-bolic \cite[Prop. 6.4.6 and its proof]{Katok-Hasselblatt-Book}.
\end{enumerate}

Let $L_m:=\bigcup_{i,j=1}^N\{f^k(t_{ij}):-m\leq k\leq m{+a_{ij}}-1\}$. Choosing $m$ large enough guarantees that $L_m$  contains $\cO_1,\ldots,\cO_N$, $d(f^m(t_{ij}),f^{m{+a_{ij}}}(t_{jj}))<\epsilon/2$, and
$d(f^{-m}(t_{ij}),f^{-m}(t_{ii}))<\epsilon/2$.  Let
$$
K:=\{x:\text{the orbit of $x$ is $\delta$-shadowed by an $\epsilon$-pseudo-orbit in $L_m^\Z$}\}.
$$
This set contains $\cO_1,\ldots,\cO_N$. It is also invariant, uniformly $\chi$-hyperbolic, locally maximal, and (since $L_m$ is finite) closed.

We claim that $K$ is transitive. Let $\Sigma$ be the finite state Markov shift associated with the graph with set of vertices  $L_m$ and edges
$\xi\to\eta$ when $d(f(\xi),\eta)<\epsilon$. Let $\pi:\Sigma\to K$ denote the map which sends a pseudo-orbit to the unique orbit it shadows.
The uniqueness of the shadowed orbit implies that $\pi\circ\sigma=f\circ\pi$ and that  $\pi$ is continuous, see e.g. \cite[Lemma 3.13]{Bowen-LNM}. Thus to show that $K$ is transitive it is enough to show that $\Sigma$ is topologically transitive, or equivalently that $\Sigma$ is irreducible. Any two vertices in  $V_{ii}:=\{f^k(t_{ii}):k\in\Z\}$ can be connected by a path, because $t_{ii}$ is periodic.
The paths $f^{-m}(t_{ii})\to f^{-m+1}(t_{ij})\to f^{-m+2}(t_{ij})\to \cdots\to f^{m-1}(t_{ij})\to f^{m+a_{ij}}(t_{jj})$ are admissible by choice of $m$. So for every $\xi\in V_{ij}:=\{f^k(t_{ij}):-m\leq k\leq m{+a_{ij}}-1\}$ there is a path which starts in $V_{ii}$, passes through $\xi$, and terminates in $V_{jj}$. This implies that any two $\xi,\eta\in L_m=\bigcup_{i,j} V_{ij}$ can be connected by a path, whence  $\Sigma$ is transitive.
\end{proof}

\begin{lemma}\label{l.lift-periodic}
Let $\{\cO_i\}$ be {the set of all} $\chi$-hyperbolic periodic orbits which are homoclinically related to $\mu$. Then there is
 an irreducible component $\Sigma\subset \wh \Sigma$ such that every $\cO_i$ lifts to a periodic orbit in $\Sigma$.
\end{lemma}
\begin{proof}
If $\#\{\cO_i\}\leq 1$ the lemma is trivial, so assume that $\#\{\cO_i\}\geq 2$. In this case $\cO_1$ has a transverse homoclinic intersection, and there are infinitely many $\chi$-hyperbolic periodic orbits which are homoclinic to $\mu$.  By \cite[Prop. 1.1.4]{Katok-Hasselblatt-Book}), $\{\cO_i\}$  is countable. Let
$\cO_1,\cO_2,\dots$ be an enumeration.

For each $n\geq 1$, there exists a uniformly $\chi$-hyperbolic basic set $K_n$ which contains $\cO_1,\dots,\cO_n$
(Lemma ~\ref{Lemma-K_n}).
By ($\widehat{\rm C}$5), there exists an invariant compact transitive set $X_n\subset \wh \Sigma$
which lifts $K_n$. Since $X_n$ is compact and invariant, there is a finite set of vertices $A$ such that $X_n\subset A^\Z$.
So $X_n\subset \wh \Sigma^\#$.

Let us consider a periodic point $x^i\in \cO_i\subset \NUH_\chi^\#$,
with period $\tau_i$. Let $\underline v^i$ be a lift of $x^i$ in $X_n\subset\wh\Sigma^\#$. For every $k$, $\sigma^{k\tau_i}(\un{v}^i)\in \Sigma^\#$.
Property (C1) says that $x^i$ has only finitely many $\wh{\pi}$-lifts in $\wh{\Sigma}^\#$. Hence $\underline v^i\in X_n$
is periodic. Thus, $\cO_1,\ldots,\cO_n$ lift to periodic orbits in $X_n$.

Each of the sets $X_1,X_2,\dots$ contains some lift of the orbit $\cO_1$. Since $\cO_1$ has only finitely many
lifts in $\wh \Sigma^\#$, there exists an infinite subsequence $X_{n_i}$ of the previous sets which contain the same periodic lift $\un{w}^1$ of $\cO_1$. Let $\Sigma$ denote the irreducible component of $\wh{\Sigma}$ which contains $\un{w}^1$. This is the  maximal closed, invariant, and  transitive set containing this element.
Necessarily,  $X_{n_i}\subset \Sigma$ for all $i$. So $\Sigma$ contains periodic lifts of every $\cO_i$, $i\geq1$.
\end{proof}

\begin{lemma}
Every $\chi$-hyperbolic ergodic invariant measure $\nu$ of saddle type that is homoclinically related to $\mu$ lifts to an ergodic shift invariant measure on the irreducible $\Sigma$ given by Lemma \ref{l.lift-periodic}.
\end{lemma}

\begin{proof}
Let $\nu$ be as above.
By ($\rm\widehat{C}$2), there exists an ergodic measure $\bar \nu$ supported on $\wh \Sigma$
which lifts $\nu$.
Suppose  $\underline v\in \wh \Sigma^\#$ is recurrent and generic for $\bar \nu$, i.e.  $(1/n)\sum_{i=0}^{n-1}\delta_{\sigma^i(\un{v})}\xrightarrow[n\to\infty]{w^\ast}\bar\nu$. We set $x:=\wh \pi(\underline v)$.

By recurrence, there exist sequences $m_i,n_i\to +\infty$ such that $v_{n_i}=v_{-m_i}=v_0$.
Let $\un{q}^i$ denote the periodic sequence with period $m_i+n_i$ such that  $(q^i_k)_{k=-m_i}^{n_i-1}=(v_{-m_i},\ldots,v_0,\ldots, v_{n_i-1})$.
Notice that
 $\underline q^i$ is in $\wh \Sigma^\#$, in the same irreducible component as $\bar \nu$.
Moreover, the invariant probability measure supported by the orbit of $\underline q^i$
converges as $i\to \infty$ to $\bar \nu$.

Let $x^i=\wh \pi(\underline q^i)$.
From (C6),
$x^i$ is a hyperbolic periodic point whose orbit is homoclinically related to $\mu$.
Since the invariant probability measure supported by the orbit of $\underline q^i$
converges to $\bar \nu$, (C7) shows that the corresponding measures are also $\chi$-hyperbolic for all $i$ large enough.
Therefore Lemma~\ref{l.lift-periodic} gives an irreducible component $\Sigma$, depending only on $\mu$,
containing some periodic lifts $p^i$ of all $x^i$ with $i$ large enough.

The sequence $(\underline q^i)$
is relatively compact in $\Sigma$, because $q^i_0=v_0$ for all $i\geq 0$ and ${\Sigma}$ is associated to a graph all of whose vertices have finite degrees ($\wh{\Sigma}$ is locally compact).
Since $\wh{\pi}(\un{q}^i)=\wh{\pi}(\un{p}^i)=x^i$ and $\un{p}^i,\un{q}^i\in\Sigma^\#$, we have by (C5) and Proposition~\ref{p.properness} that $\{\un{p}^i\}$ is also relatively compact in $\Sigma$.
Let $\un{p}\in\Sigma$ be the limit of some convergent sub-sequence $\{\un{p}^{i_k}\}$.
By continuity of the projection, $\wh \pi(\un{p})=x$.

We claim that $\un{p}\in\Sigma^\#$. By construction, $q^i_{n_j}=q^i_{-m_j}=v_0$ for $j=1,\ldots,i$. By the Bowen property (C5),
$p^i_{n_j},p^i_{-m_j}\in\{a: a\sim v_0\}$ for $j=1,\ldots,i$. This property is inherited by all limits of subsequences of $\{\un{p}^i\}$, and so
$p_k\in \{a: a\sim v_0\}$ for infinitely many positive and negative $k$. Since  $\{a: a\sim v_0\}$ is finite, $\un{p}\in\Sigma^\#$.

We just showed that any point $x\in E:=\wh{\pi}\{\text{generic points for $\bar\nu$}\}$ lifts to $\Sigma^\#$. Since $\nu(E^c)=0$, $\nu$-almost every $x$ has a lift to $\Sigma^\#$. The number of such lifts is finite by (C1). Now set
$$\bar \mu(E)=\int_M\left( \frac{1}{|\pi^{-1}(x)\cap \Sigma^\#|}\sum_{\underline v\in \pi^{-1}(x)\cap \Sigma^\#} {\mathbf 1}_E(\underline v)\right)d\mu(x).$$
This is an invariant probability measure on  $\Sigma^\#$, and almost every ergodic component of $\bar\mu$ is an ergodic lift of $\nu$ to $\Sigma^\#$ (see e.g. \cite[Proposition 13.2]{Sarig-JAMS}).
\end{proof}

Property (C2.b) is also true, because for every ergodic measure $\bar \nu$ supported on the irreducible component $\Sigma$ given by Lemma~\ref{l.lift-periodic}, the projection $\nu:=\widehat \pi_*(\bar \nu)$ is obviously ergodic and homoclinically related to $\mu$ by property (C6). Thus $\Sigma$ satisfies property (C2).

The proof of Proposition~\ref{SymbolicDynamics} is now complete.
\end{proof}

We deduce the following stronger version of Theorem~\ref{Theorem-Symbolic-Dynamics-C^r}.

\begin{thm}\label{Theorem-Symbolic-Dynamics-C^r2}
Let $r>1$ and $f$ be a $C^r$ diffeomorphism on a closed surface $M$. Suppose $\mu$ be an ergodic hyperbolic measure for $f$.
For every $\chi>0$ there are  a {locally compact} countable state Markov shift $\Sigma$ and {a H\"older-continuous map} $\pi:\Sigma\to M$ such that  ${\pi}\circ\sigma=f\circ{\pi}$ and which satisfies
 properties (C0)-(C8), as well as
\begin{enumerate}[{\it (f)}]
\item[\rm (C9)]
For any transitive compact $\chi$-hyperbolic set $K\subset M$ that is \emph{homoclinically related to $\mu$},
there exists a transitive invariant compact set $X\subset \Sigma$
such that $\pi(X)=K$.
\end{enumerate}
\end{thm}
\begin{proof}
Let us consider a locally compact countable state Markov shift $\wh{\Sigma}$ and a map
$\widehat \pi:\wh\Sigma\to M$ given by Theorem~\ref{thm.sarig}. It satisfies properties (C0), (C1), ($\rm\wh C 2$), (C3), (C4), (C5) and, by Propositions \ref{p-proj-measure}, \ref{p-contExp}, \ref{p.properness}, \ref{p.uniform}, also  properties  (C6), (C7), (C8),($\widehat{\rm C}$9).

Let us consider an irreducible component $\Si\subset \widehat \Si$ given by Proposition~\ref{SymbolicDynamics}
and the restriction $\pi$ of $\widehat \pi$ to $\Si$.
In particular (C2) holds for $\Sigma$ and $\pi$.
Item~(C0) of Theorem \ref{Theorem-Symbolic-Dynamics-C^r} and the local compactness hold because $\Sigma$ is an irreducible component of the locally compact Markov shift $\wh{\Sigma}$. Items (C4)-(C8) and the H\"older-continuity of $\pi$ are immediate, since  $\pi$ is the restriction of $\wh{\pi}$.

It  remains to prove (C9).
If $K$ is a transitive $\chi$-hyperbolic compact set homoclinically related to $\mu$,
one considers as in the proof of Proposition \ref{p.uniform} an invariant ergodic measure $\nu$ whose support is a transitive invariant $\chi$-hyperbolic set $K'\supseteq K$. Using the continuity of  $W^u_{loc}(\cdot)$, $W^s_{loc}(\cdot)$ on $K'$ \cite{Hirsch-Pugh-Stable-Manifold-Thm}, it is not difficult to verify that $\nu\hsim\mu$. So $\nu$ lifts to some $\widehat\nu$ on $\Sigma$. We claim that $\supp\widehat\nu$ is compact. It will follow that $K\subset\pi(\supp\widehat\nu)$ and we will conclude as in the proof of Proposition~\ref{p.uniform}.

We turn to the claim. ($\widehat{\rm C}$9) gives a compact set $Y\subset\widehat\Sigma$ such that $\widehat\pi(Y)=\supp\nu$. In particular $Y\subset A^\Z$ for some finite set $A$ and  $Y\subset\widehat\Sigma^\#$. The Bowen property (C5) for $\wh \pi|_{\wh\Sigma^\#}$
implies that $(\widehat\pi|_{\widehat\Sigma^\#})^{-1}(\supp\nu)\subset B^\Z$ where $B:=\{b:\exists a\in A$ $b\sim a\}$. The set $B$ is finite as the Bowen relation is locally finite. Since $\pi$ is a restriction of $\widehat\pi$, $\nu$-a.e. point belongs to $\Sigma\cap B^\Z$. Therefore, $\supp\widehat\nu\subset\Sigma\cap B^\Z$ is compact as claimed.
\end{proof}

\subsection{Properties of equilibrium measures}\label{equilibrium1}
Let $f$ be a $C^r$ diffeomorphism, $r>1$, of a closed surface $M$. Let $\phi:M\to\RR\cup\{-\infty\}$ be an admissible potential.
We now prove a slightly strengthened version of Corollary~\ref{c-local-uniqueness}. To be precise,  given an invariant measure $\nu$, let:
 $
    P_\phi(f,\nu):=h(f,\nu)+\int \phi\, d\nu.
 $
We consider an ergodic hyperbolic measure $\mu$ which is an equilibrium for $\phi$ \emph{in its homoclinic class}, i.e.,
 \begin{equation}\label{eq-homo-eq}
    \forall\nu\in\Prob(f)\quad \nu\hsim\mu\implies P_\phi(f,\nu) \leq P_\phi(f,\mu).
 \end{equation}
This is weaker than being an equilibrium measure.

We choose $\chi>0$ small enough so that $\mu$ is
$\chi$-hyperbolic. Theorem \ref{Theorem-Symbolic-Dynamics-C^r} applied to
the set of measures homoclinically related to $\mu$ and the threshold $\chi$ provides $\Sigma$ transitive and $\pi:\Sigma\to M$ with an invariant measure $\widehat \mu$
for $\Sigma$ such that $\pi_*(\widehat \mu)=\mu$. Theorem~\ref{t.katok} gives some hyperbolic periodic orbit $\cO$ related to $\mu$ that is $\chi$-hyperbolic. It has a periodic lift $\widehat{\cO}\subset\Sigma^\#$.

Note that $\widehat\mu$ is an equilibrium measure on $(\Sigma,\sigma)$ for the potential $\phi\circ\pi$. Indeed, for any $\widehat \nu\in\Proberg(\sigma)$, $\pi_*(\widehat \nu)$ is an ergodic, hyperbolic measure and $\pi_*\widehat \nu\hsim\mu$ by (C6) so that $P_{\phi\circ\pi}(\sigma,\widehat \nu)=P_{\phi}(f,\pi_*\widehat \nu)\leq P_\phi(f,\mu)$ (the equality is because finite-to-one factors preserve entropy, the inequality is because of eq.~\eqref{eq-homo-eq}). Moreover,  as the potential $\phi$ is admissible, $\phi\circ\pi:\Sigma\to\RR$ is H\"older-continuous
(see Remark~\ref{r.potential}).

\paragraph{\sc Uniqueness.}
{\emph{Any ergodic equilibrium measure $\nu$ for $\phi$ in the homoclinic class of $\mu$ coincides with $\mu$.}}

One takes $\chi>0$ small enough so that $\nu$ is $\chi$-hyperbolic. By {(C2.a)}, there is an invariant measure $\widehat \nu$
for $\Sigma$ such that $\pi_*(\widehat \nu)=\nu$. By the above argument, $\wh\mu$ and $\wh\nu$ are equilibrium measures for $\phi\circ\pi$. As $\Sigma$ is transitive and $\phi\circ\pi$ is H\"older-continuous, \cite{Buzzi-Sarig} implies that $\widehat \mu=\widehat\nu$,
hence $\mu=\nu$, proving the first property of Corollary~\ref{c-local-uniqueness}.

\paragraph{\sc Support.}
\emph{The support of $\mu$ is $\HC(\mu)$.}

Since $\widehat\mu$ is an equilibrium measure for a H\"older-continuous potential, and $\Sigma$ is  transitive, $\widehat\mu$ has full support in $\Sigma$
\cite{Sarig-thermodynamics,Buzzi-Sarig}.  So $\ov{\pi(\Sigma)}=\supp(\mu)\subseteq\HC(\mu)$.
We claim that  $\pi(\Sigma)$ is dense in $\HC(\mu)$.

It is enough to prove that  the $\chi$-hyperbolic periodic points which are homoclinically related to $\mu$
are dense in $\HC(\mu)$, because by (C2.a), they all belong to $\pi(\Sigma)$.

Firstly, the union of all hyperbolic periodic orbits $\cO'\hsim\cO$ is dense in $\HC(\cO)=\HC(\mu)$. Secondly any point in $\cO'$
can be approximated by a hyperbolic periodic point $x''$ homoclinically related to $\cO$ whose orbit induces a probability measure arbitrary
close to the invariant probability supported on $\cO$, for the weak-$*$ topology on $\Proberg(\wh\Sigma)$
(take $x''$ to be the initial condition of a periodic orbit which shadows $\ell$-loops of $\cO'$ and then $m$-loops of $\cO$,
and such that $\ell/m$ is large).
The property (C7) ensures that the orbit of $x''$ is $\chi$-hyperbolic, as $\cO$.

\paragraph{\sc Bernoulli property and period.}
{\emph{The measure $\mu$ is isomorphic to the product of a Bernoulli scheme with a cyclic permutation of order
$\gcd\{\operatorname{Card}(\cO)\; :\; \cO\hsim \mu\}$.}}

As shown in \cite{Sarig-Bernoulli-JMD},  Ornstein theory implies that the equilibrium measure $\widehat \mu$ for $(\sigma,\Sigma,\phi\circ\pi)$ is isomorphic to the product of a Bernoulli scheme and a cyclic permutation, and this property is inherited by the factor $\mu$.
It remains to identify the order of the permutation.

Let us decompose $\mu$: there exist disjoint measurable sets $A_1,\dots,A_\ell$ such that $f(A_i)=A_{i+1}$ when $i<\ell$, $f(A_\ell)=A_1$,
$\mu(A_i)>0$ and $\mu_i:=\mu(\ \cdot\ |A_i)\equiv \ell\mu|_{A_i}$ is Bernoulli for $f^\ell$.
Note that each $\mu_i$ is an equilibrium measure for $f^\ell$ and the potential $\phi_\ell=\frac 1 \ell (\phi+\phi\circ f+\dots+ \phi\circ f^{\ell-1})$.
From the above uniqueness (applied to $f^\ell$), the $\mu_i$ are not homoclinically related.

Let $\cO'$ be some hyperbolic periodic orbit homoclinically related to $\mu$,
it decomposes into disjoint periodic orbits $\cO'_1,\dots,\cO'_k$ for $f^\ell$ with equal cardinality.
Since $\mu$ is homoclinically related to $\cO'$, each $\cO'_j$ is homoclinically related to one and only one $\mu_i$ as these measures are not related.
 The number of orbits $O_j$ related to $\mu_i$, does not depend on $i$
(by invariance of the dynamics), hence the period $\ell$ of $\mu$ divides the number of orbits $\cO'_j$,
and then the period of $\cO'$. This proves that the period of $\mu$ divides $\gcd\{\operatorname{Card}(\cO')\; :\; \cO'\hsim \mu\}$,
the period of the homoclinic class of $\cO$.

Conversely, Proposition~\ref{p.period} proves that the period of the homoclinic class divides the period of the measure $\mu$.
Hence, the period of $\mu$ is equal to $\gcd\{\operatorname{Card}(\cO')\; :\; \cO'\hsim \mu\}$.

The proof of Corollary~\ref{c-local-uniqueness} is now complete. \qed

\subsection{Injective coding on a large set}\label{ss-injective-coding}
The coding $\pi:\Sigma\to M$ obtained in Theorem~\ref{Theorem-Symbolic-Dynamics-C^r} is finite-to-one on its regular part $\Sigma^\#$. We present two combinatorial constructions that use the Bowen property (C5) to create large injectivity sets without destroying the irreducibility property (C0).

{\paragraph{\sc Injective coding of a given measure.}
When an ergodic measure $\mu$ is given,
the constructions in \cite{Boyle-Buzzi} yield irreducible, finite-to-one, and H\"older-continuous codings that can be made $\mu$-almost everywhere injective.
We have the following variant of \cite[Prop. 6.3]{Boyle-Buzzi}.

\begin{thm}\label{thm-almostinjective1}
Let $f$ be a $C^r$ diffeomorphism, $r>1$, on a closed surface $M$. Let $\mu$ be an ergodic hyperbolic measure for $f$. Then there are $\chi>0$, a locally compact countable state Markov shift $\Sigma$ and  a H\"older-continuous map $\pi:\Sigma\to M$
such that  ${\pi}\circ\sigma=f\circ{\pi}$, (C0), (C1), (C2.b), (C3)-(C8), and:
 \begin{equation}\label{eq-aeinj}
    \mu(\{x\in M:|\pi^{-1}(x)\cap\Sigma^\#|=1\}) = 1.
 \end{equation}

 In particular, there is a unique invariant probability measure $\nu$ on $\Sigma$ such that $\pi_*(\nu)=\mu$; moreover $\pi:(\nu,\sigma)\to(\mu,f)$ is a measure-preserving conjugacy.
\end{thm}

Contrary to Theorem~\ref{Theorem-Symbolic-Dynamics-C^r}, this construction  does not ensure (C2.a), i.e., $\nu(\pi(\Sigma^\#))=1$ for every $\chi$-hyperbolic measure $\nu$ homoclinically related to $\mu$ or (C9), i.e., the lifting of any transitive  $\chi$-hyperbolic compact set homoclinically related to $\mu$.

\begin{proof}
Let $\mu$ be an ergodic and hyperbolic measure for $f$. It is $\chi$-hyperbolic for some $\chi>0$. Let $\wh\Sigma$ and $\wh\pi$ be given by Theorem~\ref{thm.sarig} for that parameter $\chi$. Observe that $\wh\pi:\wh\Sigma^\#\to M$ is continuous and satisfies (C1) and the locally finite Bowen property (C5): $\wh\pi$ is \emph{excellent} with a \emph{multiplicity bound} in the terminology of \cite{Buzzi-BFCD}. Since $\mu(\wh\pi(\wh\Sigma^\#))=1$, we can apply Theorem 5.2 of that paper and get an irreducible Markov shift $\Sigma$ and a map $\pi:\Sigma\to M$ such that  $\Sigma$ is  locally compact (because $\wh\Sigma$ was) and $\pi$ is H\"older-continuous (because $\wh\pi$ was and $\pi=\wh\pi\circ q$ for some $1$-Lipschitz map $q:\Sigma\to\wh\Sigma$). One easily checks the remaining claims:
\smallbreak

 -- (C0): $\Sigma$ is irreducible;

 -- (C1),(C5):  from item (1) of \cite[Theorem 5.2]{Buzzi-BFCD};

 -- (C2.b),(C3),(C4),(C6),(C7):  from the same properties of $\wh\pi$, since $\pi=\wh\pi\circ q$ with $q$ continuous;

 -- (C8): from item (4) of \cite[Theorem 5.2]{Buzzi-BFCD};

 -- Equation~\eqref{eq-aeinj}: from item (5) of \cite[Theorem 5.2(7)]{Buzzi-BFCD}. This implies the last part of the theorem: existence of a unique lift $\nu$ of $\mu$ and $\pi:(\Sigma,\nu)\to(M,\mu)$ is a measure-preserving conjugacy.
\end{proof}

\paragraph{Injective coding for all equilibrium measures.}
Using the magic word theory of \cite{Buzzi-BFCD}, one gets injectivity on a larger set:
\newcommand\tak{\widetilde}

\begin{thm}\label{thm-ae-inj}
Let $f$ be a $C^r$ diffeomorphism, $r>1$, on a closed surface $M$. Let $\mu$ be an ergodic hyperbolic measure for $f$. For every $\chi>0$, there are a locally compact countable state Markov shift $\tak\Sigma$ and a H\"older-continuous map $\tak\pi:\tak\Sigma\to M$ such that ${\tak \pi}\circ\sigma=f\circ{\tak \pi}$, which satisfies (C0)-(C9) and
the following additional property.
\smallskip

\noindent
(C10)\qquad There is an open set $\emptyset\ne U\subset\tak\Sigma$
such that for any ergodic $\nu$ on $M$
with $\nu(\tak \pi(U\cap\tak\Sigma^\#))>0$,
$$\nu(\{x\in M:|\tak\pi^{-1}(x)\cap\tak\Sigma^\#|=1\})=1.$$
\end{thm}
\begin{proof}
From Theorem~\ref{Theorem-Symbolic-Dynamics-C^r}, there are a Markov shift $\Sigma$ and a map $\pi:\Sigma\to M$ satisfying $\pi\circ\sigma=f\circ\pi$ and all the properties (C0)-(C9). We can apply Theorem 5.3 of \cite{Buzzi-BFCD} since $\pi|_{\Sigma^\#}$ is Borel, finite-to-one and satisfies the locally finite Bowen property. We get
a new Markov shift $\tak \Sigma$ (locally compact by item (1) of that theorem) and a H\"older-continuous map $q:\tak\Sigma\to\Sigma$ such that $\tak\pi:\tak\Sigma\to M$ defined by $\tak\pi:=\pi\circ q$ is a H\"older-continuous map satisfying $f\circ\tak\pi=\tak\pi\circ\sigma$. One checks (C0)-(C8):
 \smallbreak

 -- (C0): by item (7) since $\Sigma$ is irreducible;

 -- (C1), (C5): by items (1) since $\Sigma$ satisfies a multiplicity bound in the terminology of \cite{Buzzi-BFCD};

 -- (C2.a): by items (1) and (5);

 -- (C2.b), (C3), (C4), (C6), (C7): from the same properties of $\wh\pi$, since $\pi=\wh\pi\circ q$ with $q$ continuous;

 -- (C8): because $\pi$ satisfies the same property and $q:\tak\Sigma^\#\to\Sigma^\#$ is a proper map.
\smallskip

We then prove (C9). As  in Section~\ref{s-p-uniform}, it is enough to show that any  $\chi$-hyperbolic transitive hyperbolic set  $K$ which is the support of some ergodic $\chi$-hyperbolic measure $\nu\hsim\mu$ is contained in $\tak\pi(\tak K)$ for some compact subset $\tak K$ of $\tak\Sigma$. It follows from (C9) for $\pi$ that there is a compact set $X\subset\Sigma$ such that $K=\pi(X)$. By (C8) for $\pi$,  $\pi^{-1}(K)\cap\Sigma^\#$ is also compact.  By item (4) of  Theorem 5.3 of \cite{Buzzi-BFCD}, $\tak K:=q^{-1}(\pi^{-1}(K)\cap\Sigma^\#)\cap \tak\Sigma^\#$ must also be compact. Note that $\tak K = \tak\pi^{-1}(K)\cap\tak\Sigma^\#$ since $q(\tak X^\#)\subset X^\#$. This implies that $\tak\pi(\tak K)\supset\pi(\tak\Sigma^\#)\cap K$. The latter set $\pi(\tak\Sigma^\#)\cap K$ has full $\nu$-measure by (C2.a). Since $\tak\pi(\tak K)$ is compact, it must contain $K$ since it is the support of $\nu$.
\smallskip

We finally prove (C10) for $U:=[w]$ the cylinder in $\tak\Sigma$ defined by the word $w$  from \cite[Theorem 5.3(6)]{Buzzi-BFCD}.  Let $\nu\in\Proberg(f)$ with $\nu(\pi(U\cap\tak\Sigma^\#))>0$. Let $\tak\Sigma_1$ be the set of sequences $x\in\tak\Sigma$ such that either $w$ does not occur, or it occurs infinitely many times in the past and infinitely many times in the future. By Poincar\'e recurrence, $\tak\Sigma_1$ has full measure for all invariant probability measures of $\tak\Sigma$. Since $\tak\pi|_{\tak\Sigma^\#}$ is finite-to-one, this implies that $\nu(\tak\pi(\tak\Sigma^\#\setminus\tak\Sigma_1))=0$ so $\nu(\tak\pi(U\cap\tak\Sigma_1))>0$. Therefore for $\nu$-a.e. $y\in M$, $y=\pi(\sigma^n(x))$ with  $n\in\Z$ and $x\in U\cap\tak\Sigma_1$. By the choice of $U$ and $\tak\Sigma_1$,   $w$ occurs in $x$ infinitely often in the past and also in the future. (C10) now follows from \cite[Theorem 5.3(6)]{Buzzi-BFCD}.
\end{proof}

Recall the notion of equilibrium in a homoclinic class, {see Section~\eqref{equilibrium1}.}

\begin{corollary}
Let $f$ be a $C^r$ diffeomorphism, $r>1$, of a closed surface $M$, let $\chi>0$, let $\mu$ be an ergodic $\chi$-hyperbolic measure
and consider a locally compact countable state Markov shift $\Sigma$ and a continuous map ${\pi}:{\Sigma}\to M$ satisfying ${\pi}\circ\sigma=f\circ{\pi}$ and properties (C0), {(C2.a)}, (C10).

Let $\phi:M\to\RR\cup\{-\infty\}$ be an admissible potential and assume that there exists a $\chi$-hyperbolic measure $\nu$
which is an equilibrium measure for $\phi$ in the homoclinic class of $\mu$.

Then $\Sigma$ admits an equilibrium measure $\bar \nu$ for the potential $\phi\circ\pi$.
It is unique and $\pi:(\Sigma,\bar \nu)\to (M,\nu)$ is a measure-preserving conjugacy.
\end{corollary}

\begin{proof}[Proof of the Corollary]
As proved in Section~\ref{equilibrium1}, {(C2.a)} implies that $\mu=\pi_*(\bar\mu)$ where $\bar \mu$ is an equilibrium measure for the potential $\phi\circ\pi$. Moreover, there is no other equilibrium for that H\"older-continuous potential
and $\bar \mu$ has full support in the irreducible shift $\Sigma$. In particular $\bar \mu(U)>0$ for $U$ as in (C10).
This implies that $\pi$ is injective on a measurable subset of $\Sigma$ which has full $\bar \mu$-measure.
\end{proof}

\begin{corollary}
Let $f$ be a $C^{r}$ diffeomorphism, $r>1$, on a closed surface $M$,
let $\Sigma$ be a locally compact countable state Markov shift and ${\pi}:{\Sigma}\to M$ be a continuous map satisfying ${\pi}\circ\sigma=f\circ{\pi}$
and (C0), (C6), (C9), (C10) for some $\chi>0$. Let us assume also that $\Sigma$ supports an ergodic measure $\widehat \mu$ whose
projection $\mu:=\pi_*\widehat\mu$ is $\chi$-hyperbolic.

Then the period of $\Sigma$ equals the period $\ell:=\gcd\{\Card\cO:\cO\hsim\mu\}$ of the homoclinic class of $\mu$.
\end{corollary}
\begin{proof}
Let us consider some hyperbolic periodic orbits $\cO_1,\dots,\cO_N$ homoclinically related to $\mu$ and satisfying
$\gcd(|\cO_1|,\dots,|\cO_N|)=\ell$. Let $U\subset \Sigma$ be an open set as in (C10) and let $q_0$ be the projection by $\pi$ of a periodic point in $U$. By (C0) and (C6), the periodic orbit $\cO_0$ of $q_0$ is homoclinically related to $\mu$. By Theorem~\ref{t.katok}, there is a $\chi$-hyperbolic orbit $\cO_*$ with $\cO_*\hsim\mu$.
Lemma~\ref{Lemma-K_n} provides a hyperbolic basic set $K$ containing all the previous periodic orbits.
As in section \ref{equilibrium1} when we analyzed the support of equilibrium measures, using
one can assume that $\cO_0,\cO_1,\dots,\cO_N$ and $K$ are $\chi$-hyperbolic.
By (C9), there exists an invariant compact subset $X\subset \Sigma$ such that
$\pi(X)=K$.

We claim that there is a neighborhood $V$ of $q_0$ such that  $V\cap K\subset \pi(U)$. Otherwise one can consider points $y^n\in X\setminus U$ with $\pi(y^n)\to q_0$. Since $X$ is  a compact subset of $\Sigma$, one can assume that $(y^n)$ converges to
some point $y_*\in X\subset \Sigma^\#$. By continuity, $\pi(y_*)=\pi(q_0)$. Property (C10) then implies that $y_*=q_0$, a contradiction proving the claim.

 Using shadowing, one can replace the orbits $\cO_0,\cO_1,\dots,\cO_N$ by periodic orbits in $K$ that all meet $V$ and whose periods still have greatest common divisor equal to $\ell$. Now these periodic orbits are contained in the projection $\pi(X)$ and in the injectivity set of $\pi$: they lift as periodic orbits in $\Sigma$ having the same periods.
In particular the period of $\Sigma$ divides $\ell$.
By (C6), all the periodic orbits in $\Sigma$ project by $\pi$ on periodic orbits homoclinically related to $\mu$,
hence $\ell$ divides the period of $\Sigma$, so $\ell$ is the period of $\Sigma$.
\end{proof}

%
%
%
%
%
%

\section{Dynamical laminations and Sard's Lemma}\label{Sec.Sard}
\newcommand\Emb{\mathrm{Emb}}

Throughout this section, $M$ is a closed Riemannian surface and $f:M\to M$  is a $C^r$ diffeomorphism, $1<r<\infty$.
Let
$$
\Emb^r((-1,1),M):=\{\vf:(-1,1)\to M:\vf\text{ is a $C^r$ embedding}\}.
$$
We endow it with the \emph{compact-open} $C^r$-topology.

By a $C^1$-curve, we mean the image of the restriction of an element $\vf\in\Emb^1((-1,1),M)$
to the interval $(-1/2,1/2)$. Hence, the compact-open $C^1$-topology induces a \emph{uniform $C^1$-topology} on parametrized curves $\tau=\vf((-1/2,1/2))$. Most of the time, we will abuse the language and speak about $C^1$-close curves or $C^1$-neighborhoods of curves when we are really discussing their parametrizations.

\subsection{Laminations, transversals, holonomies, dimensions}

\noindent
A {\em partition} of a set $K\subset M$ is a collection $\mathfs L$ of non-empty pairwise disjoint subsets of $K$, whose union equals $K$. For every $x\in K$, we let
$$
\mathfs L(x):=\text{the (unique) element of $\mathfs L$ which contains $x$}.
$$
The set $K$ will sometimes be denoted by $K=\supp(\mathfs L)$.

A {\em continuous (one-dimensional) lamination with $C^r$ leaves} is a partition $\mathfs L$ into connected sets such that every point $x_0\in \supp(\mathfs L)$ has an open neighborhood $U$ and a  map $\Theta:U\cap K\to\mathrm{Emb}^r((-1,1),M)$ satisfying  for all $x\in U\cap K$:
\begin{enumerate}[\quad(L1)]
\item  $\Theta(x)(0)=x$;
\item $\mathfs L_{U}(x):=\Theta(x)((-1/2,1/2))$  is the connected component of $U\cap\mathfs L(x)$ which contains $x$;
\item $\Theta$ is continuous on {$U\cap K$} in the {compact-open} $C^r$ topology.
\end{enumerate}
The set $\mathfs L(x)$ is called the (global) {\em leaf of }$x$ and is an injectively embedded $C^r$ submanifold. The set $\mathfs L_{U}(x)$ is the {\em local leaf of $x$ (in $U$)}.
The sets $U$ above are called {\em lamination neighborhoods}.
Note that $\supp(\mathfs L)$ may be covered by a finite or  countable collection of  lamination neighborhoods.

\medskip
\noindent
{\sc Transversals.}
A {\em transversal}
to a continuous lamination $\mathfs L$ is a $C^1$-embedded one-dimensional sub-manifold $\tau\subset M$ which intersects $\supp(\mathfs L)$, such that for every $x\in\tau\cap\supp(\mathfs L)$ and every local leaf $L$ of $x$,
it holds $T_x M=T_x\tau\oplus T_x L$. We write in this case $\tau\pitchfork\mathfs L$.

\medskip
\noindent
{\sc Transverse dimension.}
Let $\dim_H$ denote the Hausdorff dimension.
 The {\em (upper) transverse dimension} of a continuous lamination $\mathfs L$ is
\begin{equation}\label{eq-tHD}
 \dbar(\mathfs L):={\sup}\{\dim_H(\tau\cap\supp(\mathfs L)): \tau\text{ is a transversal  }\}.
\end{equation}

\medskip
\noindent
{\sc Holonomy projections.}
Consider a point $x_0\in \supp(\mathfs L)$ with a lamination neighborhood $U_0$ and a transversal $\tau_0$ containing $x_0$.

Up to reducing $\tau_0$, we may assume that, for every $y\in \supp(\mathfs L)$ close to $x_{0}$, $\tau_0$ intersects the leaf ${\mathfs L}_{U_0}(y)$
in exactly one point.
Then there is an open set $V\owns x_0$
such that, for any transversal $\tau$ that is sufficiently  {uniformly} $C^1$-close to $\tau_0$, the following map $\Pi_\tau: V\cap \supp(\mathfs L)\to \tau$ is well-defined:
$$
{\Pi_{\tau}(y):=\text{unique point in  $\mathfs L_{U_0}(y)\cap \tau$}.}
$$
{This is because of the continuity condition (L3) in the definition of a lamination, and the transversality of the  intersection of $\mathfs L_{U_0}(y)$ and $\tau$ at $x_0$.}

\medskip
\noindent
{\sc Lipschitz holonomies property:}
We say that a lamination $\mathfs L$  has {\em Lipschitz holonomies} if, associated to any $x_0,U_0,\tau_0$ as above,
there exist $L>0$, a {uniformly $C^1$} neighborhood $\mathcal T$ of $\tau_0$
and a neighborhood $V$ of $x_0$ with the following property:
For any transversals $\tau,\tau'\in \mathcal T$,
the holonomy projection map $\Pi_{{\tau\to\tau'}}\colon \tau\cap V\cap \supp(\mathfs L)\to \tau'$
has Lipschitz constant  $\leq L$, where $\Pi_{{\tau\to\tau'}}$ is the restriction of the map $\Pi_{{\tau'}}$ to $\tau$.

\subsection{The stable and unstable laminations of a horseshoe}\label{Section-Dynamical-Laminations}
We will now demonstrate the definitions of the previous section by an important example. Throughout this section, we fix a real number $r>1$, and a $C^r$ diffeomorphism $f$ of a surface $M$. We also consider a basic set $\Lambda$ with stable and unstable dimensions {both equal to $1$.}

{Recall from Section~\ref{ss.hyperbolic} the local stable and unstable  manifolds $W^{s/u}_\varepsilon(x)$, where $\varepsilon>0$ is small enough.}
Let $W^s_\epsilon(\Lambda):=\bigcup_{x\in\Lambda}W^s_\epsilon(x), \quad W^u_\epsilon(\Lambda):=\bigcup_{x\in\Lambda}W^u_\epsilon(x).$
These sets are naturally partitioned: For each $x\in W^s_\epsilon(\Lambda)$, one considers the connected component
of $W^s(x)\cap W^s_\epsilon(\Lambda)$ containing $x$ (for the topology of $W^s(x)$). This defines
two one-dimensional continuous laminations with $C^r$ leaves, called (local) {\em stable} and {\em unstable laminations of $\Lambda$},
and denoted by $\mathfs W^u_\epsilon(\Lambda)$ and  $\mathfs W^s_\epsilon(\Lambda)$.  See~\cite{Hirsch-Pugh-Stable-Manifold-Thm} for details.
Set
$$\lambda^u(f,\Lambda):=\lim\limits_{n\to+\infty}\frac{1}{n}\log(\sup_{x\in\Lambda}\|Df^n_{x}\|) \text{ and }
\lambda^s(f,\Lambda):=\lim\limits_{n\to+\infty}\frac{1}{n}\log(\sup_{x\in\Lambda}\|Df^{-n}_{x}\|).$$

The following statement summarizes well-known results:

\begin{thm}\label{basic}
{Let $\Lambda$ be a basic set of a $C^r$ diffeomorphism $f$ on a closed surface, where $r>1$. Then,} {for small enough $\varepsilon>0$:}
\begin{enumerate}[(1)]
\item\label{basic0} The local manifolds $W^{s/u}_\varepsilon(x)$ are given by $C^r$ embeddings depending continuously on $x\in\Lambda$;
\item\label{basic1} The laminations $\mathfs W^u_\epsilon(\Lambda)$, $\mathfs W^s_\epsilon(\Lambda)$ have Lipschitz holonomies.
\item\label{basic2} There exist $\dbar^u_\Lambda,\dbar^s_\Lambda\in (0,1)$ such that for any $\varepsilon>0$ small enough and any $x\in \Lambda$,
$$\dim_H({W^s_\epsilon(x)}\cap \Lambda)=\dbar^u_\Lambda
\text{ and } \dim_H({W^u_\epsilon(x)}\cap \Lambda)=\dbar^s_\Lambda.$$
\item\label{basic3} $\dbar^u_\Lambda \geq h_{\top}(f|_{\Lambda})/{\lambda^s}(f,\Lambda) \text{ and }
\dbar^s_\Lambda \geq h_{\top}(f|_{\Lambda})/{\lambda^u}(f,\Lambda).$

{Also $\dbar^s_\Lambda \geq h(f,\mu)/{\lambda^u}(f,\mu)$ and $\dbar^u_\Lambda \geq h(f,\mu)/{\lambda^s}(f,\mu)$ for any ergodic measure supported on $\Lambda$.}
\end{enumerate}
\end{thm}

\noindent
\emph{Comments.}
$\dbar^u_\Lambda$ is the transverse dimension to the lamination $\mathfs W^u_\epsilon(\Lambda)$. It should not be confused with the
(non-transverse!) unstable dimension $\delta^u(\mu)$ of a measure $\mu$, introduced in Section~\ref{intro-finite-regularity}.

Part~\eqref{basic0} is standard, including for non-integer $r$, see, e.g., \cite{yoccoz}.
Part~\eqref{basic1} can be obtained for $C^2$ diffeomorphisms, by noticing that $\mathfs W^{s/u}(\Lambda)$ extend to $C^1$-foliations of an open set (see for instance~\cite{Bonatti-Crovisier}).
When  $r\in (1,2)$, one can use the methods of  \cite{Pinto-Rand}. We give a proof in Appendix~\ref{s.holonomy} for completeness.
Part~\eqref{basic2} is proved in~\cite{palis-viana}.
Part~\eqref{basic3} uses~\cite{manning-dimension}:
it asserts that for any basic set $\Lambda$ and any ergodic measure $\mu$ on $\Lambda$, there is an inequality:
 $$
  \dbar^u_\Lambda\geq h(f,\mu)/\lambda^s(f,\mu)=\delta^s(\mu).
 $$
One concludes using $\lambda^s(f,\mu)\leq \lambda^s(f,\Lambda)$ and the variational principle $h_\top(f,\Lambda)=\sup_\mu h(f,\mu)$, where $\mu$ ranges over the ergodic measures supported on $\Lambda$.
Note that~\cite{manning-dimension} assumes $f$ to be Axiom A, but does not use it.
{(In case $r\geq 2$, part~\ref{basic3} also follows from~\cite{Ledrappier-Young-II}.)}

{The parts (2), (3), (4) do not generalize in general to higher dimensions.}

\subsection{A ``dynamical" Sard's Theorem}
Fix a real number $r>1$ and let  $\mathfs L$ be a one-dimensional continuous lamination with $C^r$ leaves inside a closed  Riemannian surface $M$.

Suppose  $\gamma:[0,1]\to M$ is a $C^1$-curve such that  $\gamma'(t)\neq 0$ for all $t\in [0,1]$. We are interested in the lamination
$$
\mathfs N_\gamma(\mathfs L):=\{\mathfs L(\gamma(t)):t\in [0,1],\ \gamma(t)\in\supp(\mathfs L)\text{ and }\gamma'(t)\in T_{\gamma(t)}\mathfs L\}
$$
made from all $\mathfs L$--leaves with at least one non-transverse intersection with $\gamma$.
The following version of Sard's Theorem  says that the higher the smoothness of the leaves of $\mathfs L$, the smaller  the transverse dimension of $\mathfs N_\gamma(\mathfs L)$.
Recall the definition of the transverse dimension $\dbar$ from eq.~\eqref{eq-tHD}.

\begin{thm}\label{t.DynSard}
Suppose $\mathfs L$ is a continuous lamination with $C^r$ leaves $(r> 1)$ and  Lipschitz holonomies, inside a compact closed surface.
Then for every $C^r$ curve $\gamma$,
$$\dbar (\mathfs N_\gamma(\mathfs L))\leq 1/r.$$
\end{thm}
\begin{remark}
Theorem \ref{t.DynSard} is an immediate consequence of the classical Sard's theorem \cite[Thm 2]{Sard-1958} whenever the lamination extends to a $C^r$ foliation of a neighborhood of its support.
 But the laminations $\mathfs W^s(\Lambda),\mathfs W^u(\Lambda)$ to which we intend to apply Theorem \ref{t.DynSard} do not in general~\cite{HW} even extend to  a differentiable foliation of an open neighborhood of their support.
\end{remark}
\begin{proof}
It is enough to find a countable cover of $\mathfs N_\gamma(\mathfs L)$ by open sets $W_i$ such that
$
\dbar (\mathfs N_\gamma(\mathfs L,W_i))\leq 1/r
$, where
$$
\mathfs N_\gamma(\mathfs L,W_i):=\{\ell\in\mathfs L:\ell\text{ intersects $\gamma$ non-transversally somewhere in  } W_i\}.
$$
This is because
 $\mathfs N_\gamma(\mathfs L)\subset\bigcup \mathfs N_\gamma(\mathfs L,W_i)$, and the Hausdorff dimension of a countable union is the supremum of the Hausdorff dimensions of its elements.

We cover $\supp(\mathfs L)$ by a countable collection of lamination neighborhoods $U_\alpha$.
For each tangency point $p$, there exist $U_\alpha$ and a $C^r$ chart
$\chi\colon (-1,1)^2\to V$ such that (see Figure \ref{figure-inside-chi}):
\begin{enumerate}[\quad(1)]
\item $V\subset U_\alpha$, $\gamma\cap V=\chi((-1,1)\times\{0\})$ and $p=\chi(0,0)$.
\item For each $t\in (-1,1)$ such that $\chi(t,0)\in \supp(\mathfs L)$,
we denote by $\mathfs L_{V}(t)$ the connected component of $\mathfs L(\chi(t,0))\cap V$
containing $\chi(t,0)$. Its preimage by $\chi$ is
the graph of a $C^r$ function $\varphi_t\colon (-1,1)\to (-1,1)$.
It satisfies $\varphi_t(t)=0$ and by (L3), it varies continuously with $t$ in the $C^r$ topology { on $\{t:\chi(t,0)\in\supp(\mathfs L)\}$.}
\item Let $\tau_s:=\{s\}\times (-1,1)$ for each $s\in (-1,1)$.
For each $s$, let $\pi_{s}$ denote the holonomy
which sends  points $(x,\varphi_t(x))$ (with $\chi(t,0)\in \supp(\mathfs L)$)
to $(s,\varphi_t(s))$.
There exists $L_0>0$ such that for any $s,s'\in (-1,1)$, the restriction of the map $\pi_s$
to $\tau_{s'}$ is $L_0$-Lipschitz.
\end{enumerate}

Let $W=\chi((-1/2,1/2)\times (-1,1))$ . We will show that
$\dbar(\mathfs N_\gamma(\mathfs L,W))\leq 1/r.$
We parametrize the set of leaves in $\mathfs N_\gamma(\mathfs L,W)$ by
$$T:=\{t\in (-1/2,1/2): \chi(t,0)\in \supp(\mathfs L) \text{ and }\vf_t'(t)=0\}$$
 $$
   \text{ and } \mathfs T:=\{(0,\varphi_t(0)): t\in T\}=\pi_0({\chi(T\times\{0\})})
$$
(so that $\mathfs T$ is the locus of tangency and $T$ is the set of their abscissas).
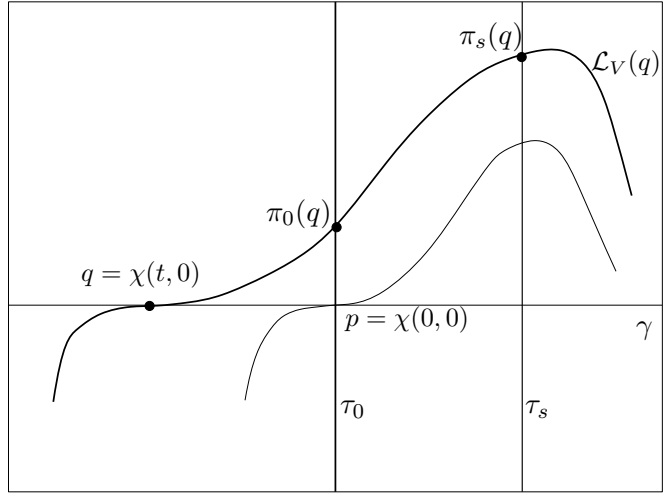
\begin{figure}
\begin{center}
\includegraphics[scale=0.55]{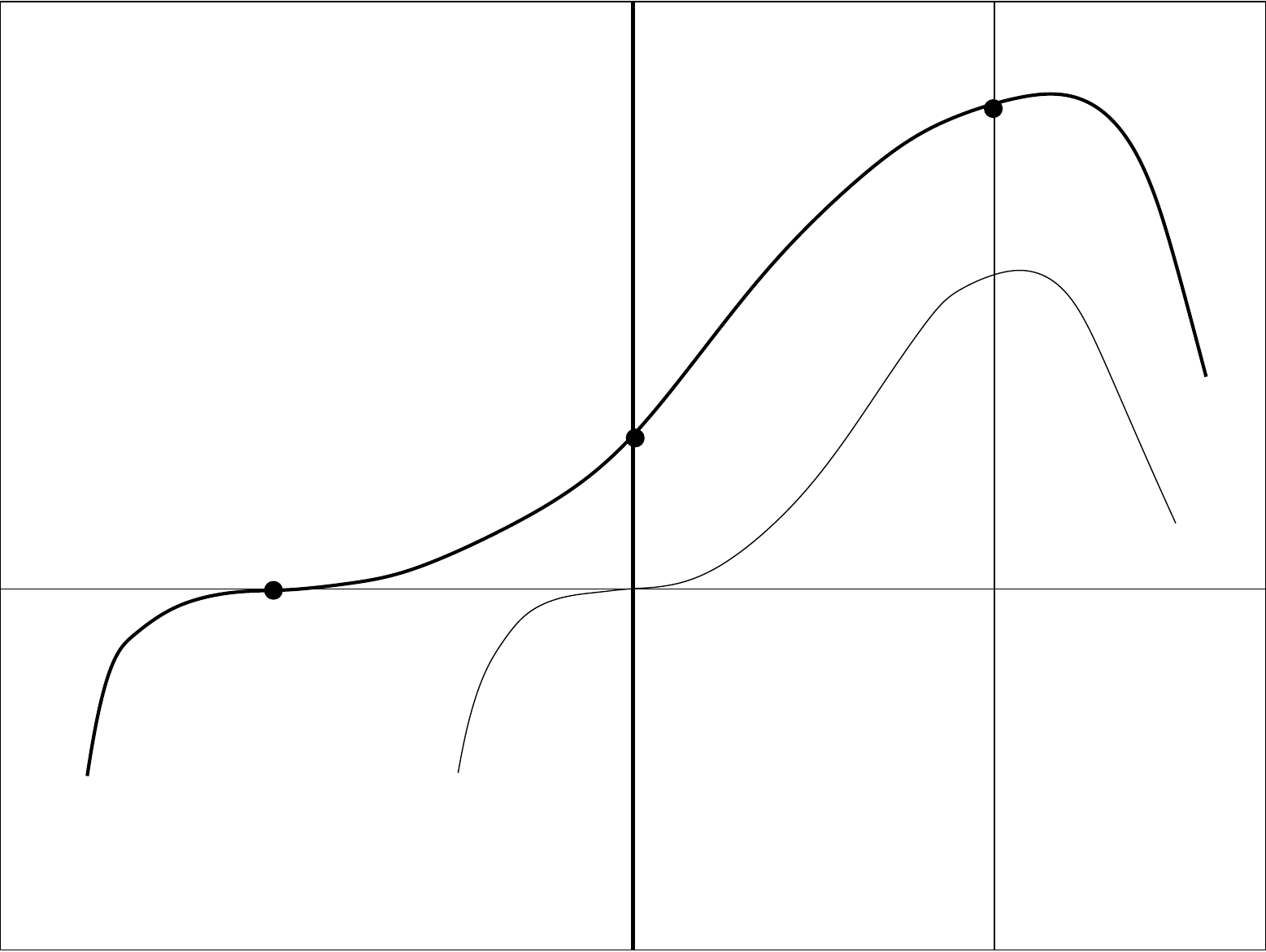}
\put(-10,60){$\gamma$}
\put(-27,160){\small $\mathcal{L}_V(q)$}
\put(-52,30){$\tau_s$}
\put(-122,30){$\tau_0$}
\put(-77,170){$\pi_s(q)$}
\put(-150,103){$\pi_0(q)$}
\put(-220,80){\small $q=\chi(t,0)$}
\put(-120,63){\small $p = \chi(0,0)$}
\caption{Inside the chart $\chi$. A point of tangency ${q}=\chi(t,0)\in\supp(\folL)\cap\gamma$ parameterized by $t\in T$. The holonomy projections $\pi_0(q)=\folL_V(q)\cap\tau_0$ and $\pi_s(q)=\chi(s,\phi_t(s))\in\tau_s$ are also drawn.}\label{figure-inside-chi}
\end{center}
\end{figure}

\bigskip
\noindent
{\sc Claim 1.\/} $\dbar(\mathfs N_\gamma(\mathfs L,W))\leq \dim_H(\mathfs T)$.

\medskip
\noindent
{\em Proof.\/}
Fix some transversal $\tau\pitchfork \mathfs N_\gamma(\mathfs L{,W})$.

 For each $x\in \tau\cap \supp N_\gamma(\mathfs L{,W})$,
 there exists a compact  disc $D\subset \mathfs L(x)$
 which contains $x$ and some point $y=\chi(0,\varphi_t(0))$ in $\mathfs T$.
Using the definition of holonomy projections and the compactness of $D$,
one can construct a sequence of points
$x_0=x, x_1, x_2,\ldots,x_{n-1},x_n=y$
in $D$, open sets $V_i, V_i'$ such that $V_i\owns x_i$, $V_i'\owns x_{i+1}$ and   $V_i'\subset V_{i+1}$,
 transversals
$\tau^0:=\tau, \tau^1:=\tau_{s_1},\ldots,\tau^{n-1}:=\tau_{s_{n-1}},\tau^n:=\tau_0$
passing through $x_i$, and bi-Lipschitz {holonomies $\Pi_{\tau^i\to\tau^{i+1}}$}$:\supp(\mathfs L)\cap \tau^{i}\cap V_i\to\supp(\mathfs L)\cap \tau^{{i+1}}\cap V_{i}'$.

Composing these maps, we obtain an open neighborhood $O$ of $y$,
an {open set } $X\owns x$
and a {bi-Lipschitz holonomy} which associates to each $x'{=\chi(t',\vf_{t'}(t'))\in X\cap \tau\cap \supp N_\gamma(\mathfs L)}$
a point $y'=\chi(0,\varphi_{t'}(0))$ in { $\mathfs T\cap O$}, by holonomy.
Note that {by choosing $X$ sufficiently small one can guarantee that} $x'$ and $y'$ are contained in a compact disc $D'\subset \mathfs L(x')$ which is close to $D$ in the Hausdorff topology.

Since the Hausdorff dimension does not {change after taking the image by a bi-Lipschitz holonomy},
one gets $\dim_H(X\cap\tau\cap \supp N_\gamma(\mathfs L))\leq \dim_H(\mathfs T)$.

 Let $\mathfs C$ denote the space the  triples $(x,D,y)$, endowed with the Hausdorff topology.
Each  triple $(x,D,y)\in \mathfs C$ admits a neighborhood $U$ in $\mathfs C$ such that
$$\dim_H(\{x'{\in\tau}:\; (x',D',y')\in U\})\leq\dim_H(\mathfs T).$$
Since the Hausdorff topology on $\mathfs C$ is separable {and metrizable},
one can cover $\mathfs C$ by countably many such neighborhoods $U_i$.
This gives
$$\dim_H[\tau\cap \supp N_\gamma(\mathfs L{,W})]= \sup_i \dim_H[\{x':\; (x',D',t')\in {U_i}\}]\leq\dim_H(\mathfs T).$$
Passing to the supremum over all  $\tau\pitchfork\mathfs N_\gamma(\mathfs L{,W})$, gives $\dbar(\mathfs N_\gamma(\mathfs L,W))\leq \dim_H(\mathfs T)$.

\medskip
To estimate  $\dim_H(\mathfs T)$, we decompose
 $\mathfs T=\bigcup_{s=1}^{[r]-1} \mathfs T_s\cup\wh{\mathfs T}_{r}$, where (identifying $s$ and $\chi(s,0)$ as convenient),
\begin{align*}
{\mathfs T}_s&:=\pi_0(T_s) \text{ and } \wh {\mathfs T}_{{r}}:=\pi_0(\wh T_r),\\
T_s&:=\{t\in T: \;{\vf_t(t)=}\vf_t'(t)=\cdots=\vf_t^{(s)}(t)=0,\vf_t^{(s+1)}(t)\neq 0\},\\
\wh{T}_r&:=\{t\in T:\;\vf_t(t)=\vf_t'(t)=\cdots=\vf_t^{([r])}(t)=0\}.
\end{align*}
The theorem follows immediately from the following claims.

\bigskip
\noindent
{\sc Claim 2.\/} $\dim_H(\wh{\mathfs T}_r)\leq 1/r$.
\begin{proof}
For $\delta>0$, let $\wh{T}_r=\bigcup_{i=1}^N J_i$ be a {partition} into sets of diameter $<\delta$
with disjoint convex hulls.
In particular $\sum_{i=1}^N \diam(J_i)\leq 2$. For each $i=1,\ldots,N$, one chooses $t_i\in J_i$
and projects $J_i$ {along $\mathfs L$} {to} a subset $\mathfs J_i:=\pi_{t_i}(J_i)$ of the vertical line $\tau_i:=\{t_i\}\times (-1,1)$ . Each $x\in J_i\subset {\wh{T}_r}$, {viewed as the point $(x,0)=(x,\vf_x(x))\in{\chi^{-1}}(\mathfs N_\gamma(\mathfs L))$},   maps to the point
$(t_i,\varphi_x(t_i))$.

The Taylor  formula at $x$ (in the Lagrange form) gives
$$
|\vf_x(t_i)|=|\vf_x(x+(t_i-x))|\leq \frac{|t_i-x|^{[r]}}{[r]!}\biggl(\sup_{|y-x|\leq \delta}|\vf_x^{([r])}(y)|\biggr)\leq C(\diam(J_i))^{r},
$$
where
$$
C:=\begin{cases}
\frac{1}{r!}\sup\limits_{{t\in T}}\|\vf^{(r)}_t\| & \text{ if $r\in \NN$,}\\
\frac{1}{[r]!}\;\sup\limits_{{t\in T}}\;\sup\limits_{-1/2\leq x<y\leq 1/2}\frac{\|\vf^{([r])}_t(x)-\vf^{([r])}_t({y})\|}{{|x-y|}^{r-[r]}} & \text{ otherwise.}
\end{cases}
$$

By (L3),  the map $t\mapsto \varphi_t$ is continuous {on $T$} in the $C^r$ topology, so $C<\infty$.

It follows that $\diam(\mathfs J_i)<2C(\diam(J_i))^r$ and
 $$
    \diam[\pi_0(J_i)]=\diam[\pi_0(\mathfs {J}_i)]\leq 2CL_0(\diam(J_i))^r.
 $$
Thus $\wh{\mathfs T}_r$ is covered by $N$ sets $\pi_0(J_i)$ of diameter smaller than $2CL_0\delta^{r}$
and
 $
    \sum_{i=1}^N [\diam \pi_0(J_i)]^{1/r}\leq (2CL_0)^{1/r}\sum_{i=1}^N \diam(J_i)\leq 2(2CL_0)^{1/r}.
 $
Hence
the $1/r$-dimensional Hausdorff measure of $\wh{\mathfs T}_r$ is finite.
\end{proof}

\bigskip
\noindent
{\sc Claim 3.\/} \emph{For $1\leq s\leq [r]-1$, $\mathfs T_s$ is at most countable,  whence $\dim_H(\mathfs T_s)=0$.}
\begin{proof} It is enough to show that  all accumulation points of  $T_s$ lie outside $T_s$.
So we consider a limit $t_\ast\in [-1/2,1/2]$ of some sequence of points $t_n\in T_s\setminus\{t_\ast\}$ and assume by contradiction that $t_\ast\in T_s$.

Let $\vf_{n}:=\vf_{t_n}$, $\vf_{\ast}:=\vf_{t_\ast}$, $c_n=\vf_{n}^{({s+1})}(t_n)/(s+1)!$ and $c_\ast=\vf_{\ast}^{({s+1})}(t_\ast)/(s+1)!$.
Since $t_\ast, t_n\in {T_s}$, we have $c_n, c_\ast\neq 0$.
Since $\vf_{n}\to\vf_{\ast}$ in the $C^r$ topology, we have $c_n\to c_\ast$.
By Taylor's formula in the Lagrange form, there exists $\varepsilon\colon \RR_+\to \RR_+$ {non-decreasing} such that $\varepsilon(u)\xrightarrow[u\to 0^+]{} 0$, and
 for any $n$ and any $x\in (-1,1)$,
$$
\big|\vf_{n}(x)-c_n (x-t_n)^{{s+1}}\big|\leq \varepsilon(|x-t_n|)\cdot|x-t_n|^{s+1},$$
$$\big|\vf_{\ast}(x)-c_\ast (x-t_\ast)^{{s+1}}\big|\leq \varepsilon(|x-t_\ast|)\cdot|x-t_\ast|^{s+1}.$$

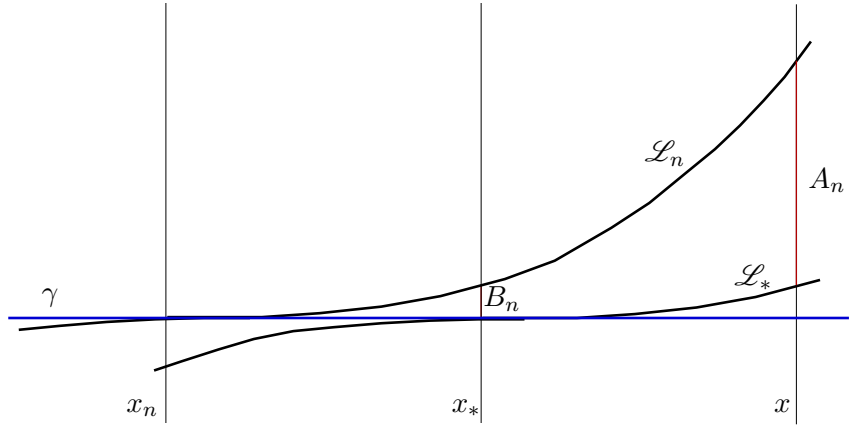
\begin{figure}[ht]
\begin{center}
\includegraphics[scale=0.55]{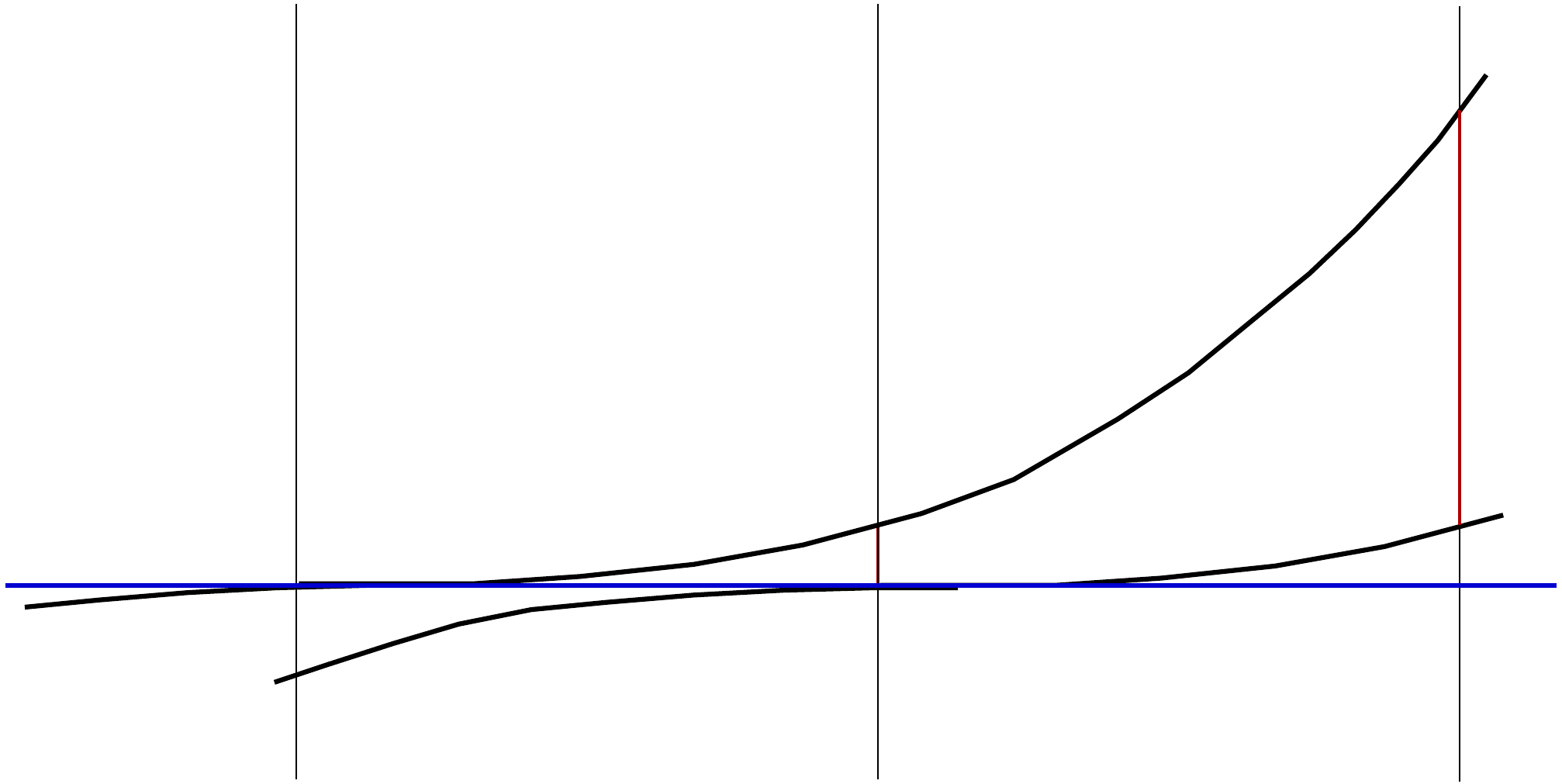}
\begin{picture}(0,0)
\put(-310,47){$\gamma$}
\put(-143,45){$B_n$}
\put(-48,53){$\mathfs L_*$}
\put(-82,100){$\mathfs L_n$}
\put(-33,5){$x$}
\put(-155,5){$x_*$}
\put(-278,5){$x_n$}
\put(-20,90){$A_n$}
\end{picture}
\end{center}
\caption{\label{f.Sard-proof} {Proof that $\mathfs T_s$ is countable (Claim 3).}}
\end{figure}

{Choose $K>2$ so large that $(K+1)^{s+1}-K^{s+1}>L_0$, where $L_0>0$ is the Lipschitz constant of the holonomies.}
For $n\gg1$, let $x:=t_\ast+K(t_\ast-t_n)$ and
\begin{align*}
A_n&:=|\vf_{n}(x)-\vf_{\ast}(x)|
\geq \bigg|c_n(K+1)^{s+1}(t_\ast-t_n)^{s+1}
-c_\ast K^{s+1}(t_\ast-t_n)^{s+1}\bigg|\\
&\hspace{5cm}-2\varepsilon\big(|K+1||t_\ast-t_n|\big).|K+1|^{s+1}|t_\ast-t_n|^{s+1},\\
B_n&:=|\vf_n(t_\ast)-\vf_\ast(t_\ast)|=|\vf_n(t_\ast)|
\leq |c_n (t_\ast-t_n)^{s+1}|+\varepsilon(|t_\ast-t_n|).|t_\ast-t_n|^{s+1}\leq (c_n+o(1))|t_\ast-t_n|^{s+1}.
\end{align*}
See Figure~\ref{f.Sard-proof}.
By definition of $\varepsilon$, this gives
$$\frac {A_n}{B_n}\geq (1-o(1))\left(\big|({K+1})^{s+1}-\frac{c_\ast}{c_n}K^{s+1}\big|-o(1)\right)\xrightarrow[n\to\infty]{}(K+1)^{s+1}-K^{s+1},
$$
so $\liminf A_n/B_n\geq (K+1)^{s+1}-K^{s+1}$.

At the same time,  $(t_\ast,\varphi_n(t_\ast))$ and $(t_\ast,\varphi_\ast(t_\ast))$ project by $\pi_{x}$
to $(x,\varphi_n(x))$ and $(x,\varphi_\ast(x))$. Since $\pi_x$ is $L_0$-Lipschitz,
 $A_n/B_n\leq L_0<(K+1)^{{s+1}}-K^{s+1}$ for all $n$, a contradiction.
\end{proof}

This concludes the proof of Theorem~\ref{t.DynSard}.
\end{proof}

\subsection{Intersection of horseshoes and su-quadrilaterals}

{We will use repeatedly  the following consequence of Theorem~\ref{t.DynSard}. Recall  the Smale pre-order $\preceq$ and $su$-quadrilaterals from Definitions  \ref{Def-H-Eq-Newhouse}~and~\ref{def-quadrilateral}.

\begin{proposition}\label{prop-common-Q}
Let $M$ be a closed surface, $f\in\Diff^r(M)$, $r>1$, and let $Q$ be a $su$-quadrilateral.
Let $x_1$ be hyperbolic periodic point of $f$ such that
$W^u(x_1)$ meets both $Q$ and $M\setminus\overline{Q}$.
Assume furthermore one of the following:
 \begin{enumerate}[(1)]
  \item $f$ is Kupka-Smale; or
  \item there is $\mu_1\in \HPM(f)$ such that $\cO(x_1)\hsim\mu_1$ and $\delta^s(\mu_1)>1/r$.
 \end{enumerate}
Then {$W^u(\cO(x_1))$ accumulates on $\partial^u Q$ for the $C^1$ topology.
Moreover, if $x_2$ is a hyperbolic periodic point of $f$ such that
$W^s(x_2)$ meets both $Q$ and $M\setminus\overline{Q}$,} then ${\cO}(x_1)\preceq {\cO}(x_2)$.
\end{proposition}

We will use the following property of the stable and unstable laminations of basic sets. Recall that the spectral decomposition $\Lambda=\Lambda_0\cup\dots\cup\Lambda_{p-1}$ of a basic set satisfies: $f(\Lambda_i)=\Lambda_{i+1}$, $f(\Lambda_{p-1})=\Lambda_0$ and $f^p|_{\Lambda_0}$ topologically mixing.

\begin{lemma}\label{lemma-W-minimality}
Let $\Lambda=\Lambda_0\cup\dots\cup\Lambda_{p-1}$ be a basic set with its spectral decomposition. If $x,y$ belong to the same $\Lambda_i$ for some $0\leq i<p$ then $W^u(x)$ accumulates on $W^u(y)$. More precisely, for any compact disc $D\subset W^u(y)$, there are compact discs $D_1,D_2,\dots\subset W^u(x)$ such that $D_n\to D$ in the $C^1$ topology.
\end{lemma}

\begin{proof}[{Proof of Proposition \ref{prop-common-Q}}]
By definition, the boundary of $Q$ is a finite union of segments of $W^s(\cO)\cup W^u(\cO)$ for some $\cO\in\HPO(f)$.

\medbreak

{To begin with, we assume (1), i.e.,  $f$ is Kupka-Smale. Clearly,} $W^u(x_1)$ intersects $\partial^sQ$, hence $W^s(\cO)$.
The Kupka-Smale property implies that this intersection is transverse. This gives  $\cO(x_1)\preceq \cO$. {The Inclination Lemma~\ref{l.density-submanifolds} then shows that $W^u(\cO(x_1))$ accumulates on $W^u(\cO)\supset\partial^u Q$, as claimed.}
Similarly we get $\cO\preceq \cO(x_2)$ and therefore $\cO(x_1)\preceq \cO(x_2)$, concluding the proof in this case.

\medbreak

We then assume (2): $\cO(x_1)\hsim\mu$ with $\mu\in\HPM(f)$ such that $\delta^s(\mu)>1/r$. By Theorem \ref{t.katok}, for arbitrarily small $\eps>0$, the measure $\mu$ is homoclinically related to a horseshoe $\Lambda$ with $h_\top(f,\Lambda)>h(f,\mu)-\eps$ and $\lambda^s(f,\nu)<\lambda^s(f,\mu)+\eps$ for all $\nu\in\Prob(f|_{\Lambda})$. A routine argument (see, e.g., Proposition~\ref{p.chihypset}) shows that $\lambda^s(f,\Lambda)<\lambda^s(\mu)+\eps$. So $h_{\top}(f,\Lambda)>h(f,\mu)-\epsilon\equiv \frac{\lambda^s(\mu)}{r}+\lambda^s(\mu)\left(\delta^s(\mu)-\frac{1}{r}\right)-\epsilon$.
Since by assumption $\delta^s(\mu)>1/r$,  $h_\top(f,\Lambda)>\lambda^s(f,\Lambda)/r$ for all $\epsilon>0$ sufficiently small.
\medbreak

The Inclination Lemma and the transitivity of Smale's pre-order show that we can also assume that $x_1\in\Lambda$.
Let $\Lambda=\Lambda_0\cup\dots\cup\Lambda_{p-1}$ be the spectral decomposition with, say, $x_1\in\Lambda_0$.

Since $f$ is $C^r$, Theorem \ref{basic} \eqref{basic0}-\eqref{basic1} shows that $\mathfs W^u_\epsilon(\Lambda)$ is  a continuous lamination with $C^r$ leaves,
that $W^s(\cO)$ is $C^r$, and that $\mathfs W^u_\epsilon(\Lambda)$ has Lipschitz holonomies.
By Lemma~\ref{lemma-W-minimality}, for all $x\in\Lambda_0$, $W^u(x)$ accumulates on $W^u(x_1)$ and therefore, for all $x\in\Lambda$, $W^u(x)\cap W^s(\cO)\ne\emptyset$.
By invariance, there exists a curve $\gamma\subset W^s(\cO)$ which intersects each leaf of $\mathfs W^u_\epsilon(\Lambda_0)$.
Let us assume that there is no $x\in\Lambda_0$ such that $W^u(x)\pitchfork \gamma\neq \emptyset$,
so $\mathfs W^u_\varepsilon(\Lambda_0)\subset\mathfs N_\gamma(\mathfs W^u_\varepsilon(\Lambda))$.
Thus Theorem~\ref{t.DynSard} gives:
$$
\dbar(\mathfs W^u_\epsilon(\Lambda_0)) \leq 1/r.
$$
But Theorem~\ref{basic} (\ref{basic2}) and (\ref{basic3}) say that for every $x\in\Lambda$,
$$
{\dbar(\mathfs W^u_\epsilon(\Lambda_0))\geq \dim_H(W^s_\epsilon(x)\cap\Lambda_0)=\dbar^u_{\Lambda_0} \geq h_\top(f,\Lambda) / \lambda^s(f,\Lambda)> 1/r.}
$$
This contradiction {gives some} $x\in\Lambda_0$ with $W^u(x)\pitchfork W^s(\cO)\neq\emptyset$.
From Lemma~\ref{lemma-W-minimality} this holds also for $x_1$. By the inclination lemma, $W^u(\cO(x_1))$ accumulates on $W^u(\cO)$ and $\partial^u(Q)$ in the $C^1$ topology.

Let us consider now some hyperbolic periodic point $x_2$ such that $W^s(x_2)$ intersects both $Q$ and $M\setminus \overline{Q}$.
The manifold $W^s(x_2)$ crosses $W^u(q)$ topologically for some $q\in\cO$, and hence intersects each global unstable manifold of $\Lambda_j$
for some $0\leq j<p$.
We refer to Figure \ref{figure-intersectQK}.

\begin{figure}[ht]
\begin{center}
\includegraphics[width=10cm,angle=0]{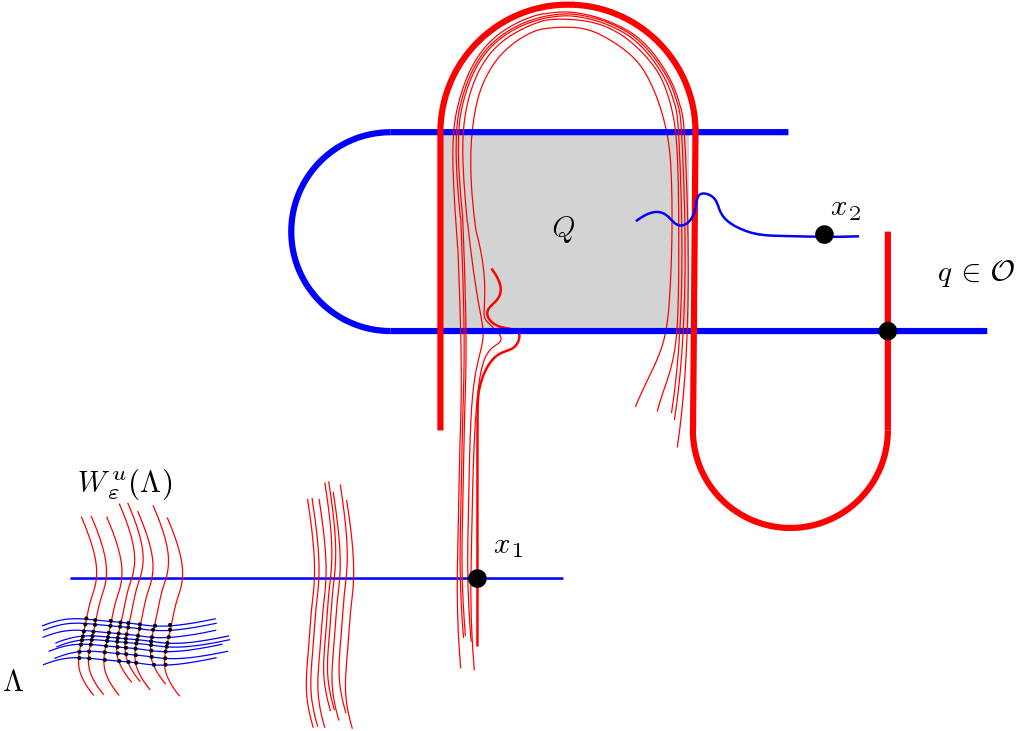}
\end{center}
\caption{\label{figure-intersectQK} Proof of Proposition \ref{prop-common-Q}: {The $su$-quadrilateral $Q$, the two periodic points $x_1,x_2$, the horseshoe $\Lambda$ and the relevant stable (horizontal) and unstable (vertical) manifolds and foliations.}}
\end{figure}

We can conclude as before that, for some (and then for all) $x\in\Lambda_j$, one has $W^u(x)\pitchfork W^s(x_2)\ne\emptyset$.
Considering $x=f^j(x_1)\in\Lambda_j$, we get $\cO(x_1)\preceq \cO(x_2)$.
\end{proof}

\section{Bounding the number of homoclinic classes with large entropy}\label{sec-finite}
In this section, we show that for some small  constant $c$,  the number of homoclinic classes of ergodic hyperbolic measures with entropy bigger than $c$ is finite. We do this by showing that $\limsup h(f,\mu_n)\leq c$ for every infinite sequence of pairwise non homoclinically related  ergodic hyperbolic measures $\mu_n$.
The value of $c$ depends on the regularity of the diffeomorphism and its expansion properties. Suppose  $f\in\Diff^r(M)$ and define  $\|Df^k\|:=\max\{\|Df^k v\|: v\in T M, \|v\|=1\}$. Let
\begin{equation}\label{lambda-max-lambda-min}
\begin{aligned}
\lambda_{\min}(f)&=\min\{\lambda^s(f),\lambda^u(f)\}\\
\lambda_{\max}(f)&=\max\{\lambda^s(f),\lambda^u(f)\}
\end{aligned}
\text{ where }\lambda^u(f):=\lim\limits_{n\to\infty}\frac{1}{n}\log\|Df^n\|,\ \lambda^s(f):=\lim\limits_{n\to\infty}\frac{1}{n}\log \|Df^{-n}\|.
\end{equation}
We shall see below that $c=\lambda_{\min}(f)/r$. In particular, if $f\in \Diff^\infty(M)$ then for every $\epsilon>0$ there are at most finitely many different homoclinic classes carrying measures with entropy bigger than $\epsilon$.

\subsection{Preliminaries on entropy}\label{ss.entropy}
We need the following facts on the metric entropy of invariant measures of homeomorphisms on compact metric spaces (such as $f:M\to M$).

Given $\epsilon>0$ and $n\in\mathbb N$, an  {\em $(\epsilon,n)$-Bowen ball} is a set of the form
$$
B(x,\epsilon,n):= \{y\in M:  d(f^k(x),f^k(y))\leq \epsilon\ \forall 0\leq k<n\},
$$
where $x\in M$.
Consider $Y\subset M$ (not necessarily invariant),  and let $\varepsilon>0$.
A set $F\subset Y$ is called {\em $(\epsilon,n)$-spanning} for $Y$, if $Y\subset \bigcup_{x\in F} B(x,\eps,n)$.

Let
$r_f(\epsilon,n,Y)$ denote the minimal cardinality of an $(\epsilon,n)$-spanning subset of $Y$.
When $Y=M$, we  write $r_f(\epsilon,n)$.

\medskip
\noindent
{\sc Topological entropy on non-invariant sets \cite{Bowen-Entropy-Expansive}.} The {\em topological entropy} of $f$ on a  set $Y$ is
\begin{align*}
&h_{\top}(f,Y):=\lim_{\epsilon\to 0^+} h(f,Y,\epsilon),\text{where}\\
&h_{\top}(f,Y,\epsilon):=\limsup_{n\to\infty}\frac{1}{n} \log r_f(\epsilon,n,Y).
\end{align*}

\medskip
\noindent
{\sc Katok entropy formula \cite[Theorem 1.1]{KatokIHES}.} Let $\mu$ be an $f$-ergodic invariant probability measure, then the metric entropy of $\mu$ satisfies
 $$\begin{aligned}
    &h(f,\mu)=\lim_{\eps\to 0^+} h(f,\mu,\eps),\text{ where }\\
    &h(f,\mu,\eps):=\limsup_{n\to\infty} \frac1n\inf_{\mu(Y)>1/2}\log r_f(\eps,n,Y).
 \end{aligned}$$
Since $\mu$ is ergodic, $1/2$ can be replaced by any $0<\tau<1$ without affecting the value of $h(f,\mu,\epsilon)$.
Katok's entropy formula implies that $h(f,\mu,\epsilon)\leq h(f,\mu)\leq h_{\top}(f)$ (the last inequality was first proved in \cite{Goodwyn}).

\medskip

The metric entropy is an affine function of the measure, and if  $\mu=\int_X \mu_x\, d\mu(x)$ is the ergodic decomposition of $\mu\in\Prob(f)$, then $h(f,\mu)=\int_X h(f,\mu_x)\, d\mu(x)$ (see \cite[Sec. 9.8]{RokhlinLectures}). This suggests the following
 definition of $h(f,\mu,\epsilon)$ for non-ergodic measures $\mu$:
 $$
    h(f,\mu,\eps):=\int_X h(f,\mu_x,\eps)\, d\mu(x).
 $$
By the  monotone convergence theorem,  $h(f,\mu)=\lim_{\eps\to 0^+} h(f,\mu,\eps)$.

\medskip
For $C^1$ surface diffeomorphisms, ergodic measures with positive entropy are hyperbolic of saddle type because of Ruelle's Entropy Inequality~\cite{Ruelle-Entropy-Inequality}:
\begin{theorem}[Ruelle]\label{thm.ruelle}
Let $f$ be a $C^1$ diffeomorphism on a surface, and let $\mu$ be an ergodic invariant measure. The Lyapunov exponents $\chi_1(f,\mu)\leq \chi_2(f,\mu)$
(counted with multiplicity) satisfy
$$
    h(f,\mu)\leq \max\{\min(-\chi_1(f,\mu), \chi_2(f,\mu)),0\}.
$$
In particular, if $h(f,\mu)>0$ then $\mu\in \HPM(f)$ and $\lambda_{\min}(f)\geq h(f,\mu)$.
\end{theorem}

\medskip
\noindent
{\sc Tail Entropy.}
The {\em  tail entropy} of $f$  is defined by
 $$\begin{aligned}
     h^*(f) \quad &:=\lim_{\eps\to0^+} h^*(f,\eps) \text{ where }\\
    h^*(f,\eps)&:=\sup_{x\in M} h_{\top}(f,\{y\in M:\forall n\geq0\; d(f^ny,f^nx)<\eps\}).
     \end{aligned}$$

The quantities $h^*(f,\eps)$ were introduced by Bowen \cite{Bowen-Entropy-Expansive} and then studied by Misiurewicz \cite{misiurewicz,MisiurewiczTCE} together with their limit $h^*(f)$ under the name of topological conditional entropies. The relevance of this concept for us lies in the following well-known inequalities (variants go back to \cite{Bowen-Entropy-Expansive,misiurewicz}). First, recalling that for any $\delta>0$
 $$
   h^*(f,\eps)= \lim_{n\to\infty} \frac1n\log\sup_{x\in M} r_f(\delta,n,B(x,\eps,n))
 $$
(see Proposition 2.2 in \cite{Bowen-Entropy-Expansive}), it follows that:
\begin{equation}\label{Tail-Inequality}
  h(f,\mu)\leq h(f,\mu,\eps)+h^*(f,\eps).
\end{equation}
Second, we have the following upper semicontinuity property (see \cite[Theorem 4.2]{MisiurewiczTCE}):
\begin{equation}\label{Tail-semi-continuity}
    \limsup_{n\to\infty} h(f,\mu_n)\leq h(f,\mu)+h^*(f) \text{ if }\mu_n\to\mu \text{ weak-$*$}.
\end{equation}

When $f$ is a $C^r$ diffeomorphism of a closed surface, Burguet has obtained in~\cite{Burguet} the following bound on the tail entropy,
which improves previous estimates by Buzzi~\cite{BuzziSIM} using Yomdin theory:
\begin{equation}\label{Tail-bound}
    h^*(f)\leq \lambdamin(f)/r.
\end{equation}
This implies the following special case of a theorem of Newhouse \cite{Newhouse-Entropy}: Given a $C^\infty$ surface diffeomorphism $f$, the entropy function $\mu\mapsto h(f,\mu)$ is upper semicontinuous over the set of invariant probability measures. (Newhouse's result holds in any dimension.)

\medskip
\noindent
{\sc Dependence on the diffeomorphism.}
According to a theorem of Newhouse~\cite{Newhouse-Entropy}, $f\mapsto h_{\top}(f)$ is continuous on $\Diff^\infty(M)$. In $\Diff^r(M)$ we have the following bound on the defect of upper semicontinuity. The \emph{robust tail entropy} (introduced in \cite{BurguetFibered2017}) is defined by
 $$
      \HrobR(f):=\lim_{\eps\to0^+} \HrobR(f,\eps) \text{ with }\HrobR(f,\eps):=\limsup_{g\stackrel{C^r}\to f} h^*(g,\eps).
 $$
It bounds the defect in upper semi-continuity as follows:
 $$
    \limsup_{(f_n,\mu_n)\to(f,\mu)} h(f_n,\mu_n) \leq h(f,\mu) + \HrobInf(f) \text{ when $(f_n,\mu_n)\to f$ in $\Diff^r(M)\times\Prob(M)$.}
 $$

In this setting, uniform estimates from Yomdin's theory \cite[Prop. 7.17]{BuzziPhD} show :\footnote{Burguet has explained to us how the variational principle established in  \cite{BurguetFibered2017} yields the sharp estimate $\HrobR(f)\leq \lambdamin(f)/r$ in the case of  surface diffeomorphisms.}
 $$
   \HrobR(f) \leq \frac{\dim M}{r} \cdot \lambdamin(f).
 $$
In particular, $\HrobInf(f)=0$ for any $f\in\Diff^\infty(M)$. Observe that if $f_n\to f$ in the $C^\infty$-topology and $\mu_n$ is a m.m.e. for $f_n$, then any weak-$*$ accumulation point of $\mu_n$ is a m.m.e. for $f$.

\subsection{Finiteness of the number of homoclinic classes with large entropy}
There are only finitely many homoclinic classes containing measures with large entropy:

\begin{theorem}\label{thmEntropyDecay}
Suppose $M$ is a closed surface and $f:M\to M$  is a $C^r$ diffeomorphism with $r>1$. If $\mu_n$ are ergodic hyperbolic measures such that $\mu_i,\mu_j$ are not homoclinically related for $i\neq j$, then
 \begin{equation}\label{e.limit}
    \limsup_{n\to\infty} h(f,\mu_n)\leq \lambdamin(f)/r.
 \end{equation}
If $f$ is Kupka-Smale we have the better bound $\limsup_{n\to\infty} h(f,\mu_n)\leq h^*(f)$.
\end{theorem}
\begin{proof}
By~\eqref{Tail-bound},  $\lambdamin(f)/r\geq  h^*(f)$.
We assume by contradiction that there exists some number $h$ close to $\limsup_{n\to\infty} h(f,\mu_n)$
such that
$$h^*(f)<h<\limsup_{n\to\infty} h(f,\mu_n) \text{ if $f$ is Kupka-Smale,}$$
$$\lambdamin(f)/r<h<\limsup_{n\to\infty} h(f,\mu_n) \text{ if $f$ is not Kupka-Smale.}$$

Without loss of generality, $\lambdamin(f)=\lambda^s(f)$ (otherwise we replace $f$ by $f^{-1}$).
From our assumptions, when $f$ is not Kupka-Smale, the measures $\mu_n$ with $n$ large have entropy larger than $\lambda^s(f)/r$,
hence satisfy $\delta^s(\mu_n):=h(f,\mu)/\lambda^s(\mu)>(\lambda^s(f)/r)/\lambda^s(\mu)\geq \frac 1 r$.
Since $M$ is compact, we may assume without loss of generality that
$$
\mu_n\xrightarrow[n\to\infty]{\text{weak-$*$ }}\nu.
$$
From~\eqref{Tail-semi-continuity}
 $$
     h(f,\nu)\geq \limsup h(f,\mu_n)-h^*(f)\geq h-h^*(f)>0.
 $$
The limit measure $\nu$ needs not be ergodic. Using the ergodic decomposition, we can represent
$$
\nu=(1-\alpha)\nu_0+\alpha\nu_1,
$$
where $\alpha\in (0,1]$, and where $\nu_0,\nu_1$ are two $f$-invariant probability measures such that $h(f,\nu_0)=0$ and almost every ergodic component of $\nu_1$ has positive entropy.

Let $\delta:=h-h^*(f)$, and fix $\eps>0$ such that
\begin{equation}\label{e.eps}
   h^*(f,\eps)<h^*(f)+\delta/2.
\end{equation}
Pick a finite subset $\mathcal C_1\subset M$ with  $\bigcup_{x\in \mathcal C_1} B(x,\eps/2)=M$.
Choose  $\kappa>0$ so small that
$$
     \kappa< \frac{\delta/10}{\log \#\mathcal C_1} \text{ and }
    h(\kappa)<\delta/10, \text{ where }h(t):=-t\log t-(1-t)\log t.
$$

\medskip
\noindent
{\sc Claim 1:}
{\em For all integers $N_0$ large enough there are subsets $\mathcal C_0=\mathcal C_0(N_0)\subset M$  with
 $\# \mathcal C_0\leq  e^{\delta N_0/10}$ and such that $\nu_0(B_0)>1-\kappa$, where
$$
    B_0:=\bigcup_{x\in \mathcal C_0} \mathrm{int}\,B(x,\eps/2,N_0).
$$}

\begin{proof}[Proof of the Claim] Find a finite measurable partition $\beta$ of $M$ into sets of diameter less than $\epsilon/2$.
Let $\beta_n:=\bigvee_{i=0}^{n-1}f^{-i}\beta$ and $\beta_n(x):=$ the atom of $\beta_n$ which contains $x$.

By definition, the metric entropy with respect to the partition $\beta$ is:
 $$
   h_{\nu_0}(f,\beta):=\lim_{N\to\infty} -\frac1{N}\sum_{A\in\beta_N} \nu_0(A)\log\nu_0(A).
 $$
Since $h({\nu_0},f)=0$,  $h_{\nu_0}(f,\beta)=0$. A direct computation shows that for all $N$ large enough,
 $$
    B_0':=\{x\in M: \nu_0(\beta_N(x))>e^{-N\delta/10}\}\text{ satisfies } \nu_0(B_0')>1-\kappa.
 $$
Enumerate $\{\beta_{N}(x):x\in B_0'\}$ as $\beta_N(x_1),\ldots,\beta_N(x_{\ell(N)})$, and  take $\mathcal C_0:=\{x_1,\ldots,x_{\ell(N)}\}.$
Observe that $\ell(N)\le e^{N\delta/10}$ and
$B_0=\bigcup_{x\in \mathcal C_0}B(x,\epsilon/2,N)\supset\bigcup\beta_N(x_i)=B_0'$, whence $\nu_0(B_0)>1-\kappa$.
\end{proof}

\medskip
Choose $N_0$ and $\mathcal C_0$ as in the claim, making sure that  $N_0$ is large enough to satisfy $h(1/N_0)<\delta/10$.

\medskip
By construction, almost every ergodic component of $\nu_1$ has positive entropy and is therefore hyperbolic of saddle type. By Proposition \ref{p.rectangles},  there are  finitely many $su$-quadrilaterals  $Q_1,\dots,Q_N$, associated to $\wt{\cO}_1,\dots,\wt{\cO}_N\in\HPO(f)$   such that
$$
   \diam (Q_i)<\epsilon \text{ and }\nu_1\left(\bigcup_{i=1}^N Q_i\right)>1-\kappa,
$$

\medskip

The measures $\mu_n$ converge weak-$*$ to $\nu$ and $B_0\cup\bigcup_i Q_i$ is open, so
 $$
    \nu\left(B_0\cup\bigcup_i Q_i\right) \leq \liminf_{n\to\infty} \mu_n\left(B_0\cup\bigcup_i Q_i\right).
 $$
Therefore, for all large $n$,
 \begin{equation}\label{eq-large-nu-n}
   \mu_n\left(B_0\cup\bigcup_{i=1}^N Q_i\right)>1-\kappa, \text{ and   }
   \mu_n\left(\bigcup_{i=1}^N Q_i\right)>\alpha(1-\kappa)>0.
 \end{equation}

Since $h_{\mu_n}(f)>h$, Theorem \ref{t.katok} gives a horseshoe $\Lambda_n$  such that $h(f|_{\Lambda_n})>h$ and $\Lambda_n$ contains  a saddle $\cO_n$ homoclinically related to $\mu_n$. In particular, $\mu_n(\HC(\cO_n))=1$. Consider the spectral decomposition of $\HC(\cO_n)$ from Proposition \ref{p-hc-cyclic}:
 $$
   \HC(\cO_n)=K_{n,0}\cup\dots\cup K_{n,\tau_n-1}
 $$
 where $K_{n,j}=\overline{W^u(f^jx_n)\ti W^s(f^jx_n)}$ for some $x_n\in\cO_n$ and all $0\leq j<\tau_n$.
 We have $f(K_{n,j})=K_{n,j+1}$ for $j=0,\ldots,\tau_n-2$ and $f(K_{n,\tau_n-1})=K_{n,0}$. It is convenient to extend $K_{n,j}$ to $j\geq \tau_n$ periodically: $K_{n,j+\tau_n}=K_{n,j}$.

\medbreak\noindent
{\sc Claim 2.} {\em  For all large $n$, there exist $0\leq j< \tau_n$ and $1\leq i\leq N$ such that $K_{n,j}$ intersects both $Q_i$ and $M\setminus\ov{Q_i}$.}

\begin{proof}[Proof of the Claim]
We fix $n$ large. Since $\mu_n(\bigcup_{j=0}^{\tau_n-1} K_{n,j})=\mu_n(\HC(\cO_n))=1$ and $\mu_n(\bigcup_{i=1}^N Q_i)>0$, there exist $0\leq j\leq \tau_n-1$ and $1\leq i\leq N$ such that  $K_{n,j}$ intersects $Q_i$. Assume by way of contradiction that every $K_{n,j}$ that intersects $Q_i$   is contained in $\ov{Q_i}$.

Notice that $\mu_n(K_{n,\ell})\geq 1/\tau_n$ for all $\ell$. We fix a measurable $G\subset K_{n,0}$ such that $\mu_n(G)>1/(2\tau_n)$
and so that for some $N_1$, for all $m>N_1$ and every $y\in G$
\begin{equation}\label{genericity}
\#\{k\in \{0,\dots,m-1\}\: : \: f^k(y)\in  B_0\cup\bigcup Q_i\} > (1-\kappa)m.
\end{equation}
Such a set exists because of the ergodicity of $\mu_n$, \eqref{eq-large-nu-n}, and the pointwise ergodic theorem.

We fix some large integer $m>N_1$, and cover $G$ by a ``small" collection of $(\eps,m)$-Bowen balls.
 Let
 $$
    J_1:=\{0\leq j<m:K_{n,j}\subset \bigcup_{1\leq i\leq N} \ov{Q_i}\}.
 $$
For every $j\in J_1$,  let $i(j)$ be the smallest  $1\leq i\leq N$ such that $\ov{Q_i}$ contains $K_{n,j}$. Clearly,  for all $j\in J_1$, $j+\tau_n\in J_1$ and $i(j+\tau_n)=i(j)$.
Since $G\subset K_{n,0}$, for every $y\in G$,
 $
 j\in J_1\implies f^j(y)\in K_{n,j}\subseteq \ov{Q}_{i(j)}
 $.

Let  $N_0$ be as in  Claim 1, set $a_0:=-N_0$, and define inductively $a_{\ell+1}:=\min\{j\geq a_\ell+N_0:f^j(y)\in B_0\}$ if the set is non-empty, and $a_{\ell+1}:=+\infty$ otherwise. Let $\ell_*:=\max\{\ell\geq0:a_\ell<m\}$ and
 $$
 J_0:=\{a_\ell:1\leq \ell\leq\ell_*\}\,\subset\,
   J_2:=\bigcup_{1\leq \ell\leq \ell_*}[a_\ell,a_\ell+N_0-1] \text{ , }
   J_3:=\{0,1,\dots,m-1\}\setminus(J_1\cup J_2).
 $$
Note that $J_0$, $J_2$ and $J_3$ (but not $J_1$) depend on $y$ and that $j\in J_3\implies f^j(y)\notin B_0\cup\bigcup_i Q_i$. So by (\ref{genericity}), we have  $\#J_3< \kappa\cdot m$.

Recall the finite collections of points $\mathcal C_0,\mathcal C_1$ such that $B_0=\bigcup_{x\in\mathcal C_0} B(x,\epsilon/2,N_0)$ and   $\bigcup_{x\in \mathcal C_1}B(x,\epsilon/2)=M$. To each $y\in G$, we associate the following data:
 \begin{enumerate}[(i)]
  \item $J_0\subset[0,m-1]$ whose elements are separated by at least $N_0$;
  \item $X_0:=(x_{0,j})_{j\in J_0}\in(\mathcal C_0)^{J_0}$ such that for all
   $j\in J_0$, we have $f^j(y)\in B(x_{0,j},\eps/2,N_0)$;
  \item $J_3\subset[0,m-1]$ with $\#J_3< \kappa\cdot m$;
  \item $X_3:=(x_{3,j})_{j\in J_3}\in (\mathcal C_1)^{J_3}$ such that
  for all $j\in J_3$ we have  $f^j(y)\in B(x_{3,j},\eps/2).
  $
 \end{enumerate}

Here is an  upper bound for the number of possibilities for $(J_0,X_0,J_3,X_3)$ as $y$ ranges over $G$:
  $$
    m\choice{m\\ \lceil m/N_0\rceil} \times (\#\mathcal C_0)^{\lceil m/N_0\rceil} \times
    m\choice{m\\ \lfloor m\cdot \kappa\rfloor} \times (\#\mathcal C_1)^{\lfloor \kappa\cdot m\rfloor}.
  $$
The first and third factors are upper bounds for the number of subsets of $\{1,\ldots,m\}$ with cardinality less than $m/N_0$ or $\kappa m$ (bounds satisfied by $J_0$, $J_3$).

Recall the de Moivre-Laplace approximation: for every $p\in (0,1)$, if  $p_m\to p$ as $m\to\infty$, then  ${m\choose p_m m}\sim e^{m h(p_m)}/\sqrt{2\pi m p (1-p)}=e^{m(h(p)+o(1))}$ as $m\to\infty$.  It follows that the number of possibilities for $(J_0,X_0,J_3,X_3)$ as $y$ ranges over $G$ is bounded by
 $$
    \exp m\biggl( (h(1/N_0)+o(1)) + \delta/10 + (h(\kappa)+o(1)) + \kappa\log\#\mathcal C_1 \biggr) , \text{ as $m\to\infty$.}
 $$
Recall that $h(1/N_0)<\delta/10$, $h(\kappa)<\delta/10$, and $\kappa\log\#\mathcal C_1<\delta/10$. So for all $m$ large enough, the number of possibilities for $(J_0,X_0,J_3,X_3)$ is less than
$\exp (\delta m/2)$.

For  $j\in J_1$,  since $y\in G\subset K_{n,0}$,   $f^j(y)\in K_{n,j}\subset \ov{Q_{i(j)}}$ which has diameter $<\eps$.
For $j\in J_2$, the location of $f^j(y)$ is determined up to error $\epsilon/2$  by $X_0$, and  for $j\in J_3$ by $X_3$.
  So  if $y,y'\in G$ share the same data $(J_0,X_0,J_3,X_3)$,
then $y'\in B(y,\eps,m)$.

It follows that  for every $m$ large enough,  $G$ has a cover by $(\epsilon, m)$-Bowen balls with cardinality at most $\exp(\delta m/2)$. By Katok's entropy formula, this means that
  $
  h(f,\mu_n,\eps) \leq {\delta}/{2}
 $. We chose $\epsilon$ so that $h^*(f,\epsilon)<h^*(f)+\delta/2$. By  (\ref{Tail-Inequality}) and \eqref{e.eps},
 $$
  h_{\mu_n}(f)\leq h(f,\mu_n,\eps)+ h^*(f,\eps)
   < \frac{1}{2}\delta +h^*(f)+\frac1{2}\delta
  = h^*(f)+\delta= h.
 $$
But by assumption, $h_{\mu_n}(f)>h$. This contradiction proves the claim.
\end{proof}

\medskip
We can now complete the proof of the theorem as follows.
Fix $M_0$ so large that every $n>M_0$ satisfies Claim 2. For such $n$, let $0\leq j(n)<\tau_n$ and $1\leq i(n)\leq N$ be indices such that $K_{n,j(n)}$ intersects $Q_{i(n)}$ and $M\setminus\ov{Q_{i(n)}}$. Since the range of $i(\cdot)$ is bounded, there are $n_1,n_2>M_0$ such that $n_1\neq n_2$ and $i(n_1)=i(n_2)=:i$.

By the definition of $K_{n,j}$ there are $y_1\in\cO_{n_1}$ and $y_2\in \cO_{n_2}$ such that  $K_{n_1,j(n_1)}=\ov{W^u(y_1)\pitchfork W^s(y_1)}$ and $K_{n_2,j(n_2)}=\ov{W^u(y_2)\pitchfork W^s(y_2)}$. So $W^\sigma(y_1), W^\sigma(y_2)$ ($\sigma=u,s$) both intersect $Q_i$ and $M\setminus\ov{Q}_i$.
We may apply  Proposition \ref{prop-common-Q} to the orbits $\cO_{n_1},\cO_{n_2}$
(either $f$ is Kupka-Smale or $\cO_{n_1},\cO_{n_2}$ are homoclinically related to measures $\mu_{n_i}$ such that
$\delta^s(\mu_{n_i})>1/r$). This gives $\cO_{n_1}\hsim\cO_{n_2}$. But this implies that $\mu_{n_1}\hsim\mu_{n_2}$ in contradiction to our assumptions.
\end{proof}
\begin{corollary}\label{Cor-finite-mme}A $C^\infty$ surface diffeomorphism with positive topological entropy has at most finitely many ergodic measures of maximal entropy.
\end{corollary}
\begin{proof}
Corollary \ref{c-local-uniqueness} with $\phi\equiv 0$ says that different ergodic measures of maximal entropy are not homoclinically related. So the existence of infinitely many ergodic measures of maximal entropy would contradict  Theorem \ref{thmEntropyDecay}.
\end{proof}

\subsection{Semi-continuity of the simplex of measures of maximal entropy (theorem \ref{theorem-N-usc})}\label{sec-semicontinuity-nbrMME}
In this section we prove theorem \ref{theorem-N-usc}.
We need the following well-known stability result \cite{Katok-Hasselblatt-Book}, \cite{Shub-Book}: For every hyperbolic periodic point $p$ of $f\in\Diff^\infty(M)$ with given local stable and unstable manifolds $W^u_{loc}(p), W^s_{loc}(p)$ and period $\pi(p)$, there are $\delta>0$ and a neighborhood   $f\in \mathcal U\subset\Diff^\infty(M)$ such that for every $g\in \mathcal U$, $g^{\pi(p)}$ has a unique hyperbolic fixed point $p(g)$ $\delta$-close to $p$.
This is a $\pi(p)$-periodic point for $g$
and  for $\mathcal U$ sufficiently small,  the local stable and unstable manifolds of $p(g)$  are $\delta$-close in the $C^1$-topology to $W^s_{loc}(p), W^u_{loc}(p)$.
We call $p(g)$, $W^s_{loc}(p(g)), W^u_{loc}(p(g))$ the {\em hyperbolic continuations} of  $p$, $W^s_{loc}(p), W^u_{loc}(p)$ {\em on} $\mathcal U$.

\medskip
The proof of Theorem~\ref{theorem-N-usc} is based on Theorem~\ref{thmEntropyDecay} and the following proposition:

\begin{proposition}\label{prop-fn}
Let $f$ be a $C^\infty$-surface diffeomorphism of a closed surface $M$ with $h_\top(f)>0$.
Consider some m.m.e. $\mu_0=\alpha \nu + (1-\alpha)\nu'$, $\alpha\in (0,1]$,
where  $\nu$ is an ergodic m.m.e. and let $\cO\hsim\nu$ be some hyperbolic periodic orbit. Then there are neighborhoods $\cU\ni f$ in $\Diff^\infty(M)$ and $U\ni\mu_0$ and a number $h<h_\top(f)$ such that, for any $g\in\cU$, and every measure $\mu\in\Proberg(g)\cap U$ with entropy $h(g,\mu)>h$, $\mu$ is homoclinically related to the hyperbolic
continuation $\cO_g$ of the orbit $\cO$.
\end{proposition}

\begin{proof}[Proof of Theorem \ref{theorem-N-usc}, given proposition \ref{prop-fn}]:
{Consider $f_n, f\in\Diff^\infty(M)$ such that $f_n\xrightarrow[n\to\infty]{C^\infty}f$, and let $\Sigma_n$ be some $k$-faces of the simplex of measures of maximal entropy of $f_n$. Assume that $(\Sigma_n)_{n\geq 1}$ converges in the Hausdorff topology to some set $\Sigma$.
}

The extremal points of $\Sigma_n$ are $k+1$ distinct ergodic m.m.e.'s $\mu^0_n,\dots,\mu^k_n$. By passing to a subsequence, we can assume that, for each $i$, the sequence $(\mu^i_n)$ converges to some probability measure $\mu^i$.
Each $\mu^i$ is a m.m.e. for $f$ (see the end of section~\ref{ss.entropy}), and  $\Sigma$ is the convex hull of $\mu^0,\ldots,\mu^k$.

To  prove that $\Sigma$ is $k$-dimensional, we show that no $\mu^i$ is in the convex hull of $\{\mu^j:j\ne i\}$. Assume by contradiction that $\mu^0=\sum_{i\neq 0}  \alpha_i\mu^i$ (say). Without loss of generality $\alpha_1\neq 0$.  By Corollary \ref{Cor-finite-mme}, $f$ has only finitely many ergodic m.m.e.'s. One of them, say $\nu$, is an ergodic component of   $\mu^0$ and $\mu^1$.

Fix $\cO\in\HPO(f)$ such that $\nu\hsim\cO$. For $i=0,1$, the proposition applied to $(\mu^i,\nu,\cO)$ yields neighborhoods $\cU_i$, $U_i$ and a number $h_i$. For $n$ large enough, $f_n\in\cU_0\cap\cU_1$, $\mu^0_n\in\Proberg(f_n)\cap U_0$ and $\mu^1_n\in\Proberg(f_n)\cap U_1$.
As recalled in Section~\ref{ss.entropy}, $\max(h_0,h_1)\leq \limsup_{(f_n,\mu_n^i)\to (f,\mu^i)}h(f_n,\mu_n^i)\leq h(f,\mu^i)$. Let $\cO_{f_n}$ denote the hyperbolic continuation of $\cO$ to $\mathcal U_0\cap\mathcal U_1$, then $\mu^0_n\hsim\cO_{f_n}\hsim\mu^1_n$ for $n$ large, whence  $\mu^0_n\hsim\mu^1_n$ (Proposition \ref{Prop-Equiv-Rel}). By  Corollary~\ref{c-local-uniqueness}, $\mu^0_n=\mu^n_1$, a contradiction.

{We have thus proved the second part of the theorem. Property~\eqref{e.semi-con} follows immediately.}
\end{proof}

\begin{proof}[Proof of Proposition \ref{prop-fn}]
Let us consider the robust tail entropy $\HrobInf(f,\eps)$ at scale $\varepsilon$ as defined in Section~\ref{ss.entropy}.
Since $f$ is $C^\infty$, it goes to $0$ as $\varepsilon\to 0$.
Hence, we can choose $\delta,\eps>0$ small enough, so that
 $$
  \HrobInf(f,\eps) <\delta < \frac{\alpha}{10} h_\top(f).
   $$
Fix some integers $N_0, \ell\geq 1$  large so that
$$
   r_f(\eps/2,N_0) \leq  \exp \bigl((h_\top(f)+\delta)N_0 \bigr) \text{ and } \tfrac{2h_\top(f)}{\delta}N_0 \leq \ell.
$$
The following number is positive and smaller than $h_\top(f)$:
 $$
    h:=4\delta+(1-\alpha/2)h_\top(f).
$$
For $g$ $C^0$-close to $f$ and all $n\geq0$, writing $n=sN_0+t$ with $s\in\N$ and $0\leq t<N_0$, we have
 \begin{equation}\label{eq-rg}\begin{aligned}
   r_g(\eps,n)& \leq r_g(\eps,N_0)^{s+1} \leq r_f(\eps/2,N_0)^{s+1}
   \leq e^{(h_\top(f)+\delta)N_0(n/N_0+1)}\\
   & \leq e^{2h_\top(f) N_0} e^{(h_\top(f)+\delta)n}
    \leq e^{\delta \ell} e^{(h_\top(f)+\delta)n}.
  \end{aligned}
  \end{equation}

Proposition \ref{p.rectangles} applied to $\nu$ yields $su$-quadrilaterals $Q_1,\dots,Q_N$ associated to periodic orbits $\cO'_1,\dots\cO'_N$ that are homoclinically related to $\cO$, such that $\diam(Q_i) <\eps/(\Lip(f)^{\ell}+1)\leq \eps$  and $\nu(\bigcup Q_i)>1/2$.

Let $\cU\ni f$ and $U\ni\mu^0$ be ``small" open sets in $\Diff^\infty(M)$ and $\Prob(M)$ (we will determine how small they need to be below). Now let $g\in\cU$ and $\mu\in\Proberg(g)\cap U$.
The $su$-quadrilaterals $Q_i$ are bounded by transverse curves contained in the stable and unstable manifolds of $\cO'_i$.
For $g\in \cU$, the local stable and unstable manifolds of $\cO_i'$, their images under bounded iterations, and
their transverse intersection admit hyperbolic continuations. Hence if $\mathcal U$ is sufficiently small, then $\cO$ and the $su$-quadrilaterals $Q_i$ admit hyperbolic continuation to $\mathcal U$. We denote these continuations by $\cO_g$ and $Q_i=Q_i(g)$ $(g\in\mathcal U)$.
Choosing $\mathcal U$ and $U$  appropriately guarantees that   $\diam Q_i(g)<\eps/\Lip(g)^\ell$ and  $\mu\left(\bigcup_{i=1,\dots,N}Q_i(g)\right)>\alpha/2$ and $\Lip(g)>1$  for all  $g\in\mathcal U, \mu\in U\cap \Proberg(g)$.

We assume by contradiction that $\mu\in U\cap\Proberg(g)$, and $h(g,\mu)>h$, but $\mu$ is not homoclinically related to $\cO_g$.
Proceeding as in the proof of Theorem~\ref{thmEntropyDecay},  let $\cO(p)$ be a hyperbolic periodic orbit (of $g$) homoclinically related to $\mu$.
Its homoclinic class must have full measure for $\mu$ and, by Proposition \ref{p-hc-cyclic},
it coincides with $\bigcup_{j=0}^{\tau-1} g^j K$ with $K:=\overline{W^s(p)\pitchfork W^u(p)}$ and $g^{\tau}(K)=K$.
Using Proposition~\ref{prop-common-Q} and arguing as in the end of the proof of Theorem~\ref{thmEntropyDecay},
one can show that if some iterate $g^j(K)$ intersects both $Q_i$ and $M\setminus \ov Q_i$, then
$\cO(p)\hsim\cO'_i$,  whence $\mu\hsim\cO_g$. This contradicts our assumption.
One deduces that for each $j\in\Z$ and $1\leq i\leq N$, either $g^j(K)\subset \overline{Q_i}$ or $g^j(K)\subset M\setminus {Q_i}$. Let $J_1:=\{j\in\Z:g^j(K)$ is contained in some $\ov Q_i\}$. Note that $J_1$ contains a fraction $\mu(\bigcup_i \ov Q_i)$ of the integers. For $j\in J_1$, let $i(j):=\min\{1\leq i\leq N:g^j(K)\subset Q_i\}$. Since $g^{\tau}(K)=K$, the set $J_1$ and the function $i$ are $\tau$-periodic.

Continuing as in the proof of Theorem~\ref{thmEntropyDecay}, we bound $r_g(\eps,n,K)$.
Observe that if $g^j(K)\subset \ov Q_i$ and $0\leq n<\ell$, then $\diam (g^{j+n}(K))\leq \diam(g^n(\ov Q_i))\leq \mathrm{Lip}(g)^\ell \diam(\ov Q_i)\leq \eps$.
Thus, the images $g^n(K)$ have diameter less than $\eps$ for all $n\in \widehat J_1:=J_1+[0,\ell-1]$.
Now, $\N\setminus\widehat J_1$ is a disjoint union of maximal subintervals $[a_1,b_1]<[a_2,b_2]<\dots$ with $b_{s+1}>a_s+\ell$ for all $s$.
For each $s\geq 1$, let $C_s$ be a $(\eps/2,b_s-a_s)$-spanning subset of $M$ with minimal cardinality.
To each $x\in K$, associate $(y_s)_{s\geq1}$ with $y_s\in C_s$ such that $g^{a_s}(x)$ belongs to the Bowen ball $B_g(y_s,\eps/2,b_s-a_s)$ for $g$.

We claim that, if $x,x'\in K$ are associated to the same sequence $(y_s)_{s\geq1}$, then all their iterates $g^k(x),g^k(x')$ stay $\eps$-close. This is clear for $k\in \widehat J_1$. If $k\in\N\setminus\widehat J_1$, then both iterates belong to $g^{k-a_s}(B_g(y_s,\eps/2,b_s-a_s))$ for some $s\geq1$
such that $0\leq k-a_s\leq b_s-a_s$. This set also has diameter smaller than $\eps$.
Let us consider some large $n=b_t$. The cardinality of $[a_1,b_1]\cup\dots \cup [a_t,b_t]$ is smaller than $(1-\alpha/2)n$,
and $t$ is smaller than $n/\ell+1$.
By \eqref{eq-rg}, the minimal cardinality of an $(\eps,n)$-spanning subset of $K$ is bounded for $n=b_t$ large by
 $$
  r_g(\eps,n,K) \leq \prod_{s=1}^t |C_s| \leq \prod_{s=1}^t e^{\delta \ell} e^{(h_\top(f)+\delta)(b_s-a_s)} \leq e^{\delta (n+\ell)} e^{(1-\alpha/2)(h_\top(f)+\delta)n}.
 $$
Since $\mu(K)>0$ and $\mu$ is ergodic, this gives $h(g,\mu,\eps)<2\delta +(1-\alpha/2)h_\top(f).$

Now make the neighborhood $\mathcal U$ of $f$  to be so small that for every $g\in\mathcal U$, $h^\ast(g,\eps)\leq h^\ast_{\Diff^\infty}(f,\eps)+\delta$.
Since $\HrobInf(f,\eps)<\delta$  one gets from \eqref{Tail-Inequality} that
 $$
   h(g,\mu) \leq h(g,\mu,\eps)+h^*(g,\eps)
      < 4\delta + (1-\alpha/2)h_\top(f) = h.
 $$
This contradicts our assumption that the entropy of $\mu$ is larger than $h$.
\end{proof}


\section{Spectral decomposition and topological homoclinic classes}\label{s.spectral-decomposition}
In this section we discuss the spectral decomposition for $C^r$ surface diffeomorphisms. To achieve this we analyze  the structure of transitive sets, and the intersection of topological homoclinic classes.
\newcommand\Pchi{\mathbb P_\chi}
\medbreak

\subsection{Thickness of homoclinic classes}
Recall that if $\mu$ is a hyperbolic ergodic measure, then  $\delta^s(\mu)=h(f,\mu)/\lambda^s(\mu)$ where $\lambda^s$ is the absolute value of the negative Lyapunov of $\mu$. The following definition is motivated by Proposition \ref{prop-common-Q}:}

\begin{definition}\label{d-thickness-HC}
Let $\mu$ be a hyperbolic ergodic measure for a diffeomorphism $f$
on a closed surface $M$. We say that $\mu$ is \emph{$s$-thick}
if  there exist $\nu\in \HPM(f)$ and $r>1$ such that $f$ is $C^r$, $\nu\hsim\mu$ and $\delta^s(\nu)>1/r$.
Similarly $\mu$ is \emph{u-thick} if it is $s$-thick for $f^{-1}$. And $\mu$ is \emph{thick}
if it is both $s$-thick and $u$-thick.
\end{definition}

\begin{remark}\label{r.thick}
As before these definitions extend to saddle periodic orbits $\cO$ by considering  $\mu_\cO$, the   invariant probability measure on $\cO$.
Thickness only depends on the equivalence class of the hyperbolic measure for the homoclinic relation.
Note that:
\begin{myitemize}
\item[--] for a $C^\infty$ diffeomorphism, any   ergodic hyperbolic non-atomic measure is thick, because it is homoclinically related to a horseshoe with positive entropy and we can choose $r$ arbitrarily large.
\item[--] for a $C^r$ diffeomorphism any ergodic measure with entropy larger than $\lambda^s(f)/r$ is $s$-thick and
any ergodic measure with entropy larger than $\max\{\lambda^u(f),\lambda^s(f)\}/r$ is thick.
\end{myitemize}
\end{remark}

\begin{proposition}
A measure $\mu\in \HPM(f)$ is $s$-thick
 if and only if there exist $r\in (1,\infty)$ and  a basic set $\Lambda$ such that $f$ is $C^r$, $\Lambda\hsim\mu$,  and $h_\top(f,\Lambda)>\lambda^s(f,\Lambda)/r$.
\end{proposition}
\begin{proof}
Suppose  $\mu$ is $s$-thick, then  $\Lambda$ can be constructed as in case (2)
of the proof of Proposition~\ref{prop-common-Q}.
Conversely, if there is a basic set $\Lambda$ such that  $\mu\hsim\Lambda$  and $h_\top(f,\Lambda)>\lambda^s(f,\Lambda)/r$, then
by the variational principle there exists  an ergodic measure $\nu$  on $\Lambda$ whose entropy is so close to $h_\top(f,\Lambda)$ that
 $h(f,\nu)>\lambda^s(\nu)/r$ and $\delta^s(\nu)>1/r$.
Since $\nu$ is supported in $\Lambda$,  $\nu\hsim\mu$. So $\mu$ is $s$-thick.
\end{proof}

\subsection{Homoclinic relation and topological transitivity}

The next theorem is the key technical result of this section. We will use it below to show that if $f$ is a topologically transitive $C^\infty$ surface diffeomorphism, then any two $\mu_1,\mu_2\in\HPM(f)$ with positive entropy are homoclinically related.

\begin{theorem}\label{thmTransitivity}
Suppose $r>1$ and  $f$ is a $C^r$ diffeomorphism on a closed surface. Let $\Lambda$ be a
transitive set, and suppose $\mu_1,\mu_2$ are two hyperbolic ergodic measures such that
$\HC(\mu_1)\cap\Lambda,\HC(\mu_2)\cap\Lambda$  carry non-atomic ergodic hyperbolic measures
(this holds whenever $h_\top(f,\HC(\mu_i)\cap\Lambda)>0$).
If
\begin{myitemize}
  \item[--] $f$ is Kupka-Smale, or
  \item[--] $\mu_1$ is $s$-thick.
\end{myitemize}
then $\mu_1\preceq\mu_2$.
\end{theorem}

\begin{remark}
The Kupka-Smale condition can be replaced by the following local assumption:
There exist hyperbolic periodic orbits of saddle type $\cO_1\hsim \mu_1$ and $\cO_2\hsim \mu_2$
such that all the intersections between $W^u(\cO_1)$ and $W^s(\cO_2)$ are transverse.
(See the comment at the end of the proof.)
\end{remark}
\newcommand\step[2]{\medbreak\noindent{\sc Step #1.} {\it #2}\medbreak}

\begin{proof}
The idea is to construct $\cO_i\in\HPO(f)$ homoclinically related to $\mu_i$ so that $W^u(\cO_1), W^s(\cO_2)$ intersect the interior and the exterior of the same $su$-quadrilateral, and then invoke Proposition \ref{prop-common-Q}.

\step{1}{
There are $su$-quadrilaterals $Q_1,Q_2$ such that $Q_i\cap\Lambda\cap\HC(\mu_i)\ne\emptyset$, $\HC(\mu_i)\not\subset \ov{Q_i}$, and
 $f^n(\partial^sQ_1)\cap \ov{Q_2}=\emptyset$ and $f^{-n}(\partial^u Q_2)\cap \ov{Q_1}=\emptyset$, for all large $n\geq0$.
}

For each $i=1,2$, we pick:
\begin{myitemize}
 \item[--] $\nu_i$ an ergodic non-atomic hyperbolic measure such that $\nu_i(\Lambda\cap\HC(\mu_i))=1$;
 \item[--]  $\cO'_i$ a periodic hyperbolic orbit with $\cO'_i\hsim\nu_i$ and $\cO'_1\ne\cO'_2$ (these exist by the assumptions on $\nu_i$.);
 \item[--] $x_i\in\supp\nu_i\setminus(\cO'_1\cup\cO'_2)$ (these exist by the assumptions on $\nu_i$.)
\end{myitemize}
Let
 $
   0<\rho<\frac13\min\{d(x_1,x_2),d(\{x_1,x_2\},\cO'_1\cup\cO'_2),\diam(\HC(\mu_1)),\diam(\HC(\mu_2))\}.
 $
Fix $i\in\{0,1\}$.
Proposition \ref{p.rectangles} gives $su$-quadrilaterals with diameters less than $\rho/2$ and whose union has measure larger than
 $1-\nu_i(B(x_i,\rho/2))$. At least  one of the quadrilaterals (call it $\wh{Q}_i$) must be contained in $B(x_i,\rho)$ and have positive $\nu_i$-measure. $\wh{Q}_i$ is associated to a periodic orbit
 homoclinically related to $\nu_i$, whence to $\cO_i'$. Using the inclination lemma, one can replace the $\wh{Q}_i$ by an  $su$-quadrilateral $Q_i$ associated to $\cO_i'$ such that $\diam(Q_i)<\rho/2$, $Q_i\subset B(x_i,\rho/2)$ and $\mu_i(Q_i)>0$
 (see Definition~\ref{def-quadrilateral}).

The choice of $\rho$ ensures that $\ov{Q_i}$ is disjoint from $\cO'_1\cup\cO'_2$. Note that $f^n(\partial^sQ_1)$ and $f^{-n}(\partial^uQ_2)$ converge to $\cO_1'$ and $\cO'_2$ as $n\to\infty$.
The claimed properties of $Q_1$ and $Q_2$ are now easy to check.

\step{2}{
There is an integer $n\geq0$ such that $f^n(\partial^uQ_1)$ meets both $Q_2$ and $M\setminus\ov{Q_2}$.
}

Since $\Lambda$ is transitive, there exists $n\geq0$  arbitrarily large such that $f^n(Q_1)\cap Q_2\ne\emptyset$. From Step~1, one sees that $f^n(\partial^s Q_1)\subset M\setminus\ov{Q_2}$. Since $\partial^u Q_1$ has the same endpoints as $\partial^s Q_1$, the image $f^n(\partial^u Q_1)$ meets $M\setminus\ov{Q_2}$. To conclude, we assume by contradiction that $f^n(\partial^u Q_1)\cap Q_2=\emptyset$. Thus $f^n(\partial Q_1)\cap Q_2=\emptyset$. Since $Q_2$ is connected and meets $f^n(Q_1)$, this implies that $Q_2$ is contained in $f^n(Q_1)$. In particular, $\partial^uQ_2\subset\ov{Q_2}\subset f^n(\ov{Q_1})$, contradicting $f^{-n}(\partial^uQ_2)\cap \ov{Q_1}=\emptyset$.

\step{3}{
Let $\cO_1\in\HPO(f)$ such that $\cO_1\hsim\mu_1$. Then $W^u(\cO_1)$ meets both $Q_2$ and $M\setminus\ov{Q_2}$.
}

$Q_1$ meets $\HC(\mu_1)=\HC(\cO_1)$, hence $W^u(\cO_1)\cap Q_1\ne\emptyset$. At the same time, $\HC(\mu_1)\not\subset \ov Q_1$, so $W^u(\cO_1)$ also meets $M\setminus \ov{Q_1}$. Proposition \ref{prop-common-Q} says that $W^u(\cO_1)$ accumulates on $\partial^u Q_1$ both in the case when $f$ is Kupka-Smale and in the case when  $\mu_1$ is $s$-thick. By step 2, $W^u(\cO_1)$ meets $Q_2$ and $M\setminus\overline{Q_2}$.

\step{4}{
Conclusion.
}

Let $\cO_2$ be any periodic orbit homoclinically related to $\mu_2$.
As before, $W^s(\cO_2)$ meets $Q_2$ since $Q_2\cap\HC(\mu_2)\ne\emptyset$  and $W^s(\cO_2)$ also meets $M\setminus\ov{Q_2}$ since $\HC(\cO_2)\not\subset\ov{Q_2}$. Proposition~\ref{prop-common-Q} then gives $\cO_1\preceq\cO_2$. The theorem follows.

\medskip
We now explain Remark 6.5.
Suppose all that we know on $f$ and $\mu_i$ is that for some $\cO_i\hsim\mu_i$, all the intersections of $W^u(\cO_1), W^s(\cO_2)$ are transverse. Since  $\HC(\mu_i)=\ov{W^u(\cO_i)\pitchfork W^s(\cO_i)}$, we can approximate the $Q_i$ from step 1 by $su$-quadrilaterals  $Q_i'$ with  sides in $W^u(\cO_i),W^s(\cO_i)$ which are $C^0$-close to the sides of $Q_i$. By step 2, for some $n$, $f^n(\partial^u Q_1)$ meets both $Q_2$ and $M\setminus\ov{Q}_2$. Make the approximation good enough that  $f^n(\partial^u Q_1')$ meets
 $Q_2'$ and $M\setminus\ov{Q_2'}$. By Jordan's theorem, $W^u(\cO_1)$ intersects $W^s(\cO_2)$. By the assumption on $\cO_i$ the intersection is transverse. So $\mu_1\preceq\cO_1\preceq\cO_2\preceq\mu_2$.
\end{proof}

\subsection{Support and homoclinic relations of measures}
\label{Section-Measure/Topological-HC}
\begin{proposition}\label{p.measure-related}
Suppose $r>1$ and let $f$ be a $C^r$ diffeomorphism on a closed surface.  Let $\mu$
be a non-atomic  ergodic hyperbolic measure supported on a homoclinic class $\HC(\cO)$. Suppose
\begin{myitemize}
  \item[--] $f$ is Kupka-Smale, or
  \item[--] $\cO$ and $\mu$ are both $s$-thick, or
  \item[--] $\cO$ is thick.
 \end{myitemize}
Then $\mu$ is homoclinically related to $\cO$.
\end{proposition}
\begin{proof}
We apply Theorem~\ref{thmTransitivity} to the transitive set $\Lambda=\HC(\cO)$ and to the
the measures $\mu_\cO,\mu$ and get $\cO\preceq\mu$.
Replacing $f$ by $f^{-1}$ (if $f$ is Kupka-Smale or when $\cO$ is thick), or exchanging $\mu$ and $\cO$
(if $\cO$ and $\mu$ are both $s$-thick), we get $\mu\preceq\cO$.
\end{proof}

\begin{corollary}\label{c.measure-related}
Suppose $r>1$ and  $f$ is a $C^r$ diffeomorphism on a closed surface. Let
$\HC(\cO)$ be a homoclinic class  and suppose $\mu\in\HPM(f)$ satisfies at least one of the following conditions:
 \begin{enumerate}[(1)]
  \item $\mu$ non-atomic and $h_\top(f,\HC(\cO))>\lambda_{\max}(f)/r$ (see \eqref{lambda-max-lambda-min});
  \item $h(f,\mu)>\lambda_{\min}(f)/r$ and $\cO\hsim\nu$ with $h(f,\nu)>\lambda_{\min}(f)/r$
  (see \eqref{lambda-max-lambda-min});
 \end{enumerate}
Then:
 $
    \mu(\HC(\cO))=1 \text{ if and only if }\mu\hsim\cO.
 $
\end{corollary}

\begin{proof}
If $\mu\hsim\cO$, then $\mu(\HC(\cO))=1$ by Corollary \ref{Cor-HC} and the transitivity of the homoclinic relation.

We prove the converse. In case (1), the variational principle gives an ergodic $\nu$ carried by $\HC(\cO)$ with $h(f,\nu)>\lambdamax(f)/r$.  So $\nu$ is thick. For any $\nu'\in\HPM(f|_{\HC(\cO)})$, we apply Theorem \ref{thmTransitivity} to $(f,\HC(\cO),\nu,\nu')$ and $(f^{-1},\HC(\cO),\nu,\nu')$ and get $\nu\hsim\nu'$. Taking $\nu'=\mu$ and $\nu'=\mu_\cO$, the transitivity of the relation $\hsim$
gives $\mu\hsim\cO$.
In case (2), if $\lambda_{\min}(f)=\lambda^s(f)$, then  $\cO$ and $\mu$ are $s$-thick and we conclude by Proposition~\ref{p.measure-related}. If $\lambda_{\min}(f)=\lambda^u(f)$, then $\lambda_{\min}(f^{-1})=\lambda^s(f^{-1})$ and conclude as before.
\end{proof}
\subsection{The intersection of different homoclinic classes}\label{Section-Intersectio-HC}
\begin{prop}\label{c.classes-intersection}
Suppose $r>1$ and $f$ is a $C^r$ diffeomorphism on a closed surface. Let $\HC(\cO_1)$, $\HC(\cO_2)$ be two homoclinic classes
such that $\cO_1\not\hsim\cO_2$.
Then $h_\top(f, \HC(\cO_1)\cap\HC(\cO_2))\leq \lambdamax(f)/r$.
If $f$ is  Kupka-Smale, or if $\cO_1,\cO_2$ are both $s$-thick, then $h_\top(f, \HC(\cO_1)\cap\HC(\cO_2))=0$.
\end{prop}
\begin{proof}
Assume by contradiction that $h_\top(f, \HC(\cO_1)\cap\HC(\cO_2))>\lambdamax(f)/r$, then by the variational principle, there exists $\mu_1\in \Proberg(f)$ carried by $\HC(\cO_1)\cap\HC(\cO_2)$ such that $h(f, \mu_1)> \lambdamax(f)/r$. Necessarily $\mu_1$ is non-atomic, hyperbolic and thick.

Since $\mu_1,\mu_{\cO_i}$ satisfy the assumptions of Theorem~\ref{thmTransitivity} with $\Lambda:=\HC(\mu_1)$, we have $\mu_1\preceq\mu_{\cO_i}$. The same theorem, this time applied to the diffeomorphism $f^{-1}$, gives $\mu_1\succeq\mu_{\cO_i}$. So
$\cO_i\hsim\mu_1$ for $i=1,2$, whence
$\cO_1\hsim\cO_2$. But this contradicts our assumption.

Next suppose $f$ is Kupka-Smale or $\cO_1,\cO_2$ are both $s$-thick, and $h_{\top}(f,\HC(\cO_1)\cap\HC(\cO_2))>0$.
By the variational principle, there is $\nu\in\Proberg(f|_{\HC(\cO_1)\cap\HC(\cO_2)})$ with positive entropy. The set $\Lambda:=\supp(\nu)$ is transitive, $\Lambda\subset \HC(\cO_1)\cap\HC(\cO_2)$, and  $h_\top(f,\Lambda)>0$. By Theorem~\ref{thmTransitivity}, $\cO_1\hsim \cO_2$.
\end{proof}

Different topological homoclinic classes may intersect, for instance along a periodic orbit or an invariant separating circle:
An example may be built by surgery, using the techniques of~\cite{katok-bernoulli-1979}.
It seems difficult to build more complicated intersections, even in low regularity:
we do not have any example of two distinct homoclinic classes of a $C^1$ surface diffeomorphism whose intersection has positive entropy.

\subsection{Equilibrium states}
We have already stated some properties of hyperbolic equilibrium states, see Corollary~\ref{c-local-uniqueness}.
Using the dynamical Sard theorem \ref{t.DynSard} and dimension estimates, we obtain further local uniqueness properties.

\begin{proposition}\label{p-transitive-equ}
Suppose $r>1$ and let $f$ be a $C^r$ diffeomorphism on a closed surface.
Let $\mu^1,\mu^2$ be two non-atomic hyperbolic ergodic equilibrium measures for some admissible potential $\phi:\Lambda\to\RR\cup\{-\infty\}$. Assume that $\mu_1,\mu_2$ are carried by the same transitive set $\Lambda$, and that
 \begin{enumerate}[(i)]
  \item $f$ is Kupka-Smale; or
  \item $\mu^1,\mu^2$ are both $s$-thick or both $u$-thick (for instance their entropy is larger than $\lambdamin(f)/r$); or
  \item $\mu^1$ is thick; or
  \item $h_\top(f,\Lambda)>\lambdamax(f)/r$.
\end{enumerate}
Then $\mu^1=\mu^2$.
\end{proposition}
\begin{proof}
By corollary \ref{c-local-uniqueness}, it is enough to show that $\mu^1\hsim\mu^2$.

In the first three  cases, we apply Theorem~\ref{thmTransitivity} twice to the transitive set $\Lambda$,
and the measures $\mu_1$ and $\mu_2$: one gets $\mu_1\hsim \mu_2$.
In the fourth case, we consider an ergodic measure $\nu$ supported on $\Lambda$
such that $h(f,\nu)>\lambdamax(f)/r$. Case (1) of Corollary~\ref{c.measure-related} applied to $\HC(\nu)$ implies that $\mu_1\hsim \nu\hsim \mu_2$.
\end{proof}

\subsection{The size of the coded set $\pi(\Sigma^\#)$ in Theorem \ref{Theorem-Symbolic-Dynamics-C^r}}

\begin{theorem}\label{thm-coding-HC1}
Suppose $r>1$ and  $f$ is a $C^r$ diffeomorphism on a closed surface. Let $\mu$ be an ergodic hyperbolic measure for $f$.  Fix $\chi>0$ and let $\pi:\Sigma\to M$ be the coding given by Theorem \ref{Theorem-Symbolic-Dynamics-C^r}. Suppose
\begin{myitemize}
\item[--] $f$ is Kupka-Smale; or
\item[--] $h_\top(f,\HC(\mu))>\lambdamax(f)/r$; or
\item[--] $\mu$ is thick.
\end{myitemize}
Then $\overline{\pi(\Sigma)}=\overline{\pi(\Sigma^\#)}=\HC(\mu)$ and $\nu(\pi(\Sigma^\#))=1$ for each $\chi$-hyperbolic \emph{non-atomic} ergodic measure $\nu$ supported on $\HC(\mu)$.
In particular, every $\nu\in\Proberg(f)$ carried by $\HC(\mu)$ with $h(f,\nu)>\chi$ is carried by $\pi(\Sigma^\#)$.
\end{theorem}
\begin{proof} Find $\cO\in\HPO(f)$ such that $\HC(\mu)=\HC(\cO)$.
By Proposition~\ref{p.measure-related} (when $f$ is Kupka-Smale or if $\mu$, whence $\cO$, is thick) and Corollary~\ref{c.measure-related}
(when $h_\top(f,\HC(\cO))>\lambdamax(f)/r$),
any non-atomic hyperbolic ergodic measure supported on $\HC(\mu)$ is homoclinically related to $\mu$.
The theorem then follows immediately from Theorem~\ref{Theorem-Symbolic-Dynamics-C^r}.
\end{proof}

\medskip

\begin{theorem}\label{thm-coding-HC2}
Suppose $r>1$ and  $f$ is a $C^r$ diffeomorphism on a closed surface. Let
$\mu$ be an \textbf{s-thick} ergodic hyperbolic measure. Fix $\chi>0$ and let $\pi:\Sigma\to M$ be the coding given by Theorem \ref{Theorem-Symbolic-Dynamics-C^r}. Then
$\overline{\pi(\Sigma)}=\overline{\pi(\Sigma^\#)}=\HC(\mu)$ and $\nu(\pi(\Sigma^\#))=1$ for each $\chi$-hyperbolic \textbf{s-thick} ergodic measure $\nu$
carried by $\HC(\mu)$.
In particular, if $0<\chi\leq\lambda^s(f)/r$, then any ergodic measure carried by
 $\HC(\mu)$ with entropy larger than $\lambda^s(f)/r$ is carried by $\pi(\Sigma^\#)$.
\end{theorem}
\begin{proof}
Again the proof uses Proposition~\ref{p.measure-related}
($\cO$ and $\nu$ are both $s$-thick) and Theorem~\ref{Theorem-Symbolic-Dynamics-C^r}.
\end{proof}

\subsection{Spectral decomposition and periods for $C^r$ diffeomorphisms}

\begin{theorem}[Thick Spectral decomposition for $C^r$ diffeomorphisms]\label{t-spectral-revised1}
Suppose $r>1$ and  $f$ is a $C^r$ diffeomorphism on a closed surface. Consider a maximal family $\{\cO_i\}_{i\in I}$ of
\textbf{s-thick} hyperbolic periodic orbits such that: $\cO_i\hsim\cO_j\implies i=j$ for any $i,j\in I$. Then:
 \begin{enumerate}[(1)]
  \item\label{item-union} $\mu\left(\bigcup_{i\in I}\HC(\cO_i)\right)=1$ for every \textbf{s-thick} ergodic measure $\mu$, and $\mu\hsim\cO_i$ for some $i\in I$.
  \item\label{item-inter} $h_\top(f,\HC(\cO_i)\cap \HC(\cO_j))=0$ for any pair of distinct $i,j\in I$.
  \item\label{item-period} Let $\ell_i:=\gcd(\{\operatorname{Card}(\cO'):\;\cO'\hsim\cO_i\})$ and set $A_i:=\overline{W^s(p)\pitchfork W^u(p)}$ for some $p\in\cO_i$. Then:
\begin{myitemize}
\item[--] $\HC(\cO_i)=A_i\cup f(A_i)\cup\dots\cup f^{\ell_i-1}A_i$ and $f^{\ell_i}A_i=A_i$,
\item[--] $f^{\ell_i}:A_i\to A_i$ is topologically mixing,
\item[--] $f^j(A_i)\cap A_i$ has empty {relative} interior in $\HC(\cO_i)$ and zero topological entropy if $0<j<\ell_i$.
\end{myitemize}
 \item\label{item-fin} For any $\chi>\lambda^s(f)/r$, the set of $i\in I$ such that $h_\top(f,\HC(\cO_i))> \chi$ is finite.
 \item\label{item-trans}
If  $\Lambda$ is a transitive invariant compact set, there is at most one $i\in I $ such that $\HC(\cO_i)\subset\Lambda$. If, additionally, $f|_\Lambda$ is topologically mixing, the unique $\HC(\cO_i)\subset\Lambda$, if it exists, has period $\ell_i=1$.
\end{enumerate}
\end{theorem}
\medskip

Note that the homoclinic class of a periodic hyperbolic orbit is either finite (and reduced a single orbit), or infinite and  has positive topological entropy.
\begin{theorem}[Spectral decomposition for Kupka-Smale diffeomorphisms]\label{t-spectral-revised2}
Let $f$ be a Kupka-Smale $C^r$ diffeomorphism, $r>1$, on a closed surface and consider a maximal family $\{\cO_i\}_{i\in I}$ of
hyperbolic periodic orbits such that: $\cO_i\hsim\cO_j\implies i=j$ for any $i,j\in I$. Then:
 \begin{enumerate}[(1)]
  \item[(1)] $\mu\left(\bigcup_{i\in I}\HC(\cO_i)\right)=1$ for any hyperbolic ergodic measure $\mu$: more precisely, $\mu\hsim\cO_i$ for some $i\in I$.
  \item[(2-3)] The properties 2 and 3 of Theorem~\ref{t-spectral-revised1} are satisfied by $\{\cO_i\}_{i\in I}$.
 \item[(4)] For any $\chi>h^*(f)$, the set of $i\in I$ such that $h_\top(f,\HC(\cO_i))> \chi$ is finite.
  \item[(5)] If  $\Lambda$ is a transitive invariant compact set, there is at most one $i\in I $ such that $\HC(\cO_i)$ is infinite and included in $\Lambda$. If, additionally, $f|_\Lambda$ is topologically mixing, then the unique infinite $\HC(\cO_i)$, if it exists, has period $\ell_i=1$.
  \end{enumerate}
\end{theorem}

\begin{proof}[Proof of Theorems~\ref{t-spectral-revised1} and~\ref{t-spectral-revised2}]
Item \eqref{item-union} is a consequence of Katok's horseshoe theorem (Corollary~\ref{Cor-HC}): If $\mu$ is a non-atomic ergodic measure,
there exists $\cO\in\HPO(f)$ such that $\mu\hsim \cO$ and $\mu$ is supported on $\HC(\cO)$.
By definition, $\cO$ is $s$-thick if $\mu$ is. Item \eqref{item-inter} is because of Proposition \ref{c.classes-intersection}.

Now let us consider some $\cO\in \HPO(f)$ and let $\ell=\gcd(\{\operatorname{Card}(\cO'):\;\cO'\hsim\cO\})$ be the period of the homoclinic class of $\cO$. We will work with the diffeomorphism $F:=f^\ell$.

Let $A:=\overline{W^s(p)\pitchfork W^u(p)}$ for some $p\in\cO$.
Proposition~\ref{p-hc-cyclic} gives most of Item (3). What remains to be shown is that $f^j(A)\cap A$ has zero entropy and empty interior relatively to $\HC(\cO)$ for $0<j<\ell$.
We fix $0<j<\ell$ and denote by $\cO^0$ and $\cO^j\subset\cO$ the $F$-orbits of $p$ and $f^j(p)$.
From Proposition~\ref{p-hc-cyclic}, the sets $A$ and $f^j(A)$ are the $F$-homoclinic classes of $\cO^0$ and $\cO^j$.

Since $f$ is Kupka-Smale, $F$ is Kupka-Smale, and if $\cO$ is $s$-thick for $f$, then $\cO^0,\cO^j$ are $s$-thick for $F$. So
if $h_{\top}(F,f^j(A)\cap A)>0$, then Proposition \ref{c.classes-intersection}  implies that $\cO^0,\cO^j$ are $F$-homoclinically related. So $W^s(p)\pitchfork W^u(f^{j+k\ell}p)\ne\emptyset$ for some integer $k$, in contradiction to Proposition~\ref{p-hc-cyclic}.(3).

If $f^j(A)\cap A$ has non-empty interior in $\HC(\cO)$, then Proposition~\ref{p-hc-cyclic} claims that $f^j(A)=A$, whence $h_{\top}(F,f^j(A)\cap A)=h_{\top}(F,A)>0$, and we obtain a contradiction as before. Item \eqref{item-period} follows.
\medskip

Let us pick $\chi$ larger than $\lambdamin(f)/r$, or larger than the tail entropy $h^\ast(f)$ when $f$ is Kupka-Smale.
Let $J\subset I$ be the set of indices $i$ such that $h_\top(f,\HC(\cO_i))\geq \chi$.
For each $i\in J$, there exists an ergodic measure $\mu_i$ supported on $\HC(\cO_i)$ with entropy larger than $\chi$.
Note that $\mu_i$ is $s$-thick when $\chi>\lambdamin(f)/r$.
Hence  either  $\cO_i$ and $\mu_i$ are both $s$-thick, or $f$ is Kupka-Smale. In both cases,  $\cO_i\hsim \mu_i$
by Proposition~\ref{p.measure-related}.
Since $\cO_i$ are pairwise non-homoclinically related,
 $\mu_i$ are pairwise non-homoclinically related ($i\in J$).
Theorem \ref{thmEntropyDecay} therefore implies that $J$ is finite, proving item \eqref{item-fin}.
\medskip

We turn to item \eqref{item-trans}. Assume
$\cO,\cO'\in \HPO(f)$ have {infinite} homoclinic classes included in some transitive compact set $\Lambda$. {Infinite} homoclinic classes contain transverse homoclinic intersections, therefore they have positive topological entropy.
If $f$ is Kupka-Smale, or if both $\cO$ and $\cO'$ are $s$-thick, then we can apply Theorem \ref{thmTransitivity}
and conclude that $\cO\hsim \cO'$.

If $f|_{\Lambda}$ is topologically mixing and has an  {infinite} homoclinic class $\HC(\cO)$,
we decompose it: $\HC(\cO)=A\cup f(A)\cup\dots\cup f^{\ell-1}(A)$ as in item \eqref{item-period}.
Note  that the sets $A,f(A),\dots,f^{\ell-1}(A)$ are $\ell$ distinct homoclinic classes included in $\Lambda$ for $f^\ell$.
Since $f^\ell|_{\Lambda}$ is topologically transitive, we have uniqueness so $\ell=1$, proving the item.
\end{proof}

\section{Proof of the Main Theorems}\label{s-proof-main-theorems}

\subsection{$C^\infty$ diffeomorphisms}

\noindent
{\bf Theorem~\ref{mainthm-Spectral}.}
This theorem follows from Theorem~\ref{t-spectral-revised1} and the simple observation that if $f\in\Diff^\infty(M)$ and $\dim M=2$ then every $\cO\in\HPO(f)$ such that $h_{\top}(f,\HC(\cO))>0$ is thick.

\medskip
\noindent
{\bf Theorem~\ref{mainthm-class}.} Suppose $\HC(\cO)$ has positive topological entropy.

By Corollary~\ref{c.measure-related}, part (1), any hyperbolic non-atomic ergodic measure carried by $\HC(\cO)$ must be homoclinically related to $\cO$.  In dimension two, every  ergodic measure with positive entropy is hyperbolic and non-atomic, so this proves the second part of the Theorem.

 If $f\in \Diff^\infty(M)$, then $\mu\mapsto h(f,\mu)$ is an upper semi-continuous function  with respect to the weak-$*$ topology on the compact space of invariant probability measures carried by $\HC(\cO)$ \cite{Newhouse-Entropy}. Therefore $h(f,\mu)$  attains its maximum at some measure $\mu$ carried by $\HC(\cO)$. By the variational principle,  $h(f,\mu)=h_{\top}(f,\HC(\cO))$.

The entropy of a measure is an average of the entropies of its ergodic components. It follows that a.e. ergodic component of $\mu$ also has entropy $h_{\top}(f,\HC(\cO))$. So  without loss of generality, $\mu$ is ergodic.
Since $\mu$ is ergodic with positive entropy, by the beginning of the proof, $\mu$ is homoclinically related to $\cO$.
Therefore, by Corollary~\ref{c-local-uniqueness}, $\mu$ is the unique m.m.e. for $f|_{\HC(\cO)}$, has full support in $\HC(\cO)$, and is isomorphic to the product of a Bernoulli scheme and the cyclic permutation of order $\ell:=\gcd\{\Card(\cO'):\cO'\hsim\cO\}$.
Since $\mu$ is thick, we can apply Theorem~\ref{t-spectral-revised1} part (5) to see that  if $\HC(\cO)$ is contained in some topologically mixing compact invariant set then $\ell=1$. This finishes the proof of the first item of Theorem~\ref{mainthm-class}.

The third item follows from Theorem~\ref{thm-coding-HC1}.

\medskip
\noindent
{\bf Main Theorem (page \pageref{Page-Main-Theorem}).} By Theorems~\ref{mainthm-Spectral} and~\ref{mainthm-class}, the number of ergodic measures of maximal entropy is less than or equal to the number of homoclinic classes with full topological entropy. By Theorem~\ref{mainthm-Spectral}, this number is finite, and in the topologically transitive case equal to one.

In the topologically mixing case the unique measure of maximal entropy is Bernoulli, because of Theorem~\ref{mainthm-class}, part 1, and Theorem~\ref{mainthm-Spectral}, part 5.

\medskip
\noindent
{\bf Corollary~\ref{Cor-finite-eq-C-infty}.} The proof is the same as the proof of Theorem 2.

\medskip
\noindent
{\bf Corollary~\ref{c.attractor}.} So far all our surfaces were assumed to be  closed surfaces without boundary. But since  all our work took place inside a small neighborhood of a homoclinic class,  our methods also apply to diffeomorphisms on more general surfaces with compact global attractors.

One can reduce to the case of a closed surface in the following way: the attractor $\Lambda$ is contained in
an open set $U$ disjoint from the boundary of the surface, such that $f(\overline U)\subset U$.
Then, using~\cite[Proposition 3.3]{ABCD},
one can identify $U$ with the open set of a closed surface $\widetilde M$ and find a diffeomorphism $\widetilde f$
of $\widetilde M$ which coincides with $f$ on $U$ and such that:
\begin{myitemize}
\item[--] $\widetilde M\setminus U$ has finitely many connected components, each homeomorphic to the $2$-disc,
\item[--] the intersection of the non-wandering set of $\widetilde f$ with any connected component of $\widetilde M\setminus U$
is a repelling periodic point.
\end{myitemize}
Now Theorems~\ref{t.katok} and \ref{thmTransitivity} show that $\Lambda$ contains exactly one {infinite} homoclinic class
$\HC(\cO)$ and that every ergodic measure of maximal entropy is homoclinically related to $\cO$.
By  Corollary~\ref{c-local-uniqueness}, the measure of maximal entropy of $\widetilde f$ is unique.
By definition of $\widetilde f$, the measures of maximal entropy of $f$ and $\widetilde f$ coincide, hence $f$
has a unique measure of maximal entropy.

\subsection{$C^r$ diffeomorphisms}

\noindent
{\bf Main Theorem Revisited (page~\pageref{MainTheoremRevisited})}.
To prove the first part we assume without loss of generality that $\lambda_{\min}(f)=\lambda^s(f)$. Otherwise replace $f$ by $f^{-1}$ noting that the statement we are trying to prove does not change, whereas $\lambda^u(f)=\lambda^s(f^{-1})$. We fix $\chi>\lambda_{\min}(f)/r=\lambda^s(f)/r$ and let $\mathcal E$ be the set of ergodic equilibrium measures $\mu$ with $h(f,\mu)>\chi$. We must check that $\mathcal E$ is finite.

Every  measure $\mu\in\mathcal E$ is $s$-thick, because $\delta^s(\mu)=h(f,\mu)/\lambda^s(\mu)>\chi/\lambda^s(f)>1/r$ by choice of $\chi$.
Let $\cO_i$, $i\in I$, be a maximal collection of pairwise non homoclinically related hyperbolic periodic orbits which are $s$-thick. Let $J :=\{i\in I:h_\top(f,\HC(\cO_i))>\chi\}$. By part (4) of Theorem~\ref{t-spectral-revised1}, the set $J$ is finite, and  by part (1) of Theorem~\ref{t-spectral-revised1} for each $\mu\in\mathcal E$ there is  $i(\mu)\in J$ such that $\mu\hsim\cO_{i(\mu)}$.  The map $\mu\mapsto i(\mu)$  is injective by  Corollary~\ref{c-local-uniqueness}. Hence $\mathcal E$ is finite.

For the second part, note that Proposition~\ref{p-transitive-equ} implies that there is at most one $s$-thick equilibrium measure supported by any given transitive compact set $\Lambda$. Thus there is at most one ergodic equilibrium measure $\mu$ on $\Lambda$ with $\delta^s(\mu)>1/r$ (and at most one ergodic equilibrium measure $\mu$ on $\Lambda$ with $\delta^u(\mu)>1/r$).
}

\medskip
\noindent
{\bf Corollary \ref{Cor-finite-r}.} As in the previous proof, there is no loss of generality in assuming that $\lambda_{\min}(f)=\lambda^s(f)$. If $\mu$ is an ergodic equilibrium measure for $\phi$ then $P_{\top}(\phi)=h(f,\mu)+\int \phi d\mu$, whence by the assumption
$P_{\top}(\phi)>\sup\phi+\frac{\lambda_{\min}(f)}{r}$,
$
h(f,\mu)>\lambda_{\min}(f)/r
$, whence $\delta^s(\mu)>1/r$.
The first part of the Main Theorem Revisited says that there can be at most finitely many ergodic equilibrium measures like that.
The second   part of the Main Theorem Revisited says that if $f$ is topologically transitive, then there can be at most one ergodic equilibrium measure like that.

\medskip
\noindent
{\bf Corollary ~\ref{c.HHTU}.}
As explained in the introduction, an ergodic SRB for a $C^r$ diffeomorphism $f$  $(r>1)$ is a hyperbolic equilibrium measure for the admissible
potential $\phi:=-\log\|Df|_{E^u}\|$. The entropy formula for SRB measures says that $\delta^u=1$ \cite{Ledrappier-Strelcyn}, so SRB measures are  always $u$-thick.
The Main Theorem Revisited then implies that each transitive compact set $\Lambda$ can support at most one such measure and its support is some homoclinic class $\HC(\cO)$.
Since $\mu$ is $u$-thick, Theorem~\ref{t-spectral-revised1} applies. By part (5) of this theorem, if $\HC(\cO)$ is contained in a topologically mixing invariant compact set, then the period $\gcd(\{|\cO'|:\cO'\hsim\cO\})$ is $1$. By Corollary~\ref{c-local-uniqueness}, $\mu$ is Bernoulli.
This proves Corollary~\ref{c.HHTU}.

\subsection{Borel classification (Theorem \ref{thmBorelConj} and Corollary \ref{Corollary-Classification-mixing})}\label{s-Borel}

 A {\em Borel space} is  a pair $(X,\mathcal B)$ where $X$ is a set, and $\mathcal B$ is a $\sigma$-algebra of subsets of $X$. Elements of $\mathcal B$ are called {\em measurable sets}.  A {\em standard Borel space} is a Borel space $(X,\mathcal B)$ such that
{there exists a metric $d$ on $X$ which makes $(X,d)$ a complete separable metric space}, and $\mathcal B$ is the smallest $\sigma$-algebra which contains all open balls in $(X,d)$. In this case elements of $\mathcal B$ are called {\em Borel sets}. See \cite{Kechris-book} for background.

 An {\em isomorphism} of Borel spaces $(X_1,\mathcal B_1)$, $(X_2,\mathcal B_2)$ is a invertible map $\psi:X_1\to X_2$ such that for all $E\in \mathcal B_1$, $\psi(E)\in \mathcal B_2$, and for all $E\in\mathcal B_2$, $\psi^{-1}(E)\in\mathcal B_1$. If $(X_1,\mathcal B_1)=(X_2,\mathcal B_2)$ we call $\psi$ a  {\em Borel automorphism}.

We classify the dynamics of homoclinic classes up to (partial) Borel conjugacies defined as follows (see also \cite{Weiss1984,Weiss1989,Hochman1,Buzzi-Lausanne,Boyle-Buzzi}).

\begin{definition}
Suppose $(X_1,\mathcal B_1), (X_2,\mathcal B_2)$ are Borel spaces, and  $f_1:X_1\to X_1$ and $f_2:X_2\to X_2$ are Borel automorphisms.
\begin{enumerate}[(1)]
\item A subset $Y_i $ of $X_i$ is \emph{almost full} if it is measurable and carries all atomless $f_i$-invariant and ergodic Borel probability measure of $(X_i,\mathcal B_i)$. The complement of an almost full set is called \emph{almost null}.
\item We call $f_1,f_2$  \emph{almost Borel conjugate} if there are $f_i$-invariant measurable $X_i^0\subset X_i$ which carry all atomless $f_i$-invariant ergodic probability measures and  an isomorphism $\psi:(X_1^0,\mathcal B_1\cap X_1^0)\to (X_2^0,\mathcal B_2\cap X_2^0) $  such that $f_2\circ\psi=\psi\circ f_1$. Here $\mathcal B_i\cap X_i^0:=\{E\cap X_i^0: E\in\mathcal B_i\}$.
\item We say that there is an  \emph{almost Borel embedding} of $(X_1,f_1)$ into $(X_2,f_2)$ if there is an $f_2$-invariant measurable
$Z_2\subset X_2$ such that $f_1$, $f_2|_{Z_2}$ are almost Borel conjugate. Here $Z_2$ is viewed as a Borel space with sigma-algebra $\mathcal B_2\cap Z_2$.

\item We call $f_1$,$f_2$ \emph{Borel conjugate modulo zero entropy} if there exist invariant measurable $X_i^0\subset X_i$ which
carry all the positive entropy ergodic invariant Borel probability measures for $f_i$, and an isomorphism $\psi:(X_1^0,\mathcal B_1\cap X_1^0)\to (X_2^0,\mathcal B_2\cap X_2^0)$ such that $f_2\circ\psi=\psi\circ f_1$.

\end{enumerate}
A measurable subset which carries  all atomless invariant and ergodic Borel probability measures as in (1) is called \emph{almost full}. The complement of an almost full set is called \emph{almost null}.
\end{definition}
\medskip

Let $f$ be a $C^r$ diffeomorphism of a closed manifold $M$ $(r>1)$ and let $\cO$ be a hyperbolic periodic orbit of saddle type. Recall the set $Y'$ of regular points introduced in section \ref{s.hyp-measures} and the set
$$
H_{\cO}:=\{x\in Y': W^u(x)\pitchfork W^s(\cO)\neq\varnothing\text{ and }W^s(x)\pitchfork W^u(\cO)\neq\varnothing\}
$$
from
Proposition~\ref{prop-HcO}. We also need the following invariants of $\HC(\cO)$:
\begin{definition}
Let $\cO$ be a hyperbolic periodic orbit of saddle type.
\begin{myitemize}
\item[--] The  \emph{period} {of the homoclinic class of $\cO$} is $\ell(\cO):=\gcd\{\Card(\cO'):\cO'\hsim\cO\}$,
\item[--] The \emph{entropy} {of the homoclinic class of $\cO$} is  $h(\cO):=\sup\{h(f,\mu):\mu\hsim\cO\}$,
\item[--] The \emph{{multiplicity}} {of the homoclinic class of $\cO$} is $m(\cO):=\#\{\mu\hsim\cO: h(f,\mu)=h(\cO)\}$.
\end{myitemize}
\end{definition}

Note that the entropy $h(\cO)$ is zero only when $\cO$ is finite.
Note also (see~\cite{Ruette-Note-2015}) that for any $h(\cO)>0$, $\ell\geq 1$ and $m(\cO)\in \{0,1\}$,
there exists an irreducible countable state Markov shift $\Sigma$  with Gurevi\v{c} entropy (recall \eqref{eq-hG}) equal to $h(\cO)$,
period $\ell(\cO)$, and exactly $m(\cO)$ ergodic m.m.e.'s.
\medskip

From now on,  $M$ is a closed surface. In this case, $m(\cO)\in \{0,1\}$ by Corollary~\ref{c-local-uniqueness}.
Moreover, there exists an almost Borel conjugacy between $H_\cO$ and any an irreducible countable state Markov shift
with the same invariants.
\medskip

\begin{theorem}\label{thm-Borel-HcO}
Let $f$ be a $C^r$ diffeomorphism $(r>1)$ on a closed surface and let $\cO$ be a hyperbolic periodic orbit of saddle type.
Let $\Sigma$ be an irreducible countable state Markov shift with period $\ell(\cO)$, entropy $h(\cO)$ and with $m(\cO)$ ergodic m.m.e.'s.
Then $(f,H_\cO)$ is almost Borel conjugate to $(\sigma,\Sigma)$.
\end{theorem}

The following notion and facts will be useful (see \cite[Sec. 1.3]{Hochman1} and \cite{Buzzi-Lausanne} for more background). Recall the notion of period of an ergodic measure (see Proposition~\ref{p.period}). A Borel automorphism $S:X\to X$ is \emph{$(t,p)$-universal} \cite[p.~2748]{Boyle-Buzzi} for some real number $t>0$ and integer $p\geq1$, if any automorphism $T:Y\to Y$ of a standard Borel space satisfying
 \begin{equation}\label{eq-tp-class}
    \forall \mu\in\Proberg(T),\quad h(f,\mu)<t \text{ and $p$ is a period of }(T,\mu),
  \end{equation}
has an almost Borel embedding into $(S,X)$.
{We say that $(S,X)$ is \emph{strictly $(t,p)$-universal} if it is $(t,p)$-universal
and satisfies~\eqref{eq-tp-class}.}
Given $t>0$, the \emph{$t$-slice $X_t$} of an automorphism $T$ of a standard Borel space $X$ is any invariant measurable $X_t\subset X$ such that:
 $$
   \forall \mu\in\Proberg(T)\quad \mu(X_t)=1 \iff h(T,\mu)<t.
 $$
The $t$-slice always exists and is unique up to an almost null set (see, e.g., \cite[ Prop. 2.9~and~2.12]{Buzzi-Lausanne}).
We will use the following variant of \cite[Sect. 4]{Boyle-Buzzi}.

\begin{lemma}\label{lem-ex-univ}
For any number $t>0$ and integer $p\geq1$, the following are  $(t,p)$-universal systems:
 \begin{enumerate}[(i)]
  \item any irreducible countable state Markov shift of period $p$ and entropy $t$;
  \item any automorphism of a standard Borel space containing almost Borel conjugate copies of horseshoes of period $p$ and entropy $h$ arbitrarily close to $t$.
 \end{enumerate}
\end{lemma}
\begin{proof}
Item (i) follows from  \cite[Prop.~4.2(1)]{Boyle-Buzzi} since by definition, systems that contain strictly $(t,p)$-universal systems are $(t,p)$-universal.
For item (ii), note using \cite{Hochman1} that a horseshoe of period $p$ and entropy $h$ contains a strictly $(h,p)$-universal system. The same must hold for any almost Borel conjugate system. Now,  \cite[Lemma~3.2]{Boyle-Buzzi} shows that the whole system is $(t,p)$-universal.
\end{proof}

\begin{proof}[Proof of Theorem~\ref{thm-Borel-HcO}]
We assume that $H(\cO)$ is infinite since the claims are otherwise trivial.
Let $H_\cO^0$ and $\Sigma^0$, be the $h(\cO)$-slices of $H_\cO$ and $\Sigma$. Note that
$$h(\cO)=\sup_{\mu\in\Prob(f|_{H_\cO})} h(f,\mu)=\sup_{\nu\in\Prob(\Sigma)} h(\sigma,\nu).$$

We first claim that the systems $(f,H_\cO\setminus H_\cO^0)$ and $(\sigma,\Sigma\setminus\Sigma^0)$ are almost Borel conjugate.
Indeed they  carry the m.m.e.'s of $H_\cO$ and $\Sigma$ and no other invariant probability measures.
When $m(\cO)=0$, there are  no such measures and the claim holds trivially.
When $m(\cO)=1$, Corollary~\ref{c-local-uniqueness} ensures that the m.m.e. of $H_\cO$ is isomorphic to a Bernoulli scheme times the cyclic permutation of order $\ell(\cO)$ and therefore is measure preservingly conjugate to  $\Sigma$ with its m.m.e. This gives an almost Borel conjugacy between $(f,H_\cO\setminus H_\cO^0)$ and $(\sigma,\Sigma\setminus\Sigma^0)$.

It  thus suffices to show that $(f,H_\cO^0)$ and $(\sigma,\Sigma^0)$ are almost Borel conjugate. By Hochman's dynamical Cantor-Bernstein argument \cite{Hochman1}, it is sufficient to show that there are almost Borel embeddings
$(f,H_\cO^0)\hookrightarrow(\sigma,\Sigma^0)$ and $(\sigma,\Sigma^0)\hookrightarrow (f,H_\cO^0)$.

\medbreak

By Proposition~\ref{p.period}, $\ell(\cO)$ is a period of each $\mu\in\Proberg(f|_{H_\cO^0})$. For these measures, we also have: $h(f,\mu)<h(\cO)$. By Lemma~\ref{lem-ex-univ}, part (i), $\Sigma$ is $(h(\cO),\ell(\cO))$-universal, hence there is an almost Borel embedding of $H_\cO^0$  into $\Sigma$.
Since $\Sigma\setminus \Sigma^0$ only carries a measure with entropy $h(\cO)$, its preimage by the embedding carries no invariant measure: we thus get an almost Borel embedding of $H_\cO^0$ into $\Sigma^0$.

\medbreak

We claim that there are horseshoes $\Lambda$ with period $\ell(\cO)$ and entropy arbitrarily close to $h(\cO)$ which are contained in $H_\cO$ up to almost null sets.
To prove the claim, let $\cO_1,\dots,\cO_N$ be hyperbolic periodic orbits homoclinically related to $\cO$ and such that $\gcd\{\Card\cO_i:i=1,\dots,N\}=\ell(\cO)$. For any $h<h(\cO)$, there is $\mu\in\Proberg(f)$ with $\mu\hsim\cO$ and $h(f,\mu)>h$. By Theorem~\ref{t.katok}, there is a horseshoe $\Lambda_0$ homoclinically related to $\cO$ and with $h_\top(f,\Lambda_0)$ close to $h(f,\mu)$.
One then builds a horseshoe $\Lambda\supset\Lambda_0\cup\cO_1\cup\dots\cup\cO_N$
(see Proposition~\ref{prop-approx-HC}).
By construction, $\Lambda$ is homoclinically related to $\cO$.
From the definition of $H_\cO$, the set $\Lambda\setminus H_\cO$ has zero measure for all invariant probability measures supported on $\Lambda$,
so $\Lambda$ is almost Borel embedded in $H_\cO$. Observe that the topological entropy  of $\Lambda$ is bigger than $h$ and its period is $\ell(\cO)$. The claim is proved.

It now follows, by Lemma~\ref{lem-ex-univ}, part (ii), that $H_\cO$ is $(h(\cO),\ell(\cO))$-universal. Hence there is an almost Borel embedding of $\Sigma^0$ into $H_\cO$
(and thus $H_\cO^0$). This concludes the proof.
\end{proof}

We now consider topological homoclinic classes. We further assume $C^\infty$ smoothness
(leaving the cases with finite smoothness and Kupka-Smale property or entropy conditions to the interested reader).
We first deduce Theorem \ref{thmBorelConj} from Theorem \ref{thm-Borel-HcO} by comparing $\HC(\cO)$ to $H_\cO$.

\medskip
\noindent
\noindent {\bf Proof of Theorem~\ref{thmBorelConj}.}
Again we can  assume that $\HC(\cO)$ is infinite.
Since $f$ is $C^\infty$, Corollary~\ref{c.measure-related} implies that the symmetric difference $H_\cO\diffsym\HC(\cO)$ carries only measures with zero entropy. Now Theorem~\ref{thmBorelConj} is an immediate consequence of Theorem~\ref{thm-Borel-HcO}.\qed

\medskip
In the special case of $C^\infty$ surface diffeomorphisms, the above local analysis allows to recover the global classification first obtained in \cite{Boyle-Buzzi} {(which only required a regularity $C^{r}$ regularity for $r>1$)}. More importantly, the present approach
shows that a complete invariant of almost Borel conjugacy can be computed  from the topological entropies and periods of the infinite homoclinic classes.

\begin{corollary}\label{cor-Cinf-Borel}
Let $f$ be a $C^\infty$ diffeomorphism of a closed surface and {consider
a maximal family $\{\cO_i\}_{i\in I}$ of non-homoclinically related hyperbolic periodic orbits
with infinite homoclinic classes.}
The entropy and the periods of the homoclinic classes $(h(\cO_i),\ell(\cO_i))_{i\in I}$ determine the Borel conjugacy class of $f$ modulo zero entropy measures.

More precisely, the following double sequence $(h_k(f),m_k(f))_{k\geq1}$ is a complete invariant:
 $$\begin{aligned}
    &h_k(f):=\sup(\{h(\cO_i):\ell(\cO_i)|k\}\cup\{0\}) \text{ and }\\
    &m_k(f):=\Card(\{i\in I:\ell(\cO_i)=k\text{ and }h(\cO_i)=h_k(f)\}).
 \end{aligned}$$
\end{corollary}
\begin{proof}
For each $i\in I$ with $h(\cO_i)>0$, let $\Sigma_i$ be an irreducible Markov shift with entropy $h(\cO_i)$ and period $\ell(\cO_i)$ with a m.m.e.  (this exists by \cite{Ruette-Note-2015}).
 Since $f$ is $C^\infty$,
the upper semi-continuity~\eqref{Tail-Inequality} implies that $\HC(\cO_i)$ supports an ergodic measure $\mu_i$ with entropy
$h_\top(f,\HC(\cO_i))$.
{Moreover} $\mu_i\hsim\cO_i$ by Corollary~\ref{c.measure-related}, part 1.
The proof of Theorem~\ref{thmBorelConj}
implies that $(\sigma,\Sigma_i)$ and $(f,\HC(\cO_i))$ are Borel conjugate
modulo zero entropy measures. Theorem~\ref{mainthm-Spectral} then shows that $f$ is Borel conjugate modulo zero entropy measures to the disjoint union $\bigsqcup_{i\in I} \HC(\cO_i)$ and therefore to $\bigsqcup_{i\in I}\Sigma_i$.
To conclude, we apply \cite[Theorem~1.5]{Boyle-Buzzi} which classifies Markov shifts modulo zero entropy measures.
\end{proof}

As mentioned in the introduction, we obtain simpler and more powerful results when there is mixing. We prove the following strengthening of Corollary~\ref{Corollary-Classification-mixing}. A \emph{Borel conjugacy modulo periodic orbits} is a Borel conjugacy between the restrictions to the complement sets to the unions of periodic orbits. In general, this is much stronger than just preserving all atomless invariant probability measures,  see~\cite{Weiss1989}.

\begin{corollary}\label{Cor-Classif-mixing2}
Consider the $C^\infty$ diffeomorphisms of a closed surface that are
topologically mixing and with positive topological entropy.
\begin{enumerate}
\item Any such diffeomorphism is  Borel conjugate modulo periodic
orbits to a mixing Markov shift.
\item Any two such diffeomorphisms are Borel conjugate modulo periodic orbits
if and only if they have the same topological entropy.
\end{enumerate}
\end{corollary}

\begin{corollary}\label{Cor-Classif-mixing3}
Consider the topologically mixing topological homoclinic classes for
$C^\infty$ diffeomorphisms of closed surfaces.
\begin{enumerate}
\item Any such homoclinic class is Borel conjugate modulo periodic
orbits to a mixing Markov shift.
\item Any two such homoclinic classes  are Borel conjugate modulo
periodic orbits if and only if they have the same topological entropy.
\end{enumerate}
\end{corollary}

\medskip

\noindent {\bf Proof of Corollaries \ref{Cor-Classif-mixing2} and \ref{Cor-Classif-mixing3}.}
Let $f$ be a $C^\infty$ diffeomorphism of a closed surface. We first consider some topologically mixing
homoclinic class $\HC(\cO)$ and prove {item (1) of Corollary \ref{Cor-Classif-mixing3}}. We can assume that $\HC(\cO)$ has positive topological entropy, since otherwise it is a {fixed point} and there is nothing to show.

By Theorem~\ref{mainthm-Spectral} and topological mixing, the period $\ell(\cO)$ of the class is equal to $1$. By Newhouse's theorem, $f|\HC(\cO)$ has a m.m.e.  Theorem~\ref{thmBorelConj} yields a mixing Markov shift $\Sigma$, invariant Borel subsets $H_1\subset \HC(\cO)$, $\Sigma_1\subset\Sigma$ and a Borel conjugacy $\Sigma_1\to H_1$, such that the only ergodic invariant measures carried by $\HC(\cO)\setminus H_1$ or $\Sigma\setminus \Sigma_1$ are measures with zero entropy.

By a variant of the Cantor-Bernstein theorem of set theory (see, e.g., \cite{Hochman1}), in order to prove that
$f|\HC(\cO)$ and $\sigma|\Sigma$ are Borel conjugate modulo periodic orbits, it suffices to find two Borel embeddings $\Sigma\setminus\Per(\Sigma)\hookrightarrow \HC(\cO)$ and $\HC(\cO)\setminus\Per(f)\hookrightarrow\Sigma$
which
intertwine the actions. To do this, we follow \cite{Boyle-Buzzi-Gomez-2014} and especially the proof of Theorem~1.4 there.

By Katok's horseshoe Theorem \ref{t.katok}, $\HC(\cO)$ contains a topologically mixing horseshoe $K$ with  $0<h_\top(K)<h_\top(f)$. We fix some mixing subshift of finite type $Y$ such that $0<h_\top(Y)<h_\top(K)$
and set $Y':=Y\setminus \Per(Y)$.
Well-known facts from symbolic dynamics yield a Borel embedding $\gamma:Y'\times\{0,1,2,\ldots\}\to \HC(\cO)$.
Since $Y'\times\{0,1,2,\ldots,\}$ and $Y'\times\{1,2,3,\ldots\}$ are Borel conjugate, there is a Borel conjugacy $\HC(\cO)\to \HC(\cO)\setminus \gamma(Y'\times\{0\})$.
Combining with the embedding $\Sigma_1\to \HC(\cO)$ from Theorem~\ref{thmBorelConj}, we get a Borel embedding of $\Sigma_1$ into $\HC(\cO)\setminus \gamma(Y'\times\{0\})$.

As $\Sigma':=\Sigma\setminus (\Sigma_1\cup\Per(\Sigma))$ carries no shift-invariant finite measure with positive entropy, Hochman's generator theorem \cite{Hochman2} gives a Borel embedding $\phi:\Sigma'\to Y'\times\{0\}$.
We have thus built a Borel embedding of $\Sigma\setminus\Per(\Sigma)$ into $\HC(\cO)$ as was needed.

The converse embedding is similarly constructed (except that we do not need Katok's horseshoe theorem). Thus we get the Borel conjugacy modulo periodic orbits and item (1) of Corollary \ref{Cor-Classif-mixing3} follows.
The classification of the topologically mixing homoclinic classes now follows from the  result of \cite{Hochman2} that mixing Markov shifts with a m.m.e. and finite entropy are classified by their entropy up to Borel conjugacy modulo periodic orbits. Item (2) of Corollary \ref{Cor-Classif-mixing3}
is proved.

\medbreak
We turn to Corollary \ref{Cor-Classif-mixing2} and assume $f$ itself to be topologically mixing with $h_\top(f)>0$. Theorem~\ref{mainthm-Spectral} implies that there is exactly one homoclinic class $\HC(\cO)$ with positive topological entropy. Moreover, the topological entropy of this class is equal to $h_\top(f)$ and its period is equal to one.

Corollary \ref{Cor-Classif-mixing3} implies that $f|\HC(\cO)$ is Borel conjugate modulo periodic orbits to a mixing Markov shift with positive entropy. Since $M\setminus\HC(\cO)$ may carry only zero entropy invariant probability measures, the argument above shows that the whole of $f$ is Borel conjugate modulo periodic orbits to a Markov shift. \qed


\appendix

\section{Lipschitz holonomies}\label{s.holonomy}
The purpose of this section is to prove the following:
\begin{thm}\label{Theorem-Holonomy-is-Lipschitz}
Fix $r>1$ and suppose $f:M\to M$ is a $C^r$ diffeomorphism of a closed surface $M$.
For every basic set $\Lambda$ and $\epsilon>0$ small enough,
the lamination $\mathfs W^s_\epsilon(\Lambda)$ has Lipschitz holonomies.
\end{thm}

\begin{proof}
We choose $\alpha>0$ small and define the following cones of $\RR^2$:
 $$\begin{aligned}
&\mathcal C^u:=\{(u,v), |u|\leq \alpha |v|\} \; \subset \; \widehat {\mathcal C}^u:=\{(u,v), \alpha|u|\leq |v|\},\\
&
\mathcal C^s:=\{(u,v), \alpha|u|\geq  |v|\} \; \subset \; \widehat {\mathcal C}^s:=\{(u,v), |u|\geq \alpha|v|\}.
 \end{aligned}$$

Replacing $f$ by an iterate, we may assume that
$\|Df|_{E^s}(x)\|,\|Df^{-1}|_{E^u}(x)\|<\alpha$ for any $x\in \Lambda$.
We choose a family of $C^r$ charts $\psi_x$ from a neighborhood of $0$ in $\RR^2$ to a neighborhood
of $x$ in $M$ such that $D\psi_x{1\choose 0}$ is a unit vector in $E^s(x)$ and $D\psi_x{0\choose 1}$ is a unit vector in $E^u(x)$.
By compactness, one can extract a finite collection of charts $\psi_1,\dots,\psi_\ell$ and a number $\rho>0$ with the following properties:
\begin{myitemize}
\item[--] Each chart $\psi_i$ is a diffeomorphism from a ball $B(0,2\rho)\subset \RR^2$ to an open set $U_i\subset M$.
\item[--] The union $\cup_i \psi_i(B(0,\rho))$ contains $\Lambda$.
\item[--] For $x \in \psi_i^{-1}(U_i\cap f^{-1}(U_j))$, the linear map $A:= D_{f\circ\psi_i(x)}\psi_j^{-1}.D_{\psi_i(x)}f.D_x\psi_i$
sends $\widehat {\mathcal C}^u$ into $\mathcal C^u$ and expands its vectors by a factor larger than $\alpha^{-1}$.
Symmetrically, $A^{-1}$ expands vectors in $\widehat {\mathcal C}^s$ by a factor larger than $\alpha^{-1}$.
\end{myitemize}
By choosing $\varepsilon>0$ small, each point $x$ in the support of the lamination $\mathfs W^s_\epsilon(\Lambda)$
belongs to some image $\psi_i(B(0,\rho))$ and the tangent space to $W^s(x)$ belongs to the image
$D_{\psi_i^{-1}(x)}\psi_i(\mathcal C^s)$.
\medskip

In order to prove the theorem,
we choose a transversal $\tau_0$ to $\mathfs W^s_\epsilon(\Lambda)$ inside a lamination neighborhood $U_0$
and a point $x_0\in \tau_0\cap \mathfs W^s_\epsilon(\Lambda)$.
From the inclination lemma, the iterates $f^n(\tau_0)$ converge to the unstable manifolds of $\Lambda$:
there exist an integer $n_0\geq 1$
and a chart $\psi_i$ such that $f^{n_0}(x_0)\in \psi_i(B(0,\rho))$
and the tangent space to $\psi_i^{-1}\circ f^{n_0}(\tau_0)$ at $\psi_i^{-1}(f^{n_0}(x_0))$
is contained in $\mathcal C^u$.
If $V$ is a small neighborhood of $x_0$ in $M$
and $\mathcal T$ is a small neighborhood of $\tau_0$ for the {uniform} $C^1$-topology (see section~\ref{Sec.Sard}), we have:
\begin{myitemize}
\item[--] for any $\tau\in \mathcal T$, $y\in \tau \cap V$,
the tangent space to $\psi_i^{-1}\circ f^{n_0}(\tau_0)$ at $\psi_i^{-1}(f^{n_0}(x_0))$
is contained in $\mathcal C^u$,
\item[--] for any $\tau,\tau'\in \mathcal T$, the holonomy $\Pi_{\tau\to\tau'} \colon \tau\cap V\cap
W^s_\epsilon(\Lambda)\to \tau'$
is well defined.
\end{myitemize}
We choose any small interval $I\subset \tau\cap V\cap W^s_\epsilon(\Lambda)$. We will show that its image $I':=\Pi_{\tau\to\tau'}(I)$ satisfies:
 \begin{equation}\label{eq-Lip-hol}
   |I'|\leq L |I| \text{ for some positive contant $L=L(f,V,\mathcal T)$, independent of $\tau'\in\mathcal T$}.
 \end{equation}
By symmetry,  $|I|\leq L|I'|$ and Theorem~\ref{Theorem-Holonomy-is-Lipschitz} will immediately follow.
\medskip

The endpoints of $I$ and $I'$ are connected by arcs $J_1,J_2$ contained in leaves of $\mathfs W^s_\epsilon(\Lambda)\cap U_0$.
Since $V$ has been chosen small, the lengths of $I,I',J_1,J_2$ are all small.
Consequently the union $I\cup I'\cup J_1\cup J_2$ is contained in the image $U_i$ of a chart $\psi_i$.
Since the forward iterates of $J_1,J_2$ remain inside leaves of $\mathfs W^s_\epsilon(\Lambda)$,
the forward iterates $f^k(I\cup I'\cup J_1\cup J_2)$ remain in the images of charts
$\psi_{i(k)}$, until
a time $k_{\max}$ such that $f^{k_{\max}}(I)$ or $f^{k_{\max}}(I')$ reach a length comparable to $\rho$ in the charts.
Since we assume $I$ to be small,
we can assume that $k_{\max}\geq n_0$.
The iterates $f^k(J_1)$ and $f^k(J_2)$ for $0\leq k\leq k_{\max}$ seen in the chart
$\psi_{i(k)}$ are tangent to the cone $\mathcal{C}^s$, hence their lengths decrease by a factor smaller than $\alpha$ at each iteration by $f$. The iterates $f^k(I)$ and $f^k(I')$ for $n_0\leq k\leq k_{\max}$
are tangent to $\mathcal C^u$ in the charts and their lengths increase by a factor larger than $\alpha^{-1}$ at each iteration by $f$.
Consequently, there exists a first time $N\leq k_{\max}$ such that (see the Figure~\ref{f.strip}) in the chart $\psi_{i(N)}$,

\begin{equation}\label{e.shrink}
\min\big(|f^N(I)|, |f^N(I')|\big)\geq
\alpha^{-1}\max\big(|(f^N(J_1)|, |f^N(J_2)|\big).
\end{equation}

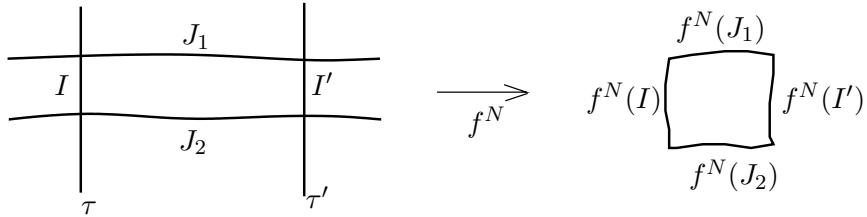
\begin{figure}[h]
\begin{center}
\includegraphics[scale=1.1]{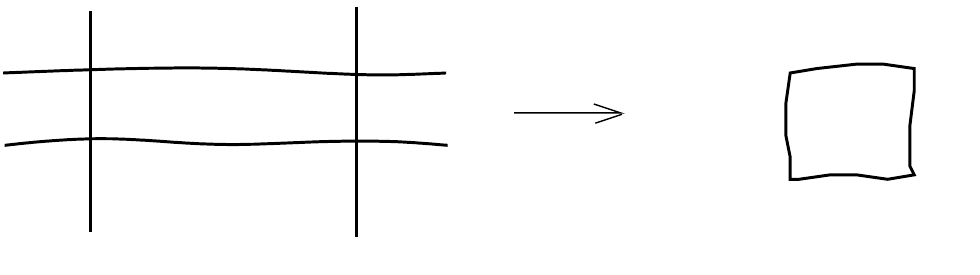}
\end{center}
\begin{picture}(0,0)
\put(105,73){$I$}
\put(200,30){$\tau'$}
\put(202,73){$I'$}
\put(153,92){$J_1$}
\put(152,52){$J_2$}
\put(115,28){$\tau$}
\put(345,39){$f^{N}(J_2)$}
\put(340,95){$f^{N}(J_1)$}
\put(380,68){$f^{N}(I')$}
\put(261,59){$f^{N}$}
\put(306,68){$f^{N}(I)$}
\end{picture}
\caption{Image of a stable strip $I\cup I'\cup J_1\cup J_2$ by the iterate $f^N$.
\label{f.strip}}
\end{figure}

\begin{lemma}\label{l.foliation}
There exists a foliation of a set containing the topological rectangle of $\mathbb R^2$
bounded by $\psi^{-1}_{i(N)}\circ f^N(I\cup I'\cup J_1\cup J_2)$,
whose leaves are tangent to $\mathcal C^s$.
Its holonomy defines a homeomorphism $\Pi$ between the transverse arcs $\psi^{-1}_{i(N)}\circ f^{N}(I)$ and $\psi^{-1}_{i(N)}\circ f^{N}(I')$,
which is Lipschitz with constant $Lip(\Pi)<1+10\alpha$.
\end{lemma}
\begin{proof} The two curves $\psi^{-1}_{i(N)}\circ f^N(J_i)$ are tangent to the horizontal cone $\mathcal C^s$,
hence are contained in graphs of two $\alpha$-Lipschitz functions $\varphi_1,\varphi_2$.
For $u\in [1,2]$, the functions
$$\varphi_u:=(2-u)\varphi_1+(u-1)\varphi_2$$
are $\alpha$-Lipschitz and their graphs define the leaves of the foliation.
We denote by $\Pi$ the holonomy map between $\widetilde I:=\psi^{-1}_{i(N)}\circ f^{N}(I)$ and $\widetilde I':=\psi^{-1}_{i(N)}\circ f^{N}(I')$.

Let $z\in \widetilde I$ and $z'=\Pi(z)$.
Let $V,V'$ be two vertical segments which connect the graphs of $\phi_1$ and $\phi_2$
and which contain $z$ and $z'$ respectively.
The holonomy $\Pi_{V,V'}$ between
$V$ and $V'$ is linear since the leaves $\varphi_t$ have been obtained as barycenters.
As a consequence, $\Pi_{V,V'}$ is a Lipschitz map whose Lipschitz constant equals $\frac{|V'|}{|V|}$.
Using that the curves are tangent to $\mathcal C^s$ or $\mathcal C^u$
and~\eqref{e.shrink}, we get
$$\Lip(\Pi_{V,V'})=\frac{|V'|}{|V|}\in [1-3\alpha, 1+3\alpha].$$

The holonomy map $\Pi$ decomposes into
$$\Pi=\Pi_{V',\widetilde I'}\circ \Pi_{V,V'}\circ \Pi_{\widetilde I,V}.$$
The holonomy map $\Pi_{\widetilde I,V}$ fixes $z$.
Note that the foliation is $C^1$, the slope of the leaves is smaller than $\alpha$
and the (absolute value of the) slope of $\widetilde I$ is larger than $\alpha^{-1}$.
A simple application of the implicit function theorem using  $|(\phi_2(x)-\phi_1(x))/(\phi_2(y)-\phi_1(y))|\leq 2$ implies that the map $\Pi_{\widetilde I,V}$ as well as its inverse is differentiable with Lipschitz constant bounded by $1+\alpha$.
One argues similarly for $\Pi_{V',\widetilde I'}$.
This gives the conclusion of the lemma.
\end{proof}

We thus obtain a bi-Lipschitz homeomorphism $\Pi_{I,I'}$ between
$I$ and $I'$ defined by
$$\Pi_{I,I'}=f^{-N}\circ D\psi_{i(N)}\circ\Pi\circ D\psi_{i(N)}^{-1}\circ f^{N}.$$
Its Lipschitz constant at any point $x\in I$ is bounded by
$$\|Df^{N}|_{I'}(\Pi_{I,I'}(x))\|^{-1}\; \cdot\; \|D\psi_{i(N)}\|\; \cdot\;\operatorname{Lip}(\Pi)\; \cdot\;\|D\psi_{i(N)}^{-1}\|\;\cdot\;\|Df^{N}|_{I}(x)\|.$$
Defining $x':=\Pi_{I,I'}(x)$, we thus have to bound the following quantity:
\begin{equation}
\frac{\|Df^N|_{I}(x)\|}{\|Df^N|_{I'}(x')\|}\;=\;
\prod_{i=0}^{N-1}
\frac{\|Df|_{f^{i}(I)}(f^i(x))\|}{\|Df|_{f^{i}(I')}(f^i(x'))\|}.
\end{equation}

The diffeomorphism $f$ induces a homeomorphism $F$ of the unit tangent bundle $T^1M$
defined by:
$$F(u)=\|Df.u\|^{-1}\;Df.u.$$
We endow $T^1M$ with a distance $d_1$ induced by an arbitrary Riemannian metric.

\begin{lemma}\label{l.slope}
There exist $C>0$ and $\lambda\in (0,1)$ with the following property:
consider a sequence of charts $U_{i(0)},\dots,U_{i(n)}$, two points $x,x'\in M$
and two unit vectors $u\in T_xM$, $v\in T_{x'}M$ such that:
\begin{myitemize}
\item[--] for each $0\leq k\leq n$, both points $f^k(x),f^k(x')$ belong to $U_{i(k)}$
and their images by $\psi_{i(k)}^{-1}$ belong to a curve tangent to $\mathcal C^s$,
\item[--] the preimages by $\psi_{i(0)}^{-1}$ of $u$ and $v$ are tangent to $\mathcal C^u$.
\end{myitemize}
Then $d_1(F^n(u),F^n(v))\leq C\;{\lambda}^n$.
\end{lemma}
\begin{proof}
Using the charts, we denote by $x_k,x'_k$ the images of $f^k(x),f^k(x')$ by the map $\psi^{-1}_{i(k)}$
and $u_k=(u_k^1,u_k^2)$, $v_k=(v_k^1,v_k^2)$ the
vectors $D(\psi_{i(k)}^{-1}\circ f^k).u$ and $D(\psi_{i(k)}^{-1}\circ f^k).v$.
The diffeomorphism $f$ lifts as maps $f_k:=\psi^{-1}_{i(k+1)}\circ f \circ \psi_{i(k)}$.
Since the number of charts is finite, it is enough to estimate the distance between
$x_n,x_n'$ and the angle between the vectors $u_n$ and
$v_n$ in the charts.

The contraction of the vectors in the cone $\mathcal C^s$, the smallness of $\alpha$,
and the fact that $x_n,x'_n$
belong to a curve tangent to $\mathcal C^s$ contained in $B(0,2\rho)$, imply that for any $0\leq k\leq n$,
 $$
     d(x_k,x'_k)\leq 5\rho 2^{-k}.$$

It is thus enough to find $c>0$ and $\lambda\in (0,1)$ and show:
\begin{equation}\label{e.induction}
\left|\frac{u_k^1}{u_k^2}-\frac{v_k^1}{v_k^2}\right|\leq c {\lambda}^{k}.\end{equation}
This is proved inductively. The case $k=0$ holds if $c=2\alpha$ since each quotient is in $[-\alpha,\alpha]$
as $u_0$ and $v_0$ are tangent to $\mathcal C^u$.
Let us denote $w_k:=(w_k^1,w_k^2)$ the image of $v_k=(v_{k-1}^1,v_{k-1}^2)$
by $Df_{k-1}(x_{k-1})$ (rather than by $Df_{k-1}(x'_{k-1})$).
The contraction and expansion in the cones by $Df_{k-1}(x_{k-1})$ gives
$$\left|\frac{u_k^1}{u_k^2}-\frac{w_k^1}{w_k^2}\right|\leq \frac 1 2 \left|\frac{u_{k-1}^1}{u_{k-1}^2}-\frac{v_{k-1}^1}{v_{k-1}^2}\right|.$$

Let $\beta:=r-1$ when $r\in (1,2)$ and $\beta:=\frac{1}{2}$ when $r\geq 2$, then $Df_k$ are uniformly $\beta$-H\"older continuous.
This allows to compare the image of $v_{k-1}$
under $Df_{k-1}(x_{k-1})$ and $Df_{k-1}(x'_{k-1})$: there exists a constant $c'>0$ such that:
 $$
   \left|\frac{v_k^1}{v_k^2}-\frac{w_k^1}{w_k^2}\right|\leq c' d(x_k,x'_k)^\beta\leq 5\rho^{\beta}c' 2^{-\beta k}.
 $$
Assuming that~\eqref{e.induction} holds for $k-1$, one thus gets
 $$
    \left|\frac{v_k^1}{v_k^2}-\frac{u_k^1}{u_k^2}\right|\leq
\frac c 2 \lambda^{k-1} + 5\rho^{\beta}c' 2^{-\beta k}.
 $$
This gives~\eqref{e.induction} for $k$ provided one has chosen
$ 2^{-\beta}<\lambda<1$ and $c\geq5c'\rho^{\beta}/(1-(2\lambda)^{-1})$.
\end{proof}

The points $f^N(x)$, $f^N(x')$ belong to the image by $\psi_{i(k)}$ of a curve tangent to $\mathcal C^s$,
hence the same holds for any iterate $f^k(x),f^k(x')$, $0\leq k\leq N$ by backward invariance of the cone $\mathcal C^s$.
The lemma~\ref{l.slope} can thus be applied to the tangent spaces of
$f^{n_0}(I)$ and $f^{n_0}(I')$ at $f^{n_0}(x)$ and $f^{n_0}(x')$ until the time $N-n_0$.
Since $Df$ is $(r-1)$-H\"older continuous, there exists a constant $C'>0$
such that:
$$\log\frac{\|Df_{|f^{i}(I')}(f^i(x'))\|}{\|Df_{|f^{i}(I)}(f^i(x))\|}\leq
C'\lambda^{i(r-1)} \qquad \text{for }n_0\leq i\leq N.$$
Hence there is a constant $C''>0$ which only depends on
$n_0,C'$ and the norms of $Df,\; Df^{-1}$ such that
\begin{equation}
\frac{\|Df^N|_{I}(x)\|}{\|Df^N|_{I'}(x')\|}\;\leq C'' \exp \sum_{i\geq0} \lambda^{i(r-1)}.
\end{equation}

Combining with Lemma~\ref{l.foliation},  the Lipschitz constant $L$ of $\Pi_{I,I'}$ only depends on $C'',\lambda, r$ and on
the norms of $D\psi_i$ and $D\psi_i^{-1}$. Eq.~\eqref{eq-Lip-hol} and therefore the theorem is proved.
\end{proof}

\small
\bibliographystyle{plain}
\bibliography{finiteMME}{}
\bigskip

\hspace{-2.5cm}
\begin{tabular}{l l l l l}
\emph{J\'er\^ome Buzzi}
& &
\emph{Sylvain Crovisier}
& &
\emph{Omri Sarig}
\\

Laboratoire de Math\'ematiques d'Orsay
&& Laboratoire de Math\'ematiques d'Orsay
&& Faculty of Mathematics\\
CNRS - UMR 8628
&& CNRS - UMR 8628
&&  and Computer Science\\
Universit\'e Paris-Sud 11
&&  Universit\'e Paris-Sud 11
&& The Weizmann Institute of Science\\
Orsay 91405, France
&& Orsay 91405, France
&& Rehovot, 7610001,  Israel
\end{tabular}

\end{document}